\documentclass[
  journal=bjps,
]{cup-journal}

\usepackage{booktabs,microtype,siunitx,tabularx,relsize,appendix}
\usepackage[T1]{fontenc}
\usepackage{lmodern}
\usepackage[style=authoryear,backend=biber]{biblatex}

\bibliography{pick.bib}

\title{Estimating Max-Stable Random Vectors with Discrete Spectral Measure using Model-Based Clustering}

\author{Alexis Boulin}
\affiliation{Ruhr-Universität Bochum, Fakultät für Mathematik. Email: alexis.boulin@ruhr-uni-bochum.de}

\keywords{Extremes, High dimensional estimation, Latent model, Soft clustering, Variable clustering.}

\begin{document}

\begin{abstract}
    This study introduces a novel estimation method for the entries and structure of a matrix $A$ in the linear factor model $\mathbf{X} = A\textbf{Z} + \textbf{E}$. This is applied to an observable vector $\mathbf{X} \in \mathbb{R}^d$ with $\textbf{Z} \in \mathbb{R}^K$, a vector composed of independently regularly varying random variables, and lighter tail noise $\textbf{E} \in \mathbb{R}^d$. The spectral measure of the regularly varying random vector $\mathbf{X}$ is subsequently discrete and completely characterised by the matrix $A$. It follows that the behaviour of its maxima can be modelled by a max-stable random vector with discrete spectral measure. Every max-stable random vector with discrete spectral measure can be written as a linear factor model. Each row of the matrix $A$ is supposed to be both scaled and sparse. Additionally, the value of $K$ is not known a priori. The problem of identifying the matrix $A$ from its matrix of pairwise extremal correlation is addressed. In the presence of pure variables, which are elements of $\mathbf{X}$ linked, through $A$, to a single latent factor, the matrix $A$ can be reconstructed from the extremal correlation matrix. Our proofs of identifiability are constructive and pave the way for our innovative estimation for determining the number of factors $K$ and the matrix $A$ from $n$ weakly dependent observations on $\mathbf{X}$. We apply the suggested method to weekly maxima rainfall and wildfires to illustrate its applicability.
\end{abstract}

\section{Introduction}

In this current study, our aim is to estimate the $d \times K$ loading matrix A, which might exhibit sparsity, and serves as the parameter for the decomposition of an observable random vector $\mathbf{X}$. This can be expressed as

\begin{equation}
    \label{eq:lfm}
    \mathbf{X} = A \textbf{Z} + \textbf{E}.
\end{equation}
In this equation, $\textbf{Z}$ represents an unobservable, $K$-dimensional random vector, serving as an underlying latent factor, $E \in \mathbb{R}^d$ as a unobservable random noise. The precise count of factors, $K$, remains undisclosed and both $d$ and $K$ are permitted to increase and be larger than $n$, the number of observations. To establish the foundation of our framework inside extreme value theory, we assume that $\textbf{Z}$ comprises of asymptotic independent random variables characterised by a tail index $\alpha$, for the purposes of our study, we will set this tail index to a fixed value of $\alpha$ equal to unity. As per the construction, the vector $\textbf{Z}$ is regularly varying with the subsequent exponent measure
\begin{equation*}
    \Lambda_\textbf{Z} = \sum_{k=1}^K \delta_0 \otimes \dots \otimes \Lambda_{Z^{(k)}} \otimes \dots \otimes \delta_0, \quad \Lambda_{Z^{(k)}}(dy) = y^{-2} dy.
\end{equation*}
    The random noise vector $E \in \mathbb{R}^d$ is postulated to possess a distribution with a tail that is lighter than that of the associated factors. Hence, $\mathbf{X}$ is also regularly varying which can be equivalently described by the existence of an angular measure $\Phi$ where the following weak convergence holds true on the positive unit sphere for an arbitrary norm $||\cdot||$ on $\mathbb{R}^d$,
\begin{equation*}
    \underset{x \rightarrow \infty}{\lim} \mathbb{P}\left\{ \frac{\mathbf{X}}{\lVert \mathbf{X}\rVert} \in \cdot \, |\, \lVert \mathbf{X} \rVert > x \right\} = \Phi(\cdot),
\end{equation*}
where $\Phi$ has the discrete representation
\begin{equation}
    \label{eq:X_ang_meas}
    \Phi(\cdot) = \sum_{k=1}^K \lVert A_{\cdot k} \rVert \delta_{\frac{A_{\cdot k}}{\lVert A_{\cdot k} \rVert}}(\cdot),
\end{equation}
with $\delta_x(\cdot)$ the Dirac measure that puts unit mass at $x$ and $A_{\cdot k}$ is the $k$th column of the matrix $A$. Taking $\mathbf{X}_1, \dots, \mathbf{X}_m$, $m$ i.i.d. replications of $\mathbf{X}$ in \eqref{eq:lfm}, it follows that, see for example \cite[Theorem 2.1.6]{kulik2020heavy} for details,
\begin{equation*}
	\underset{m \rightarrow \infty}{\lim} \, \mathbb{P}\left\{ \bigvee_{i=1}^m c_m^{-1} X_i^{(j)} \leq x^{(j)}, j=1,\dots,d \right\} = e^{-\sum_{a=1}^K \bigvee_{j=1}^d \frac{A_{ja}}{x^{(j)}}}, \quad \mathbf{X} \geq 0,
\end{equation*}
where $c_m$ is a scaling sequence. The limiting distribution on the right is max-stable, which means there exist $\mathbf{a}_m > 0$ and $\mathbf{b}_m$ such that $H^m(\mathbf{a}_m \mathbf{x} + \mathbf{b}_m) = H(\mathbf{x})$ for any $m \in \mathbb{N}$ with $H$ a cumulative distribution function. Furthermore, under our model in \eqref{eq:lfm} it has the property of having a discrete angular measure. Throughout this paper, we will refer to a random vector with a max-stable distribution as a max-stable random vector.

The expression \eqref{eq:lfm} can also be considered as a linear adaptation of the max-linear models, sharing the same angular measure $\Phi$. This essentially follows from the fact that the ratio of the probabilities of the sum and the maximum of the $A_{ij}Z^{(j)}$ exceeding $x$ tends to $1$ as $x \rightarrow \infty$ (see \cite[page 38]{embrechts2013modelling} or \cite[Example 2.2.8, Example 2.2.9]{kulik2020heavy}). Each max-stable distribution with a discrete spectral measure is inherently max-linear, see \cite[Section 3.1]{fougeres2013dense}. The max-linear model, in turn, is dense in the class of $d$-dimensional multivariate extreme value distribution (\cite{fougeres2013dense}). Consequently, any multivariate max-stable vector can be finely approximated through a max-linear model provided that $K$ is large. Additionally, \cite{cooley2019decompositions} established the existence of a finite natural number, denoted as $q \in \mathbb{N}$, such that the tail pairwise dependence matrix $\Sigma_\mathbf{X}$ for any multivariate regularly varying random vector $\mathbf{X}$ with a tail index $\alpha = 2$ is equivalent to that of a max-linear model with $q$ factors. This equivalence is expressed through the relationship $\Sigma_\mathbf{X} = A A^\top$.

Corresponding max-linear models, have also been explored in the field of time series for extremes \cite{davis1989basic, hall2002moving}. More recently, they have found applications in the domain of structural equation models \cite{gissibl2018max, kluppelberg2019bayesian}, as well as in the context of clustering extremes \cite{janseen2020clustering, medina2021spectral, avella2022kernel}. Factor models of this kind find widespread use across diverse applications. For instance, they are often employed to represent underlying factors that influence financial returns (\cite{cui2018max}) as well in environmental sciences (\cite{kiriliouk2022estimating}).

\paragraph{Outline of the literature.} Estimating parameters in linear factor models poses a difficult task, primarily because there is no spectral density that rules out standard maximum likelihood procedures. Instead, \cite{einmahl2012mestimator} and \cite{einmahl2018continuous} opt for a least square estimator based on the stable tail dependence function to tackle this task. \cite{janseen2020clustering, medina2021spectral} propose spectral clustering algorithms designed for extremes employing its output to estimate $A_{\cdot 1} / ||A_{\cdot 1}||, \dots, A_{\cdot K} / || A_{\cdot K}||$ and $||A_{\cdot 1}||/w,\dots,||A_{\cdot K}||/w$. These parameters characterise the angular measure of the linear factor model. However, this approach falls short in estimating the matrix $A$ which can be crucial for practical interpretation and computing failure sets (see Section \ref{sec:extreme_precip}). Additionally, these methods face limitations in higher dimensions, grappling running time difficulties or curse of dimension. Moreover these methods also assume that the number $K$ is known a priori, a requirement that is often scarcely fulfilled  in practical scenario. Addressing this hurdle, additional methods, as proposed by \cite{medina2021spectral, avella2022kernel}, introduce a procedure coupled with the so-called screeplot to aid in the selection of the elusive number $K$. Despite the practical utility of such an approach, the theoretical underpinnings supporting these findings are still in their early stage of development. To our current understanding, methods for estimating $A$ in higher dimensions have emerged specifically under the condition of a squared matrix $A \in \mathbb{R}^{d \times d}$. Notably, these methods have found fruitful application in contexts characterised by moderate dimensions. For instance, in Diricted Acyclic Graph, \cite{kluppelberg2021estimating} have made noteworthy contributions, while \cite{kiriliouk2022estimating} have demonstrated successful applications of their estimator in environmental and financial dataset. However, it is crucial to acknowledge that these achievements are contingent upon the specific conditions and dimensions involved. A noteworthy recent paper worth emphasising is \cite{zhang2023wasserstein}. The paper investigates minimax risk bounds for estimators of the spectral measure in multivariate linear factor models, particularly when the number of latent factors $K$ exceeds the dimension $d$ and the latter is fixed. Foremost, a critical lens on the theoretical foundations reveals a reliance on a i.i.d. sample and the asymptotic framework in the mentioned literature. The assumption of serial independence may face scrutiny when these methods are extended to environmental datasets, where deviations from serial independence are legitimately suspected. Moreover, the asymptotic framework, with a fixed arbitrary dimension $d$ while the sample size $n \rightarrow \infty$, may offer limited insights into the performance of estimation processes in high-dimensional setting, i.e., $d$ vary with $n$ and might even surpass the sample size.

\paragraph{Our contribution.} Drawing inspiration of \cite{bing2020adaptative}, we propose a model-based clustering via $A$ with the crucial distinction that the covariance matrix of $\mathbf{X}$ does not exist in our model. Within the framework of model \eqref{eq:lfm}, we consider two components, namely $X^{(i)}$ and $X^{(j)}$ belonging to the vector $\mathbf{X}$, as akin if they share a non-zero association. This association is established through the intermediary of the matrix $A$, connecting them to a common latent factor $Z^{(a)}$. Variables exhibiting this similarity are grouped together within the cluster denoted as $G_a$:
\begin{equation}
    \label{eq:pure_clusters}
    G_a = \{ j \in \{1,\dots,d\} \, : A_{ja} \neq 0 \}, \quad \textrm{for each } a \in \{1,\dots,K\}.
\end{equation}
Given that $X^{(j)}$ can potentially be linked to multiple latent factors, the resulting clusters are characterised by overlap. In terms of terminology, groups that may become large without the others are called extreme directions. More precisely, if $J \subset \{1,\dots,d\}$ is such that the components $(X^{(j)})_{j \in J}$ can be large simultaneously while the other components $(X^{(j)})_{j \in [d] \setminus J}$ are small, then $J$ defines an extreme direction (see \cite{simpson2020determining} for a precise definition). By examining \cite[Example 3.7]{mourahib2023multivariate}, one can observe that soft clusters in Equation \eqref{eq:pure_clusters} also represent the extreme directions.

In this endeavor, our focus centers on presenting a model-based clustering approach through the utilisation of A. Specifically, we contemplate a variant of model \eqref{eq:lfm} where in every rows of $A$ undergoes scaling. To be specific, we posit the following assumption:
\begin{Assumption}{(i)}
	\label{cond:(i)}
    $\sum_{a=1}^K A_{ja} = 1.$
\end{Assumption}
The weights $A_{i1},\dots,A_{iK}$ indicate the degree to which component align with each cluster. This condition dives our model into both hard and soft clustering. Additionally, to employ the model effectively for clustering purposes, it is important to circumvent the trivial scenario where each component $X^{(j)}$ is associated with all latent factors. To adress this concern, we permit the row of $A_{j\cdot} = (A_{j1},\dots,A_{jK})$ to exhibit sparsity for $j$ in the range of $\{1,\dots,d\}$. However, it is noteworthy that this property is not required for establishing the identifiability of $A$.

Condition \ref{cond:(i)}, if not explicitly specified, fails to guarantee the identifiability of $A$ in model \eqref{eq:lfm} solely based on bivariate measures on the vector $\mathbf{X}$. We term the following condition, denoted as \ref{cond:(ii)}, the ``pure variable assumption''. In simple terms, this assumption posits the presence of at least one pure variables $X^{(j)}$, among the components of $\mathbf{X}$. These pure variables are uniquely associated with a single latent factor and no other.
\begin{Assumption}{(ii)}
	\label{cond:(ii)}
    For every $a \in \{1,\dots,K\}$, there exists at least one indice $j \in \{1,\dots,d\}$ such that $A_{ja} = 1$ and $A_{jb} = 0, \, \forall b \neq a$.
\end{Assumption}

Cluster denoted as $G_a$, established in accordance with Equation \eqref{eq:pure_clusters}, derive their definition from the unobservable factor $Z^{(a)}$. In this context, a pure variable $X^{(j)}$ serves as an observable representation of $Z^{(a)}$, contributing to the elucidation of the ambiguous nature of cluster $G_a$. Moreover, a more stringent version of Condition \ref{cond:(ii)} has a rich history, specifically
\begin{Assumption}{(ii')}
	\label{cond:(ii')}
    For every $a \in \{1,\dots,K\}$, there exist at least one known indice $j \in \{1,\dots,d\}$ such that $A_{ja} = 1$ and $A_{jb} = 0, \, \forall b \neq a$.
\end{Assumption} 
This condition stands out as one of the few interpretable parametrisations of $A$, effectively eliminating the ambiguity associated with latent factors. In psychology, the variables generated purely through the parametrisation as referred to as factorial simple items (\cite{mcdonald2013test}). A comparable condition find its roots in the topic modeling literature, where the identifiability of topics is assured under the assumption that anchor words exists, i.e., words that exclusively appear in one topic. In hydrology, the concept of pure variables was utilised to pinpoint catchments stations as representatives of pollution sources in \cite[page 700]{tolosana2005latent}.
In section \ref{sec:identifiability}, we demonstrate that, under Conditions \ref{cond:(i)} and \ref{cond:(ii)}, the matrix $A$ can be recovered solely through the use of bivariate measures, namely extremal correlations. In Section \ref{sec:estimation_chap_4}, we develop \ref{alg:clust}, a new soft clustering algorithm specifically for linear factor model that estimate the loading matrix $A$ and the overlapping groups. We provide a sparse estimator $\hat{A}$ of $A$ that is tailored to our model specifications. Our approach follows the constructive techniques used in our identifiability proofs. We first construct an estimator $\hat{I}$, an estimator of the pure variable set $I$, and $\hat{K}$, an estimator of the number of clusters $K$. These are used to estimate the rows in A corresponding to pure variables. The remaining rows of A are estimated via an easily implementable program that is tailored to this problem. We base our theoretical study on mixing conditions over the studied process. These conditions make explicit the independence between ``past'' and ``future''; meaning that the ``past'' is progressively forgotten. Mixing conditions are consequently more adapted to work in areas like finance or climate sciences where history is of considerable importance. To make this more precise, we consider processes with exponentially decaying strong mixing coefficients, as introduced in Section \ref{sec:stat_guarantees}. The algorithm \ref{alg:clust} recovers the number of latent variables with high probability under a strong signal condition in Section \ref{subsec:stat_guarantees_k_i}. We establish an upper bound on the $L_2$ norm ($L_2(\hat{A},A)$, as defined in Section \ref{subsec:stat_guarantees_A}) for the matrix $A$ specified by model \eqref{eq:lfm} and subject to Conditions \ref{cond:(i)}-\ref{cond:(ii)}, as discussed in Section \ref{subsec:stat_guarantees_A}. A control of cluster or, equivalently stated, extreme direction recovery is also given in Section \ref{subsec:stat_guarantees_A}.

\paragraph{Notations} All bold letters $\mathbf{X}$ correspond to vector in $\mathbb{R}^d$. The notation $\delta_x$ corresponds to the Dirac measure at $x$. Throughout this paper, we are concerned with a simple undirected graph $G = (V,E)$ with a finite set $V$ of vertices and a finite set $E$ of unordered pairs $(v,w)$ of distinct vertices, called edges, we denote by $\bar{E}$ its complementary adjacent matrix. A pair of vertices $v$ and $w$ are said to be adjacent if $(v,w) \in E$. For the subset $W \subseteq V$ of vertices, $G(W) = (W, E(W))$ with $E(W) = \{(v,w) \in W \times W \, | (v,w) \in E \}$ is called a subgraph of $G = (V,E)$ induces by $W$. Given the subset $Q \subseteq V$ of vertices, the induced subgraph $G(Q)$ is said to be complete if $(v,w) \in E$ for all $v,w \in Q$ with $v \neq w$. In this case, we may simply state that $Q$ is complete subgraph. A complete subgraph is also called a clique. A clique is maximum if its cardinality is the largest among all the cliques of the graph.

\section{Identifiability}
\label{sec:identifiability}

Within this section, we present a demonstration that the allocation matrix $A$, as defined by model \eqref{eq:lfm} and subject to conditions \ref{cond:(i)}-\ref{cond:(ii)}, is identifiable, within the exception of multiplication by a permutation matrix.

As per the construction, the vector $\textbf{Z}$ is regularly varying, it possesses an extremal correlation matrix represented by $I_K$, the identity matrix. Consequently, we deduce that the vector $\mathbf{X}$ also follows a pattern of regular variation, leading to the presence of an extremal correlation matrix denoted as $\mathcal{X} = [ \chi(i,j) ]_{i=1,\dots,d;j=1,\dots,d}$, where
\begin{equation*}
    \chi(i,j) = \underset{x \rightarrow \infty}{\lim} \frac{\mathbb{P}\{ X^{(i)} > x, X^{(j)}>x \}}{\mathbb{P}\{X^{(j)} >x \}}.
\end{equation*}
The subsequent theorem is poised to demonstrate that the extremal correlation matrix can be elegantly formulated using exclusively the loading matrix $A$. However, before going further, we introduce a novel matrix operation defined over matrices $A \in \mathcal{M}_{p,K}(\mathbb{R})$ and $B \in \mathcal{M}_{K,q}(\mathbb{R})$. Here, the notation $\mathcal{M}_{p,q}(\mathbb{R})$ refers to the collection of matrices encompassing $p$ rows and $q$ columns, with coefficients in the real number domain.
\begin{definition}
    We call $\odot$ the application:
    \begin{align*}
    \odot\colon&\mathcal{M}_{p,K}(\mathbb{R}) \times \mathcal{M}_{K,q}(\mathbb{R})\longrightarrow\mathcal{M}_{p,q}(\mathbb{R})\\
    &\phantom{++++++}(a_{ik}, b_{mj})\mapsto c_{ij},
     \end{align*}
     where
     \begin{equation*}
         c_{ij} = a_{i1} \wedge b_{1j} + \dots + c_{iK} \wedge b_{Kj}.
     \end{equation*}
\end{definition}
With all the essential tools at our disposal, we are ready to present the ensuing fundamental theorem.
\begin{theorem}
    \label{thm:matrix_product}
    Let $\mathbf{X}$ defined in \eqref{eq:lfm} and $A$ satisfies Condition \ref{cond:(i)}. Then $\mathbf{X}$ is regularly varying and its extremal correlation $\mathcal{X}$ can be written as
    \begin{equation*}
        \mathcal{X} = A \odot A^\top,
    \end{equation*}
    with
    \begin{equation*}
        \chi(i,j) = \sum_{k=1}^K A_{ik} \wedge A_{jk}.
    \end{equation*}
\end{theorem}

For any loading matrix $A$ that adheres to model \eqref{eq:lfm}, we can subdivide the set $[d] = \{1,\dots,d\}$ into two distinct non-overlapping segments: $I$ and its complement, designated as $J$. Within each row $A_{i\cdot}$ of $A_I$, there exists precisely at least one value $a \in [K]$ for which $A_{ia} = 1$. We assign the term ``pure variable set'' to $I$, while $J$ corresponds to the ``non-pure variable set''. To be more specific, for any given matrix $A$, the pure variable set is outlined as follows
\begin{equation}
    \label{eq:pure_set}
    I(A) = \cup_{a=1}^K I_a, \quad I_a := \{ i \in [d] : A_{ia} = 1, A_{ib} = 0 \; \forall b \neq a \}.
\end{equation}
In Equation \eqref{eq:pure_set}, we use the notation $I(A)$ to underscore that the pure variables set finds its definition in relation to matrix $A$. Moving forward, we will omit this explicit statement whenever there is no ambiguity. Additionally, it is worth mentioning that the sets $\mathcal{I} := \{I_a\}_{1 \leq a \leq K}$ constitute a partition of the pure variable set $I$.

To establish the identifiability of matrix $A$, our task is simplified by focusing on distinct identifiability of $A_I$ and $A_J$, each with allowance for a transformation by a permutation matrix. With respect to the definition of $A_I$, its identifiability is assured as long as the partition of the pure variable set $I$ remains identifiable. The heart of the challenge lies in the identifiability of set $I$ and the inherent issue of distinguishing between $I$ and $J$, based solely on the distribution of the vector $\mathbf{X}$. This stands as the central hurdle of the problem. Theorem \ref{thm:identifiability_I} holds a central position in our discussion. In the first part \ref{item_i:identifiability_I}, it offers both a necessary and sufficient description of the set $[K]$ by examining the extremal correlation matrix $\mathcal{X}$. In the second part \ref{item_ii:identifiability_I}, it provides a necessary and sufficient characterisation of the set $I$ when the cardinality of $I_a$ is greater than one. Finally, in the third part \ref{item_iii:identifiability_I}, it illustrates that both the set $I$ and its partition into subsets $\mathcal{I} = \{I_a\}_{1 \leq a \leq K}$ can be successfully identified. Let
\begin{equation}
    \label{eq:M_i}
    M_i = \underset{j \in [d] \setminus \{i\}}{\max}\chi(i,j)
\end{equation}
denote the greatest value among the entries of row $i$ of matrix $\mathcal{X}$ excluding $\chi(i,i) = 1$. Additionally, let $S_i$ represent the index set for which $M_i$ reaches its maximum
\begin{equation}
    \label{eq:S_i}
    S_i = \{ j \in [d] \setminus \{i \}, \; \chi(i,j) = M_i \}.
\end{equation}

\begin{theorem}
    \label{thm:identifiability_I}
    Assume that model \eqref{eq:lfm} and conditions \ref{cond:(i)}-\ref{cond:(ii)} hold. Then:
    \begin{enumerate}[label=\textcolor{frenchblue}{\bf(\alph*)}]
        \item \label{item_i:identifiability_I} The set $[K]$ is any maximum clique of the undirected graph $G = (V,E)$ where $V = [d]$ and $(i,j) \in E$ if $\chi(i,j) = 0$. 
        \item \label{item_ii:identifiability_I} Let $i \in I_a$, $a \in [K]$ and $|I_a| \geq 2$, then 
            \begin{equation*}
            j \in I_a \iff \chi(i,j) = 1 \textrm{ for any } j \in S_i. 
            \end{equation*}
        \item \label{item_iii:identifiability_I} The pure variable set $I$ can be determined uniquely from $\mathcal{X}$. Moreover its partition $\mathcal{I} = \{ I_a \}_{1 \leq a \leq K }$ is unique and can be determined from $\mathcal{X}$ up to label permutations. 
    \end{enumerate}
\end{theorem}

\begin{remark}
    \label{rmk:gissibl} It is crucial to emphasise that the identification of recursive max-linear models on a Directed Acyclic Graph involves examining the extremal correlation matrix and the initial nodes $V_0$, specifically, the nodes without parent connections (see Section 4.3 of \cite{gissibl2018tail}). Furthermore the set $V_0$ can be determined from the tail dependence matrix since it is a maximum clique of the undirected graph $G=(V,E)$ where $V= [d]$, and $(i,j) \in E$ if $\chi(i,j) = 0$ (refer to \cite[Theorem 2.7]{gissibl2018tail}), akin to the characterisation provided in Theorem \ref{thm:identifiability_I} \ref{item_i:identifiability_I} to identify the set $[K]$.
\end{remark}

The decision problem of the maximum clique problem is one of the first 21 NP-complete problems, as introduced by Karp in his influential paper on computational complexity (\cite{Karp1972}). This problem is known for its exponential complexity as the number of vertices increases in the worst cases. However, in the real world, many graphs tend to be sparse, meaning they have low degrees of connectivity (as noted in works by \cite{buchanan2014solving, eppstein2010listing}). This sparsity property allows for more efficient algorithms to solve the maximum clique problem in sparse graphs compared to general graphs.

In our framework, we have made the implicit practical assumption that the rows of the matrix $A$ are sparse. Consequently, the complement of the adjacency matrix $E$, denoted as $\bar{E}$, in the graph defined in Theorem \ref{thm:identifiability_I} \ref{item_i:identifiability_I}, is also sparse. In this context, a faster method to find a maximum clique is presented through the following binary problem:
\begin{equation*}
\begin{aligned}
\max_{x^{(i)}} \quad & \sum_{i=1}^d x^{(i)}\\
\textrm{s.t.} \quad & x^{(i)} + x^{(j)} \leq 1, \quad \forall (i,j) \in \bar{E}\\
  &x^{(i)} \in \{0,1\}, \quad i = 1,\dots,d.    \\
\end{aligned}
\end{equation*}
In this edge-based formulation, any valid solution defines a clique $C$ in the graph $G$ as follows: a vertex $i$ belongs to the clique if $x_i =1$, and otherwise $x_i = 0$. In our numerical studies, the use of this problem accelerated the estimation process, reducing computation time from minutes to seconds for large dimensions. In situations where the sparsity of matrix $A$ is less emphasised, attention shifts to the adjacency matrix $E$ which now becomes sparse. In such instances, well-known algorithms efficiently operates on the graph $E$. An excellent demonstration of this efficiency is found in the classical algorithm authored by \cite{bron1973algorithm}. For a more in-depth understanding of these intricacies, a comprehensive exploration awaits in \ref{sec:computation_time}.

The identifiability of the allocation matrix $A$ and that of the collection of clusters $\mathcal{G}= \{G_1,\dots,G_K\}$ in \eqref{eq:pure_set} use the results of Theorem \ref{thm:identifiability_I} in crucial ways. We state the result in Theorem \ref{thm:identifiability_J} below.

\begin{theorem}
    \label{thm:identifiability_J}
    Assume that model \eqref{eq:lfm} and conditions \ref{cond:(i)}-\ref{cond:(ii)} hold, $A$ can be uniquely recovered from $\mathcal{X} = A \odot A^\top$, up to column permutations. This implies that the associated soft clusters $G_a$, for $1 \leq a \leq K$, are identifiable, up to label switching.
\end{theorem}

\begin{remark}
    \label{remark:one}
    If Condition \ref{cond:(i)} is replaced with
    \begin{equation*}
        \sum_{a = 1}^K A_{ja} \neq 1,
    \end{equation*}
    then the loading matrix is no longer identifiable from $\mathcal{X}$. Indeed, consider $\mathbf{X} = A \textbf{Z}$ and $\tilde{\mathbf{X}} = \tilde{A}\textbf{Z}$ with $\tilde{A} = \lambda A$ for some $\lambda \neq 1$ and $A$ verifies Condition \ref{cond:(i)}. By Theorem \ref{thm:matrix_product}, we have
    \begin{equation*}
        \mathcal{X} = A \odot A^\top,
    \end{equation*}
    and using the same tools involved in the proof of Theorem \ref{thm:matrix_product}, we have
    \begin{equation*}
        \tilde{\chi}(i,j) = \underset{x \rightarrow \infty}{\lim}\frac{\mathbb{P}\{ \tilde{X}_j > x, \tilde{X}_j > x \}}{\mathbb{P}\{ \tilde{X}_i > x\}} = \frac{\sum_{k=1}^K (\lambda A_{ik}) \wedge (\lambda A_{jk})}{\sum_{k=1}^K (\lambda A_{ik})} = A_{ik} \wedge A_{jk}.
    \end{equation*}
    Thus $\mathcal{X} = \tilde{\mathcal{X}}$, and we cannot recover A from the extremal correlation matrix.
\end{remark}
\begin{remark}
    \label{remark:two}
    We show that the pure variable assumption stated in Condition \ref{cond:(ii)} is needed for the identifiability of $A$ with the extremal correlation, up to a permutation. Consider the specific example where $d=3$ and $K=2$
    \begin{equation*}
        A^{(1)} = \begin{pmatrix}
            0.7 & 0.3 \\
            0.3 & 0.7 \\
            0.5 & 0.5
        \end{pmatrix}.
    \end{equation*}
    Then with some computations using Theorem \ref{thm:matrix_product}, we obtain that the extremal correlation of $\mathbf{X}$ is equal to
    \begin{equation*}
    	\mathcal{X} = \begin{pmatrix}
            1 & 0.6 & 0.8 \\
            0.6 & 1 & 0.8 \\
            0.8 & 0.8 & 1
        \end{pmatrix}.
    \end{equation*}
    Now taking 
     \begin{equation*}
         A^{(2)} = \begin{pmatrix}
            0.8 & 0.2 \\
            0.4 & 0.6 \\
            0.6 & 0.4
        \end{pmatrix},
     \end{equation*}
     leads the same extremal correlation matrix. Thus, if $A$ does not satisfy \ref{cond:(ii)}, $A$ is generally not identifiable with the extremal correlation matrix.
\end{remark}

\section{Estimation}
\label{sec:estimation_chap_4}

Suppose that $(\mathbf{X}_t, t\in \mathbb{Z}) = (X_t^{(1)},\dots, X_t^{(d)}, t \in \mathbb{Z})$ is a multivariate strictly stationary process, and that $(\mathbf{X}_t, t=1\dots,n)$ is observable data. Let $m \in \{1,\dots,n\}$ be a block size parameter and, for $i = 1, \dots,k$ and $j = 1,\dots,d$, let $M_{m,i}^{(j)} = \max \{ X_t^{(j)} \, : \, t \in [(i-1)m, \dots,im] \}$ be the maximum of the $i$th block observations in the $j$th coordinate. For $\mathbf{X} = (x^{(1)},\dots,x^{(d)})$, let
\begin{align*}
    &\textbf{M}_{m,i} = (M_{m,i}^{(1)},\dots,M_{m,i}^{(d)}), \\
    & F_m^{(j)}(x) = \mathbb{P}\{M_{m,1}^{(j)} \leq x \}, \\
    & \textbf{F}_m(\mathbf{X}) = (F_m^{(1)}(x^{(1)}), \dots F_m^{(d)}(x^{(d)})), \\
    & U_{m,i}^{(j)} = F_m^{(j)}(M_{m,i}^{(j)}), \\
    & \textbf{U}_{m,i} = (U_{m,i}^{(1)},\dots,U_{m,i}^{(d)}).
\end{align*} 
Subsequently, we assume that the marginals of $X_1^{(1)},\dots,X_1^{(d)}$ are continuous. In that case, the marginals of $\textbf{M}_{m,1}$ are continuous as well and
\begin{equation*}
    C_m(\textbf{u}) = \mathbb{P}\{U_{m,1}^{(1)} \leq \textbf{u} \}, \quad\textbf{u} \in [0,1]^d,
\end{equation*}
is the unique copula associated with $\textbf{M}_{m,1}$. Let us consider the following set $\Delta_{d-1} = \{ (w^{(1)},\dots,w^{(d)}) \in [0,\infty)^d \,: \, \sum_{j=1}^d w^{(j)} = 1 \}$ which is the unit simplex in $\mathbb{R}^d$. Throughout, we shall work under the following data generative process.

\begin{definition}[Data generative process]
\label{def:data_gen_pro}
 Let $(\mathbf{X}_t, t\in \mathbb{Z})$ be a multivariate strictly stationary random process and $(\mathbf{X}_t, t=1,\dots,n)$ its observable data. Let $m \in \{1,\dots,n\}$ and $C_m$ the copula of $\textbf{M}_{m,1}$ such that $C_m$ is positive quadrant dependent, meaning that
\begin{equation}
	\label{eq:pod}
	C_m(\textbf{u}) \geq \Pi_{j=1}^d u^{(j)}, \quad \textbf{u} \in [0,1]^d.
\end{equation}
There exist a copula $C_\infty$, a finite Borel measure $\Phi$ on the unit positive sphere as defined in equation \eqref{eq:X_ang_meas} such that
    \begin{equation}
    	\label{eq:DMA}
        \underset{m \rightarrow \infty}{\lim} \, C_m(\textbf{u}) = C_\infty(\textbf{u}), \quad \textbf{u} \in [0,1]^d,
    \end{equation}
    where
    \begin{equation*}
        C_\infty(\textbf{u}) = \exp \left\{ - L \left(-\ln(u^{(1)}),\dots, - \ln(u^{(d)}) \right) \right\}
    \end{equation*}
    and the stable tail dependence function $L : [0,\infty)^d \rightarrow [0,\infty)$ is described by
    \begin{equation*}
        L(z^{(1)},\dots,z^{(d)}) = \sum_{a=1}^K \bigvee_{j=1}^d A_{ja}z^{(j)}.
    \end{equation*}
\end{definition}

Since extreme value copulae are positive quadrant dependent (see, e.g., \cite[Section 5.8]{resnick2008extreme}), it is expected that in practice $C_m$, i.e., a proxy to $C_\infty$ will also verify this property for a sufficiently large $m$. The max-domain of attraction in \eqref{eq:DMA} and the definition of the stable tail dependence function indicate that our observations are in the max-domain of attraction of a max-stable distribution with discrete angular measure. Then, this discrete max-stable distribution is characterised by a vector of latent factors $\mathbf{Z} \in \mathbb{R}^K$ and a matrix $A \in \mathbb{R}^{d \times K}$ that we want to estimate. As we will see, we will provide in Section \ref{sec:stat_guarantees} a non-asymptotic analysis of our estimator, which is valid for any $d$, $n$, $m$ and $k$. Such an analysis avoids the need to have the max-domain of attraction condition given in Equation \eqref{eq:DMA} to derive the main bound of the estimator's risk. However, an implicit link between $C_m$ and $C_\infty$ will be given through a bias term, which is expected to be smaller with respect to the block's length if \eqref{eq:DMA} is satisfied in the data. Nonetheless, for presentation purposes, our statistical findings will be presented with this condition satisfied.
  
Our estimation procedure is inspired from \cite{bing2020adaptative} and consists of the following four steps:

\begin{enumerate}[label=\textcolor{frenchblue}{\bf(\alph*)}]
    \item Estimate the number of clusters $K$, the pure variable set $I$ and the partition $\mathcal{I}$;
    \item Estimate $A_I$, the submatrix of $A$ with rows $A_{i\cdot}$ that correspond to $i \in I$;
    \item Estimate $A_J$, the submatrix of $A$ with rows $A_{j\cdot}$ that correspond to $j \in J$;
    \item Estimate the overlapping clusters $\mathcal{G} = \{G_1,\dots, G_K\}$.
\end{enumerate}

\subsection{Estimation of $I$ and $\mathcal{I}$}

In the context of our analysis, we need to estimate the submatrices, denoted as $A_I$ and $A_J$, separately. To begin with $A_I$, we initiate the estimation process by determining $[K]$, which subsequently allows us to identify $I$ and its partition, denoted as $\mathcal{I} = \{I_1, \dots, I_K\}$. This partition can be uniquely constructed from the extremal correlation matrix $\mathcal{X}$, as demonstrated in Theorem \ref{thm:identifiability_I}. To perform this step, we employ the constructive proof provided by Theorem \ref{thm:identifiability_I}, substituting the unknown $\mathcal{X}$ with its sampled counterpart, referred to as the extremal correlation matrix:
\begin{equation*}
    \widehat{\mathcal{X}} = [\hat{\chi}_{n,m}(i,j)]_{i =1,\dots,d, j=1,\dots,d}.
\end{equation*}
The quantity $\hat{\chi}_{n,m}(i,j)$ is the sampling version of the \emph{pre-asymptotic} extremal correlation, $\chi_m(i,j)$ between $M_{m,1}^{(i)}$ and $M_{m,1}^{(j)}$ using block-maxima approach, i.e.,
\begin{equation*}
	\chi_m(i,j) = 2 - \frac{0.5 + \nu_m(i,j)}{0.5-\nu_m(i,j)}, \quad \nu_m(i,j) = \frac{1}{2}\mathbb{E}\left[| U_{m,1}^{(i)} - U_{m,1}^{(j)} |\right].
\end{equation*}
The quantity $\nu_m(i,j)$ is the bivariate madogram (\cite{cooley2006variograms}) between $M_{m,1}^{(i)}$ and $M_{m,1}^{(j)}$. Since $\nu_m(i,j) \geq 0$, it stems down that $\chi_m(i,j) \leq 1$ and given that $C_m$ satisfies \eqref{eq:pod}, we have $\nu_m(i,j) \leq 1/6$, implying that $\chi_m(i,j) \geq 0$. These last quantities can be approached with the empirical madogram using the following relationship:
\begin{equation*}
    \hat{\chi}_{n,m}(i,j) = 2 - \frac{0.5+\tilde{\nu}_{n,m}(i,j)}{0.5-\tilde{\nu}_{n,m}(i,j)}, \quad \tilde{\nu}_{n,m}(i,j) = \min\left\{\hat{\nu}_{n,m}(i,j), \frac{1}{6} \right\}
\end{equation*}
where $\tilde{\nu}_{n,m}(i,j)$ represents the projection estimator on the segment $[0,1/6]$ of $\hat{\nu}_{n,m}(i,j)$, the bivariate empirical estimator of the madogram. This latter estimator can be readily generalised by applying the identity $|a-b| /2 = \max(a,b) - (a+b)/2$, extended to its multivariate counterpart
\begin{equation*}
    \hat{\nu}_{n,m}(1,\dots,d) = \frac{1}{k} \sum_{i=1}^k \left[ \bigvee_{j=1}^d \hat{U}_{n,m,i}^{(j)} - \frac{1}{d} \sum_{j=1}^d \hat{U}_{n,m,i}^{(j)} \right].
\end{equation*}
Here, we have standardised the marginals by ranking the observed block maxima componentwise. For a given value $x \in \mathbb{R}$, $j = 1, \dots, d$, and block size $m$, we define:
\begin{equation*}
    \hat{F}_{n,m}^{(j)}(x) = \frac{1}{k+1} \sum_{i=1}^k \mathds{1}_{ \{ M_{m,i}^{(j)} \leq x \} },
\end{equation*}
and we consider observable pseudo-observations of $U_{m, i}^{(j)}$ as follows:
\begin{equation*}
    \hat{U}_{n,m,i}^{(j)} = \hat{F}_{n,m}^{(j)}(M_{m,i}^{(j)}).
\end{equation*}

To elaborate on our approach, we construct a graph denoted as $G = (V, E)$, where the vertex set is represented as $V = [d]$. We utilise the sample version of part \ref{item_i:identifiability_I} of Theorem \ref{thm:identifiability_I} to identify the largest vector that is asymptotically independent. Using this set of indices, we then employ the sample version of part \ref{item_ii:identifiability_I} of Theorem \ref{thm:identifiability_I} to determine whether a specific index $j$ qualifies as a pure variable. If an index is not categorised as pure, we include it in the set that estimates $J$. However, if it is deemed pure, we retain the estimated set $\hat{S}_i$, as defined in \eqref{eq:S_i}, of indices $j \neq i$ that are strongly associated, through their extremal correlations, with $i$. Subsequently, we utilise the constructive proof of part \ref{item_iii:identifiability_I} of Theorem \ref{thm:identifiability_I} to declare that $\hat{S}_i \cup {i} = \hat{I}^{(i)}$ serves as an estimator for one of the partition sets within $\mathcal{I}$.

For a comprehensive understanding of the algorithm, including the specification of a tuning parameter denoted as $\delta$, please refer to Algorithm \ref{alg:pure_variable} in \ref{subsec:algorithms}. The discussion pertaining to the tuning parameter $\delta$ will be presented in detail in Section \ref{subsec:stat_guarantees_k_i}.

\subsection{Estimation of the allocation matrix $A$ and soft clusters.}
\label{subsec:estim_alloc}
Based on the estimators $\hat{I}$, $\hat{K}$, and $\hat{\mathcal{I}} = \{ \hat{I}_1, \dots, \hat{I}_{\hat{K}} \}$ obtained from Algorithm \ref{alg:pure_variable}, we estimate the matrix $A_I$. This estimation takes the form of a matrix with dimensions $|\hat{I}| \times \hat{K}$, where each row corresponding to an index $i$ in $\hat{I}$ contains $\hat{K}-1$ zeros and one entry equal to 1. This procedure induces the following estimator of $A_I$

\begin{equation}
    \label{eq:estimator_A_I}
    \hat{A}_{ka} = \hat{A}_{la} = 1, \quad \textrm{for } k,l \in \hat{I}_a, \quad a \in [\hat{K}].
\end{equation}

We continue by estimating the matrix $A_J$, row by row. To explain our approach, we first outline the structure of each row, denoted as $A_{j \cdot}$, within the matrix $A_J$, for $j \in J$. We should note that each $A_{j \cdot}$ satisfies sparsity conditions and $\sum_{a=1}^K A_{ja} = 1$, as stipulated by Condition \ref{cond:(i)}. To simplify the exposition, we rearrange $\mathcal{X}$ and $A$ as follows
\begin{equation*}
    \mathcal{X} = \begin{bmatrix}
\mathcal{X}_{II} & \mathcal{X}_{IJ}\\
\mathcal{X}_{JI} & \mathcal{X}_{JJ}
\end{bmatrix},
\quad
A = \begin{bmatrix}
A_I\\
A_J
\end{bmatrix}.
\end{equation*}
Model \eqref{eq:lfm} and Theorem \ref{thm:matrix_product} imply the following decomposition of the extremal correlation matrix of $\mathbf{X}$
\begin{equation*}
     \mathcal{X} = \begin{bmatrix}
\mathcal{X}_{II} & \mathcal{X}_{IJ}\\
\mathcal{X}_{JI} & \mathcal{X}_{JJ}
\end{bmatrix}
=
\begin{bmatrix}
A_I \odot A_I^\top & A_I \odot A_J^\top\\
A_J^\top \odot A_I  & A_J \odot A_J^\top 
\end{bmatrix}.
\end{equation*}
In particular, $\mathcal{X}_{IJ} = A_I \odot A_J^\top$. Thus for each $i \in I_a$ with some $a \in [K]$ and $j \in J$, we have
\begin{equation*}
    \chi(i,j) = A_{ja}.
\end{equation*}
Averaging the above display over all $i \in I_a$ yields
\begin{equation*}
    \frac{1}{|I_a|} \sum_{i \in I_a} \chi(i,j) = A_{ja}.
\end{equation*}
Hence
\begin{equation*}
    \beta^{(j)} := A_{j \cdot} = \left( \frac{1}{|I_1|} \sum_{i \in I_1} \chi(i,j), \dots, \frac{1}{|I_K|} \sum_{i \in I_K} \chi(i,j) \right),
\end{equation*}
which can be estimated from the data as follows. For each $j \in \hat{J}$, we denote an estimator for the $a$-th element of $\beta^{(j)}$ using a simple approach. This estimator is represented as follows
\begin{equation*}
    \bar{\chi}^{(j)} = \left( \frac{1}{|\hat{I}_1|} \sum_{i \in \hat{I}_1} \hat{\chi}_{n,m}(i,j), \dots, \frac{1}{|\hat{I}_{\hat{K}}|} \sum_{i \in \hat{I}_{\hat{K}}} \hat{\chi}_{n,m}(i,j) \right).
\end{equation*}
It is important to note that this estimator is neither sparse nor an element of the unit simplex. Given the value $\bar{\chi}^{(j)}$, our objective is to determine a Euclidean projection of $\bar{\chi}^{(j)}$ that lies within the space $\mathbb{B}_0(s) = \{ \beta \in \mathbb{R}^{\hat{K}}, \, \sum_{j=1}^{\hat{K}} \mathds{1}_{\{ \beta^{(j)} \neq 0 \} } \leq s \}$, i.e., vectors with at most $s$ non-zero entries, and the unit simplex $\Delta_{\hat{K}-1} = \{ \beta \in \mathbb{R}^{\hat{K}}, \, \beta^{(j)} \geq 0, \sum_{j=1}^{\hat{K}} \beta^{(j)} = 1 \}$ :
\begin{equation}
    \label{eq:proj_simplex}
    \mathcal{P}(\hat{\beta}^{(j)}) \in \underset{\beta : \beta \in \mathbb{B}_0(s) \cap \Delta_{\hat{K}-1}}{\textrm{argmin}} || \beta - \bar{\chi}^{(j)}||_2.
\end{equation}

\cite{kyrillidis2013sparse} have demonstrated the feasibility of computing such a projection efficiently, using a simple greedy algorithm. This algorithm is outlined in \ref{alg:htsp} (Hard Thresholding and Simplex Projector), and it involves two main steps: first, identifying the support of $\hat{\beta}^{(j)}$, and then determining these values associated with this support. Consequently, in order to construct an estimator for the support, we select only the coordinates indexed by $a$ where $\bar{\chi}^{(j)}_a$ exceeds a threshold $\delta$. This selection results in a sparse estimator for $\beta^{(j)}_a$ as follows
\begin{equation*}
    \bar{\beta}^{(j)}_a = \bar{\chi}^{(j)}_a \mathds{1}_{ \{ \bar{\chi}^{(j)}_a > \delta \} }, \, a = 1, \dots, \hat{K}.
\end{equation*}
This estimator $\bar{\beta}^{(j)}$ is often referred to as the hard thresholding estimator, where $\delta>0$ represents the threshold value. However, it is essential to note that the estimator $\bar{\beta}^{(j)}$ does not inherently belong to the unit simplex. To address this, we can obtain an alternative estimator, denoted as $\hat{\beta}^{(j)}$, by projecting $\bar{\beta}^{(j)}$ onto the unit simplex within the $\hat{K}$-dimensional space. The projection operation onto the unit simplex is achieved by utilising a specific mathematical operator, and this operator is defined as
\begin{equation*}
    (\mathcal{P}_{\Delta_{\hat{K}-1}}(\beta))_{j} = [\beta^{(j)} - \tau]_+, \, \tau := \frac{1}{\rho} \left( \sum_{i=1}^\rho \beta^{(j)} - 1 \right),
\end{equation*}
for $\rho := \max \{ k, \, \beta^{(j)} > \frac{1}{k}(\sum_{j=1}^k w^{(j)}-1) \}$. Hence, by denoting $\widehat{\mathcal{S}} = \textrm{supp}(\bar{\beta}^{(j)})$, we obtain 
\begin{equation}
    \label{eq:estimator_A_J}
    \left.\hat{\beta}^{(j)}\right|_{\widehat{\mathcal{S}}} = \mathcal{P}_{\Delta_{\hat{K}-1}}\left(\left.\bar{\beta}^{(j)}\right|_{\widehat{\mathcal{S}}}\right), \quad \left.\hat{\beta}^{(j)}\right|_{\widehat{\mathcal{S}}^c} = 0
\end{equation}
Next, we construct the matrix $\hat{A}_{\hat{J}}$ with rows corresponding to the estimators $\hat{\beta}^{(j)}$ for each $j \in \hat{J}$. Our final estimator, denoted $\hat{A}$, for the matrix $A$, is obtained by concatenating $\hat{A}_{\hat{I}}$ and $\hat{A}_{\hat{J}}$. The statistical properties of the final estimator are examined in Section \ref{sec:stat_guarantees}, where we also provide detailed specifications of the tuning parameter necessary for its development.

Recalling the definition of groups in \eqref{eq:pure_clusters}, the soft clusters are estimated by
\begin{equation}
    \label{eq:estimator_over_clust}
    \hat{\mathcal{G}} = \{ \hat{G}_1, \dots, \hat{G}_{\hat{K}} \}, \; \hat{G}_a = \left\{i \in [d] \, : \, \hat{A}_{ia} \neq 0\right\}, \textrm{ for each } a \in [\hat{K}].
\end{equation}
Variables $X^{(j)}$ that are associated (via $\hat{A}$) with the same latent factor $Z^{(a)}$ are therefore placed in the same group $\hat{G}_a$. We demonstrate in Section \ref{subsec:stat_guarantees_A} that the overlapping clusters or extreme directions can be controlled with high probability.
\section{Statistical guarantees}
\label{sec:stat_guarantees}
This section serves a dual purpose. Firstly, we derive a bound akin to Bernstein's inequality for the uniform norm of the sampled version of the extremal correlation. This is achieved by utilising a sequence of dependent and strictly stationary random variables denoted as $(\mathbf{X}_t,t \in \mathbb{Z})$. Secondly, we employ the developed techniques to investigate statistical guarantees for the following aspects:
\begin{enumerate}[label=\textcolor{frenchblue}{\bf(\alph*)}]
    \item The estimated number of clusters, denoted as $\hat{K}$;
    \item The estimated pure variable set $\hat{I}$ and its corresponding estimated partition $\hat{\mathcal{I}}$;
    \item The estimated allocation matrix $\hat{A}$ and its adjustment to account for the unknown sparsity of rows in matrix $A$;
    \item Guarantees to recover overlapping clusters / extreme directions in terms of Total False Positive Proportion (TFPP) and Total False Negative Proportion (TFNP).
\end{enumerate}
We recall the definition of strongly mixing sequences, introduced by \cite{rosenblatt1956central}: For any two $\sigma$-algebras $\mathcal{A}$ and $\mathcal{B}$, we define the $\alpha$-mixing coefficient by
\begin{equation*}
    \alpha(\mathcal{A}, \mathcal{B}) = \underset{A \in \mathcal{A}, B \in \mathcal{B}}{\sup} \, | \mathbb{P}(A \cap B) - \mathbb{P}(A)\mathbb{P}(B) |.
\end{equation*}
Let $(\mathbf{X}_t, t\in \mathbb{Z})$ be a sequence of real-valued random variables defined on $(\Omega, \mathcal{A}, \mathbb{P})$. This sequence will be called strongly mixing if
\begin{equation}
    \label{eq:alpha_mixing}
    \alpha(n) := \underset{k \geq 1}{\sup} \, \alpha(\mathcal{F}_k, \mathcal{G}_{k+n}) \rightarrow 0 \, \textrm{as} \, n \rightarrow \infty,
\end{equation}
where $\mathcal{F}_j := \sigma(\mathbf{X}_i, i \leq j)$ and $\mathcal{G}_j := \sigma(\mathbf{X}_i, i \geq j)$ for $j \geq 1$. Throughout, we assume the sequence $(\mathbf{X}_t, t\in \mathbb{Z})$ has exponentially decaying strong mixing coefficients, that is
\begin{equation}
    \label{eq:exp_dec}
    \alpha(n) \leq \exp\{-2 c n\},
\end{equation}
for a certain $c > 0$. Our first result is the following exponential inequality.
\begin{theorem}
    \label{thm:exp_ineq}
    Let $(\mathbf{X}_t, t\in \mathbb{Z})$ be a sequence of stationary random variables with exponential decaying strong mixing coefficients given in \eqref{eq:exp_dec}. Then, there are constants $c_0 > 0$ and $c_1 > 0$ which depend only on $c$ such that
    \begin{equation*}
        \mathbb{P} \left\{  \underset{1 \leq i < j \leq d}{\sup} | \hat{\chi}_{n,m}(i,j) - \chi_m(i,j) | \geq c_1 \left( \sqrt{\frac{\ln\left(k d\right)}{k}} + \frac{\ln(k) \ln \ln(k)\ln(k d)}{k} \right) \right\} \leq d^{-c_0},
    \end{equation*}
    where $n$, $m \in \{1,\dots,n\}$ denote respectively the sample and the block size, $k = \lfloor n / m \rfloor \geq 4$ and $\chi_m(i,j) \in [0,1]$ is the \emph{pre-asymptotic} extremal correlation between $M^{(i)}_{m,1}$ and $M^{(j)}_{m,1}$.
    
\end{theorem}

In the context of our section, let us introduce some mathematical notations and concepts. Consider the quantity $\chi(i,j)$, which represents the extremal correlation between $X^{(i)}$ and $X^{(j)}$, within the max-domain of attraction as specified by Equation \eqref{eq:DMA}. We also define a crucial parameter denoted as:

\begin{equation*}
    d_m = \underset{1 \leq i < j \leq d}{\sup} \, |\chi_m(i,j) - \chi(i,j) |,
\end{equation*}
where $\chi_m(i,j)$ is the \emph{pre-asymptotic} extremal correlation between $M_{m,1}^{(i)}$ and $M_{m,1}^{(j)}$. This parameter characterises the explicit bias between the subasymptotic framework and the max-domain of attraction. It essentially quantifies the rate at which the system converges to its asymptotic extreme behavior. Additionally, we introduce the following new event:
\begin{equation}
    \label{eq:mathcal_E}
    \mathcal{E} = \mathcal{E}(\delta):= \left\{ \underset{1 \leq i < j \leq d}{\sup} \, \left| \hat{\chi}_{n,m}(i,j) - \chi(i,j) \right| \leq \delta \right\}.
\end{equation}
Now, we state that
\begin{equation}
	\label{eq:theoretical_delta}
    \delta = d_m + c_1 \left( \sqrt{\frac{\ln\left(k d\right)}{k}} + \frac{\ln(k) \ln\ln(k)\ln(k d)}{k} \right),
\end{equation}
for some absolute constant $c_1 > 0$. Furthermore, we require that $\ln(d) = o(k)$. This condition ensures, with the additional max-domain of attraction in \eqref{eq:DMA}, that $\delta = o(1)$ provided that $m = o(n)$. Taking $c_1 > 0$ large enough, Theorem \ref{thm:exp_ineq} guarantees that $\mathcal{E}$ holds with high probability:
\begin{equation*}
    \mathbb{P}(\mathcal{E}) \geq 1-d^{-c_0},
\end{equation*}
for some positive constant $c_0 > 0$. The order of the threshold $\delta$ involves known quantity such as $d$, $k$ and a unknown parameter $d_m$. For the latter, there is no simple manner to choose optimally this parameter, as there is no simple way to determine how fast is the convergence to the asymptotic extreme behavior, or how far into the tail the asymptotic dependence structure appears. 

\subsection{Statistical guarantees for $\hat{K}$, $\hat{I}$ and $\hat{\mathcal{I}}$.}
\label{subsec:stat_guarantees_k_i}

We now move forward with the analysis of the statistical performance of our estimator, denoted as $\hat{I}$, which aims to estimate $I$. Alongside this estimation, we also consider its associated partition. This particular problem falls within the broader category of pattern recovery problems. It is well established that, given sufficiently strong signal conditions, we can reasonably expect that $\hat{I} = I$ with a high level of confidence. These conditions are stated below
\begin{Assumption}{$(\mathcal{SSC})$}
	\label{cond:SSC}
    Let $\mathcal{I} = \{I_a\}_{1 \leq a \leq K}$.
    \begin{enumerate}[label=\textcolor{frenchblue}{($\mathcal{SSC}\arabic*$)}]
    \item $\forall k \notin \mathcal{I}$, $A_{ka} < 1-2\delta$, $\forall a \in [K]$; \label{cond:SSC1}
    \item $\forall k \notin \mathcal{I}$, $\exists a,b \in [K]$ such that $A_{ka} > 2\delta$ and $A_{kb} > 2\delta$. \label{cond:SSC2}
\end{enumerate}
\end{Assumption}
In Theorem \ref{thm:stat_K_I}, we establish a critical result that provides a high-confidence guarantees for recovery of $K$, $I$ and $\mathcal{I}$ under the condition \ref{cond:SSC}. This theorem has several key implications. In the first aspect of these implications, our theorem demonstrate that, with high probability, the estimated set $\hat{I}$ is equal to the set of pure variables $I$. This implies that our procedure correctly identifies these fundamental variables. Ambiguous variables $X^{(j)}$ with $j \notin I$ that exhibit associations with multiple latent factors not exceeding the $1-2\delta$ threshold, as dictated by Condition \ref{cond:SSC1}, are deliberately excluded from $\hat{I}$. This exclusion is crucial to maintain the accuracy of our method. In the construction of the graph $G$ in part \ref{item_i:identifiability_I} of Theorem \ref{thm:identifiability_I}, Condition \ref{cond:SSC2} guarantees that only ambiguous variables are excluded from a maximum clique of $G$. This is of utmost importance because it ensures that the number of cluster $K$ is determined accurately.

\begin{theorem}
    \label{thm:stat_K_I}
    Let $(\mathbf{X}_t, t\in \mathbb{Z})$ be a sequence of stationary random variables with exponential decaying strong mixing coefficients given in \eqref{eq:exp_dec}, satisfying the data generative process given in Definition \ref{def:data_gen_pro}. Under Conditions \ref{cond:(i)}-\ref{cond:(ii)} and Condition \ref{cond:SSC}, then
    \begin{enumerate}[label=\textcolor{frenchblue}{\bf(\alph*)}]
    \item $\hat{K} = K$; \label{thm:stat_K_I_1}
    \item $I = \hat{I}$.\label{thm:stat_K_I_2}
    \end{enumerate}
    Moreover, there exists a label permutation $\pi$ of the set $\{1,\dots,K\}$ such that the output $\hat{\mathcal{I}} = \{\hat{I}_a\}_{1 \leq a \leq K}$ from Algorithm \ref{alg:pure_variable} satisfies
    \begin{enumerate}[label=\textcolor{frenchblue}{\bf(\alph*)}]
    \setcounter{enumi}{2}
    \item $I_{\pi(a)} = \hat{I}_{\pi(a)}$. \label{thm:stat_K_I_3}
    \end{enumerate}
    All results hold with probability larger than $1-d^{-c_0}$ for a positive constant $c_0$.
\end{theorem}

\subsection{Statistical guarantees for $\hat{A}$}
\label{subsec:stat_guarantees_A}

In this section, we present and discuss the statistical properties of the estimator $\hat{A}$ and its control over the relationship between the support of $A$ and the support of $\hat{A}$. It is worth noting that $\delta$ is defined in Equation \eqref{eq:mathcal_E} and the estimator of $A$ relies solely on this tuning parameter. Theorem \ref{thm:stat_A} established an adaptative finite sample upper bound for exponentially decaying $\alpha$-mixing observations. Both $d$ and $K$ are allowed to grow with $n$. We consider the loss function for two $d \times K$ matrices $A,A'$ as
\begin{equation}
    \label{eq:loss_function_chap_4}
    L_2(A,A') := \underset{P \in S_K}{ \min} ||AP - A'||_{\infty,2}
\end{equation}
where $S_K$ is the group of all $K \times K$ permutation matrices and
\begin{equation*}
    ||A||_{\infty,2} := \underset{1\leq j\leq d}{\max} ||A_{j\cdot}||_2 = \underset{1\leq j\leq d}{\max} \left( \sum_{j=1}^K |A_{ij}|^2 \right)^{1/2};
\end{equation*}
for a generic matrix $A \in \mathbb{R}^{d \times K}$. Finally given $\delta$ in \eqref{eq:theoretical_delta}, we define 

\begin{equation}
    J_1 =\{ j \in J : \textrm{for any $a \in [K]$ with $A_{ja} \neq 0$}, A_{ja} > 2 \delta \}.
\end{equation}

\begin{theorem}
    \label{thm:stat_A}
    Assume the conditions in Theorem \ref{thm:stat_K_I} hold. Set $s = \underset{j \in [d]}{\max} ||A_{j\cdot}||_0$, $s(j) = \sum_{a=1}^K \mathds{1}_{ \{ A_{ja} > 0 \} }$ and $t(j) = \sum_{a=1}^K \mathds{1}_{ \{ A_{ja} \leq 2 \delta \} }$ for each $j \in J$. Then for the estimator $\hat{A}$ the following holds.
    \begin{enumerate}[label=\textcolor{frenchblue}{\bf(\alph*)}]
        \item \label{item_i:thm_stat_A_upper_bound} An upper bound:
        \begin{equation*}
        L_2(\hat{A},A) \leq 4 \sqrt{s} \delta,
    \end{equation*}
    \item \label{item_ii:thm_stat_A_supp_recov} A guarantee for support recovery:
    \begin{equation*}
        \normalfont{supp}(A_{J_1}) \subseteq \normalfont{supp}(\hat{A}) \subseteq \normalfont{supp}(A),
    \end{equation*}
    \item \label{item_iii:thm_stat_A_clust_recov} Cluster recovery:
    \begin{align*}
    	&TFPP(\widehat{\mathcal{G}}) = \frac{\sum_{j\in [d], a \in [K]} \mathds{1}_{\{ A_{ja} = 0, \hat{A}_{ja} > 0 \}}}{\sum_{j\in [d], a \in [K]} \mathds{1}_{\{ A_{ja} = 0\}}} = 0, \\ &TFNP(\widehat{\mathcal{G}}) = \frac{\sum_{j\in [d], a \in [K]} \mathds{1}_{\{ A_{ja} > 0, \hat{A}_{ja} = 0 \}}}{\sum_{j\in [d], a \in [K]} \mathds{1}_{\{ A_{ja} > 0\}}} \leq \frac{\sum_{ j \in J \setminus J_1} t(j)}{|I| + \sum_{j\in J} s(j)},
    \end{align*}
    \end{enumerate}
     with probability larger than $1-d^{-c_0}$ for a positive constant $c_0$.
\end{theorem}
Some comments on the above results are in order. On a high level, larger dimensions $d$, larger values of $d_m$ lead to a larger bound. The effects of dimension $d$ and bias $d_m$ are intuitive: larger dimensions or more bias make the matrix recovery problem more difficult. The dependence on the number of latent factors $K$ is implicitly conveyed through the sparsity index $s$, and its impact on the bound is also straightforward: the sparser the matrix, the better the bound. Additionally, we have demonstrate that $\hat{A}$ possesses another desirable property, namely variable selection, or exact recovery sparsity pattern for row $j \in J_1$.

For a more quantitative analysis in increasing dimensions, assume there exists a coefficient $\rho_{\Psi} > 0$ and a continuous non-zero function $S$ on $[0,1]^2$ such that for any pair $(a,b) \in \{1,\dots,d\}^2$, the following holds:
    \begin{equation*}
        C_m^{(a,b)}(\textbf{u}) - C_\infty^{(a,b)}(\textbf{u}) = \frac{1}{m^{\rho_{\Psi}}}S(\textbf{u}) + o\left(\frac{1}{m^{\rho_{\Psi}}}\right), \quad  \mbox{ for } m \rightarrow \infty,
    \end{equation*}
    uniformly for $\mathbf{u} \in [0,1]^2$ (see, e.g., \cite{bucher2019second} for a detailed introduction to this condition). Here, $C_m^{(a,b)}$ and $C_\infty^{(a,b)}$ denote the bivariate copula of maxima and the bivariate extreme value copula, respectively. In this case, it can be shown that $d_m = O(m^{-\rho_{\Psi}})$, where the $O$-term is independent of $d$. Thus, for sufficiently large $n$, and simplifying the expression by omitting logarithmic terms in $k$ without altering the main conclusion, Theorem \ref{thm:stat_A} yields the following upper bound:
    \begin{equation*}
    	L_2(\hat{A},A) \leq c_1 \sqrt{s}\left(\left(\frac{k}{n}\right)^{\rho_\Psi} + \sqrt{\frac{\ln(d)}{k}} \right),
    \end{equation*}
    where $c_1 > 0$ is an absolute constant. Straightforward calculations lead to an optimal $k^* = c_{\rho_{\Psi}} n^{\frac{2\rho_{\Psi}}{2\rho_{\Psi}+1}} (\ln(d))^{\frac{1}{2\rho_{\Psi}+1}}$, where $c_{\rho_{\Psi}}$ is a constant dependent on $\rho_{\Psi}$, but independent of $d$. Substituting this optimal $k^*$ into the bound, we obtain:
    \begin{equation*}
    	L_2(\hat{A},A) \leq c_{\rho_{\Psi}} \sqrt{s} \left( \frac{\ln(d)}{n} \right)^{\frac{\rho_{\Psi}}{2\rho_{\Psi}+1}}.
    \end{equation*}
    Thus, the estimation procedure remains accurate as long as $\ln(d) = o(n)$.
    
To the best of our knowledge, the estimation of identifiable sparse loading matrices $A$ and overlapping clusters (or termed equivalently as extreme directions) in the model \eqref{eq:lfm}, meeting Conditions \ref{cond:(i)}-\ref{cond:(ii)}, when both $I$ and $K$ are unknown and when $d$ and $n$ are varying, has not been yet explored in existing literature. Our results fill this gap in the research landscape. 
\section{Numerical results}
\label{sec:num_res}
\subsection{A data-driven selection method for the tuning parameter}
\label{subsec:data_driven}

Theorems \ref{thm:stat_K_I}-\ref{thm:stat_A} outline the theoretical rate of $\delta$, but only up to constants. Below, we propose a method for selecting $\delta$ based only on data. We opt for a finely tuned grid of values $\delta_\ell = c_\ell (d_m + \sqrt{\ln(d)/k})$ with $1 \leq \ell \leq M$ as suggested by \eqref{eq:theoretical_delta} for $k$ sufficiently large and omitting log factors in $k$. This selection process involves varying the proportionality constants $c_\ell$. For each $\delta_\ell$ chosen, we determine the number of clusters $\hat{K}(\ell)$ using Algorithm \ref{alg:clust}, and the corresponding matrix, $\hat{A}(\ell)$. We denote the associated overlapping clusters as $\hat{G}(\ell) = \{\hat{G}_1(\ell), \dots, \hat{G}_{\hat{K}(\ell)}(\ell)\}$. Define
\begin{equation}
	\label{eq:crit_tuning}
	\mathcal{L}(\hat{G}(\ell)) = \sum_{j=1}^d \sum_{a \in [\hat{K}(\ell)]} \left(\hat{A}_{ja}(\ell) - \hat{\chi}_{n,m}(j,a) \mathds{1}_{ \{ \hat{A}_{ja}(\ell) > 0 \} } \right)^2.
\end{equation}
We then proceed to select $\delta^\star$ as the value of $\delta_\ell$ that minimises the criteria in \eqref{eq:crit_tuning}. This leads to data-driven selection of $\delta^\star = c^\star (d_m + \sqrt{\ln(d)/k})$. While our choice may not be optimal in certain challenging scenarios, it provides effective data-based guidelines for our comparative analysis in Section \ref{subsec:num_comp} and real-world data evaluation in Section \ref{sec:applications}.

\subsection{Performance of the proposed methodology in finite sample setting}

\label{subsec:num_res}

In this section, we investigate the finite-sample performance of our algorithm to estimate the matrix $A$ in a linear factor model described in \eqref{eq:lfm} by means of a simulation study.

\paragraph{The setup.} As a time series model, we consider the discrete-time, $d$-variate moving maxima process $(\textbf{Z}_t, t \in \mathbb{Z})$ of order $p \in \mathbb{N}$ given by
\begin{equation}
    \label{eq:moving_maxima}
    Z_t^{(a)} = \bigvee_{\ell = 0}^p \rho^\ell \epsilon_{t + \ell}^{(a)}, \quad (t \in \mathbb{Z}, a=1,\dots,K), \quad \rho \in (0,1).
\end{equation}
Here $(\boldsymbol{\epsilon}_t, t \in \mathbb{Z})$ is an i.i.d. sequence of $K$-dimensional random vectors having the following distribution
\begin{equation*}
    \mathbb{P}\left\{ \boldsymbol{\epsilon}_1 \leq \mathbf{X} \right\} = \varphi^\leftarrow \left( \varphi(P(x^{(1)})) + \dots + \varphi(P(x^{(K)})) \right), \quad \mathbf{X} \in \mathbb{R}^K,
\end{equation*}
where $\varphi$ is the Archimedean generator $\varphi : [0,1] \rightarrow [0,\infty]$, $\varphi^\leftarrow$ its generalised inverse and $P(x) = 1-1/x$ for $x \geq 1$ is the cumulative distribution function of a standard Pareto random variable. We assume the existence of the limit of 
\begin{equation}
    \label{eq:rv}
    \underset{s \rightarrow 0}{\lim} \,\frac{\varphi(1-st)}{\varphi(1-s)} = t, \quad t \in (0,\infty),
\end{equation}
i.e., the upper tail exhibits asymptotic independence. In \ref{sec:proof_num_res} (as seen in Proposition \ref{prop:dma}), we demonstrate that the maxima of $(\textbf{Z}_t, t \in \mathbb{Z})$ belong to the maximum domain of attraction of independent Fréchet distributions. Consequently, the related process verifies the max-domain of attraction \eqref{eq:DMA}
\begin{equation}
    \label{eq:model_num_res}
    \mathbf{X}_t = A \textbf{Z}_t + \textbf{E}_t,
\end{equation}
where $\textbf{E}_t$ represents independent and identically distributed replications of a lighter tail distributions. Noteworthy, since $(\textbf{Z}_t, t \in \mathbb{Z})$ is $p$-dependent, the overall process $(\mathbf{X}_t, t \in \mathbb{Z})$ is $\alpha$-mixing.

We generate the data in the following way. We set the number of clusters $K$ to be $20$ and simulate the latent variables $Z = (Z^{(1)},\dots, Z^{(K)})$ from \eqref{eq:moving_maxima} and $\varphi = (t^{-1}-1)$. The error terms $E^{(1)}_i,\dots,E^{(d)}_i$ are independently generate from a standard normal distribution. To speed up the simulation process, we set the first $20$ rows of $A$ as pure variables. This means these rows are predetermined and do not change during the process so that the best permutation matrix that achieves the best estimation error is the identity. To generate $A_J$, for any $j \in J$, we randomly assign the cardinality $s_j$ of the support $A_{j\cdot}$ to a number in $\{2,3,4\}$, with equal probability. Then, we randomly select the support from the set $\{1,2,\dots,K\}$ with a cardinality equal to $s_j$. We then sample $U_1,\dots,U_{s_j}$ uniformly over the segment $[0.35, 0.65]$. Finally, we assign the value of variables in the support as the corresponding sampled value divided by the sum of all sampled values for variables in the support. Thus, we can generate $\mathbf{X}$ according to the model in \eqref{eq:model_num_res} and setting $\rho = 0.8$ and $p = 2$.

\paragraph{The target values.} Our simulation study aims at investigating the performance of our algorithm to recover the number of latent variables $K$ and the performance of $\hat{A}$ as estimators of $A$. When the number of clusters is correctly identified, we compute the norm $L_2(\hat{A},A)$ in \eqref{eq:loss_function_chap_4}.

\paragraph{Calibrating parameters.} In practice, based on Equation \eqref{eq:theoretical_delta}, we recommend the following choice, for sufficiently large $k$, and omitting logarithmic terms in $k$ with $d_m = c_2/m$ as a rule of thumb (see \ref{sec:proof_num_res} for technical details of this heuristic):
\begin{equation*}
\delta = \frac{c_2}{m} + c_1 \sqrt{\frac{\ln(d)}{k}},
\end{equation*}
and set $c_1 = 1.2$ and $c_2 = 1.0$ in Algorithm \ref{alg:clust}. We have found that these choices for $c_1$ and $c_2$ not only yield good overall performance but are also robust with respect to the dimension $d$, the block's length $m$ and the numbler of blocks $k$.

\paragraph{Results and discussion.} Figure \ref{fig:num_res_lfm}, Panel \subref{subfig:exact_recov_rate_k} present simulation results on exact recovery rate of number of clusters $K$ and estimation error $L_2(\hat{A},A)$ in the case of a fixed $m=15$ and with varying $k \in \{300,500,700,1000\}$. The shape of the functions are as expected; the estimation error decreases when the number of blocks $k$ increases from $300$ to $1000$, which is in line with our theoretical results. The simulations are conducted on an Ubuntu system version 22.04 with 2.5 GHz Intel Core i7 CPU and 32 GB memory. Even with $d = 1000$ and $k = 1000$, the computing time of our method for each simulation is around $5$ seconds. In Figure \ref{fig:num_res_lfm}, Panel \subref{subfig:exact_recov_rate_m}, we present simulation results on exact recovery rate of the number of cluster $\hat{K}$ and estimation error $L_2(\hat{A}, A)$ with a fixed sample size $n = 5000$. We plot these two quantities against the number of blocks $k$. From the picture, we see that, as expected, the performance of the estimator is first decreasing in $k$ while it increases after a certain threshold. This phenomenom is an illustration of the bias-variance tradeoff that runs is the choice of the corresponding block length. Regarding the exact recovery rate and the performance of the estimator, we observe a rather good performance for a value of $k = 700$, corresponding to block length $m = 7$.

\begin{figure}[!h]
\begin{minipage}{.5\linewidth}
    \centering
    \subfloat[]{\label{subfig:exact_recov_rate_k}\includegraphics[width=\textwidth, height=0.3\textheight]{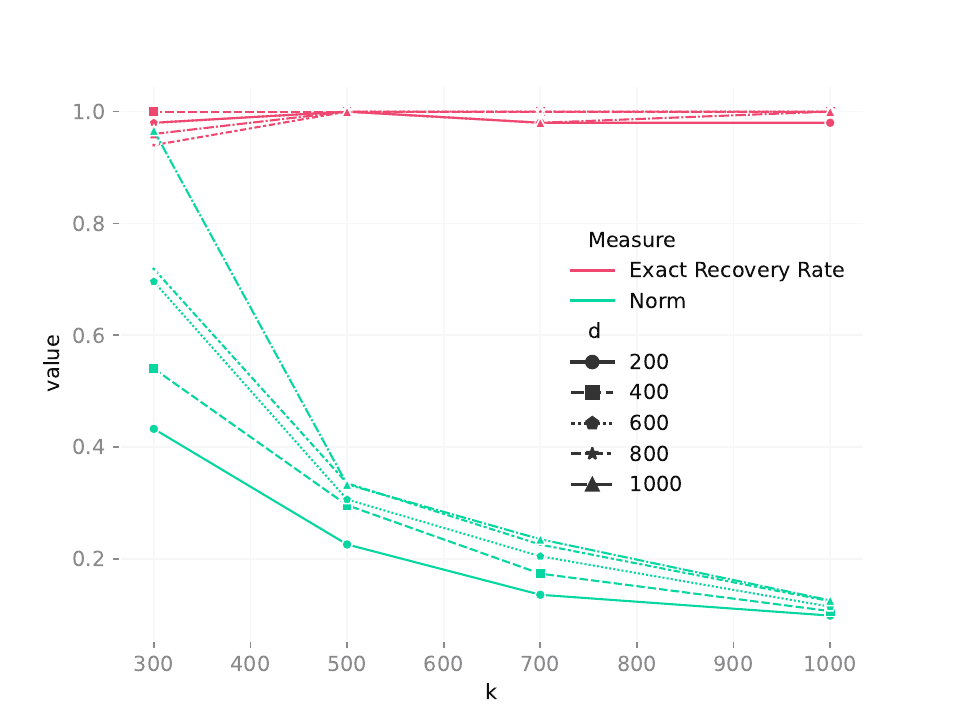}}
\end{minipage}%
\begin{minipage}{.5\linewidth}
    \centering
    \subfloat[]{\label{subfig:exact_recov_rate_m}\includegraphics[width=\textwidth, height=0.3\textheight]{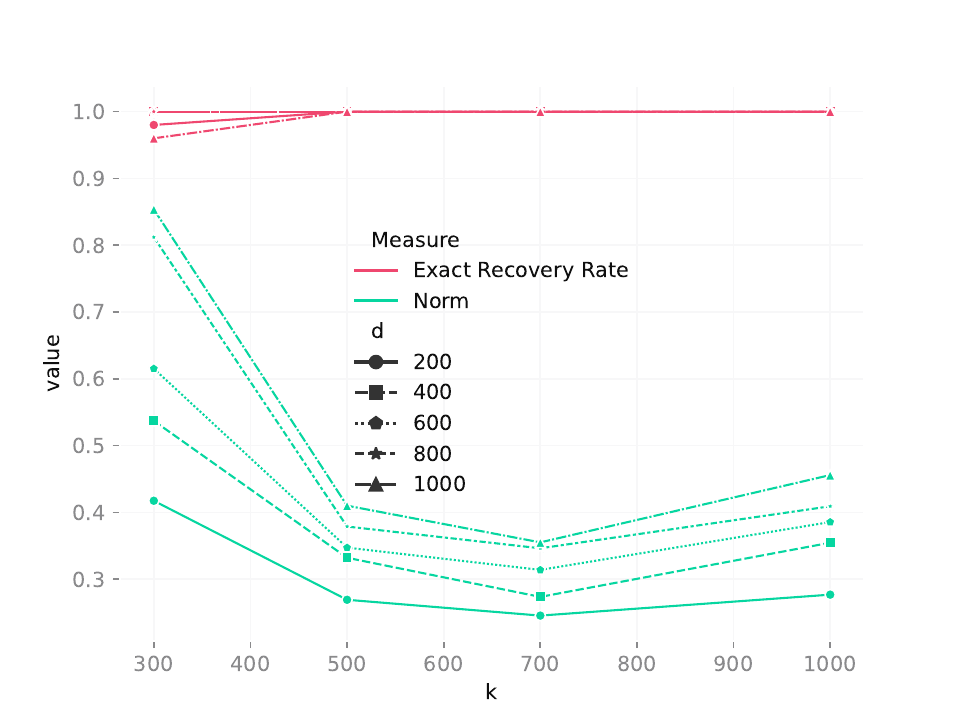}}
\end{minipage}

\caption{In Panel \subref{subfig:exact_recov_rate_k}, simulation results on exact recovery rate of number of clusters $K$ (in red) of $\mathcal{I}$ and $L_2(\hat{A},A)$ (in green) with fixed $m=20$ and varying $d \in \{200,400,600,800,1000\}$. In Panel \subref{subfig:exact_recov_rate_m}, simulations results on exact recovery rate (in blue) of $\mathcal{I}$ and $L_2(\hat{A},A)$ (in red) for fixed $n = 5000$ and varying number of blocks and $d \in \{200,400,600,800,1000\}$.}
\label{fig:num_res_lfm}
\end{figure}

Finally, within Table \ref{tab:values_k_d}, we display the parameter $\delta^\star$ which minimises the average of the criteria in \eqref{eq:crit_tuning} using the data-driven selection method proposed in Section \ref{subsec:data_driven} over $50$ runs with different numbers of block maxima $k$ and dimensions $d$. Additionally, we give the average ERR over these $50$ runs. These results provide support for this criterion in selecting $\delta$, yielding favorable outcomes in terms of ERR. However, it is also evident that this criterion may not be optimal in certain difficult scenarios, particularly for large values of $d$ and small values of $k$.
\begin{table}[htbp]
    \centering
    \begin{tabular}{c|ccccc}
        \toprule
        \( k \) / \( d \) & \multicolumn{5}{c}{\(c^\star\) / ERR} \\
        \midrule
        & 200 & 400 & 600 & 800 & 1000 \\
        \midrule
        300 & (\(1.2\), 1) & (\(1.13\), 1) & (\(1.05\), 0.98) & (\(1\), 0.92) & (\(0.985\), 0.86) \\
        500 & (\(1.41\), 1) & (\(1.38\), 1) & (\(1.3\), 1) & (\(1.28\), 1) & (\(1.29\), 1) \\
        700 & (\(1.63\), 1) & (\(1.6\), 1) & (\(1.5\), 1) & (\(1.47\), 1) & (\(1.33\), 1) \\
        1000 & (\(1.83\), 1) & (\(1.64\), 1) & (\(1.74\), 1) & (\(1.69\), 1) & (\(1.35\), 1) \\
        \bottomrule
    \end{tabular}
    \caption{Data-driven selection of $c^\star$ using the average criterion \eqref{eq:crit_tuning} and Exact Recovery Rate (ERR) of latent factors $K = \hat{K}$ over $50$ runs with varying $k \in \{300,500,700,1000\}$ and $d \in \{200,400,600,800,1000\}$.}
    \label{tab:values_k_d}
\end{table}

\subsection{Numerical comparisons}
\label{subsec:num_comp}

In this section, we analyse how well our method performs in recovering extreme directions in contrast to DAMEX (\cite{goix2017sparse}), and its efficiency for estimating normalised columns $A_{\cdot k} / ||A_{\cdot k}|| =: \textbf{a}_k$, $k \in [K]$ compare to sKmeans (\cite{janseen2020clustering}). Other algorithms (namely, \cite{chiapino2019identifying, fomichov2022spherical, meyer2023multivariate}) in the literature were also considered; however, since they suffer from computational weaknesses or yield poor performance where $\mathbf{X}$ is decomposed as a linear factor model described in \eqref{eq:lfm} under Conditions \ref{cond:(i)}-\ref{cond:(ii)}, they are omitted from the presentation of the results. We borrow the identical setup of a moving-maxima process as in Section \ref{subsec:num_res}, we hence moved to an elucidation of the target values. 

\paragraph{Target values.} Our simulation study aims to investigate the performance of our algorithm in determining the number of extreme directons and assessing its performance, when $\hat{K} = K$, using the TFPP and TFNP metrics compared to the DAMEX Algorithm. Additionally, we assess the disparity between the true centroids $\textbf{a}_1,\dots,\textbf{a}_K$ and the estimated centroids $\hat{\textbf{a}}_1,\dots,\hat{\textbf{a}}_K$ as produced by both our procedure and sKmeans when $\hat{K} = K$ by:
\begin{equation}
	\label{eq:metric_comp_skmeans}
	D(\{\textbf{a}_1,\dots,\textbf{a}_K \}, \{ \hat{\textbf{a}}_1,\dots, \hat{\textbf{a}}_K \}) = \underset{\pi}{\min}\, \sqrt{\sum_{k=1}^K ||\hat{\textbf{a}}_{\pi(k)} - \textbf{a}_k ||_2^2},
\end{equation}
where the min is taken over all permutation $\pi$ of $\{1,\dots,K\}$. By the definition of \eqref{eq:metric_comp_skmeans}, the number of factors in the experiment is reduced to $K = 6$ due to memory limitations.

\paragraph{Results and Discussion.} Figure \ref{fig:num_comp}, Panels \subref{fig:err_damex}-\subref{fig:tfpp_damex} depict results on exact recovery rate of the number of clusters $K$, TFNP and TFPP for both Algorithms \ref{alg:clust} and DAMEX over 50 simulations. The exact recovery rate of Algorithm of \ref{alg:clust} is always better than the one of DAMEX for any configurations of $d$ and $n$. 
Moreover, our procedure appears to be more resilient than DAMEX to a decrease in the sample size $n$ and an increase in the dimension $d$. However, when the number of latent factors is correctly retrieved by DAMEX Algorithm, it exhibits a better performance than \ref{alg:clust} in terms of TFNP and TFPP. The average value of $D(\{\textbf{a}_1,\dots,\textbf{a}_K \}, \{ \hat{\textbf{a}}_1,\dots, \hat{\textbf{a}}_K \})$ over 50 realisations for different values of $d$ and $n$ are show in Figure \ref{fig:num_comp}, Panel \subref{fig:centroids_skmeans}. It can be seen that the locations of the points of mass of the spectral measure are most precisely estimated by \ref{alg:clust} Algorithm.

\begin{figure}[ht]
    \centering
    \begin{subfigure}[t]{0.48\linewidth}
        \includegraphics[width=\textwidth, height=0.25\textheight]{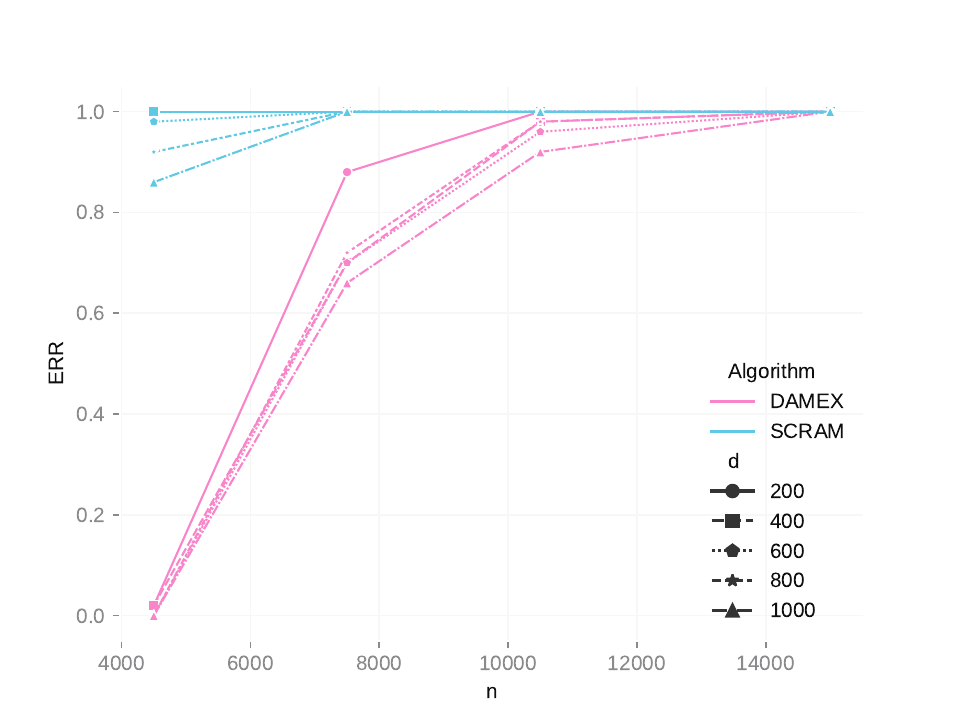}
        \caption{Exact Recovery Rate}
        \label{fig:err_damex}
    \end{subfigure}
    \hfill
    \begin{subfigure}[t]{0.48\linewidth}
        \includegraphics[width=\textwidth, height=0.25\textheight]{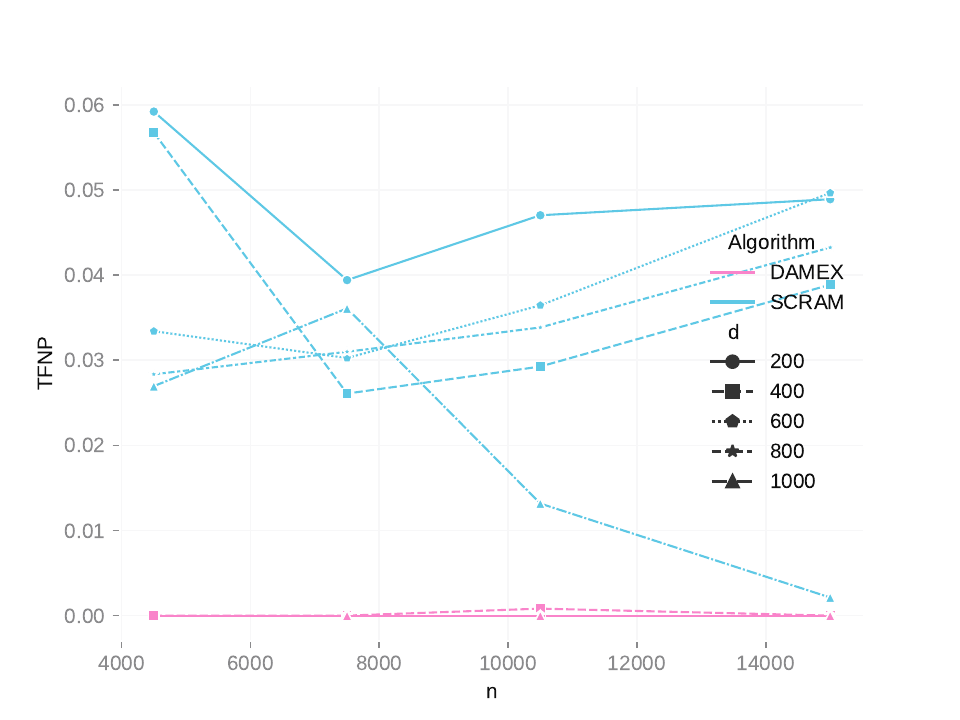}
        \caption{Total False Negative Proportion}
        \label{fig:tfnp_damex}
    \end{subfigure}
    
    \vspace{1em} 

    \begin{subfigure}[t]{0.48\linewidth}
        \includegraphics[width=\textwidth, height=0.25\textheight]{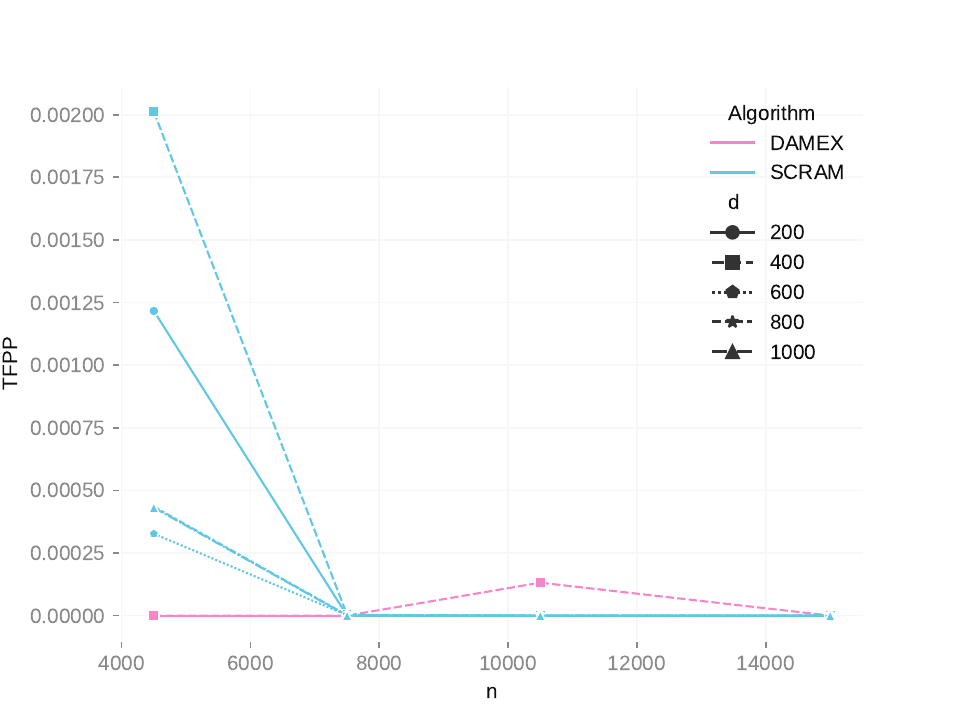}
        \caption{Total False Positive Proportion}
        \label{fig:tfpp_damex}
    \end{subfigure}
    \hfill
    \begin{subfigure}[t]{0.48\linewidth}
        \includegraphics[width=\textwidth, height=0.25\textheight]{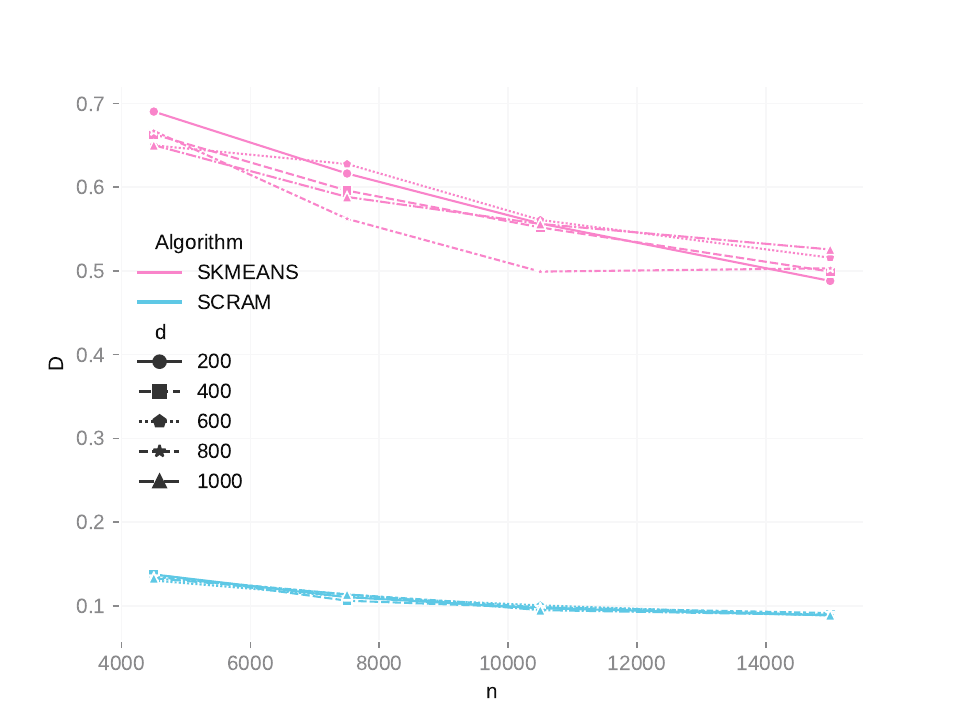}
        \caption{$D(\{\textbf{a}_1,\dots,\textbf{a}_K \}, \{ \hat{\textbf{a}}_1,\dots, \hat{\textbf{a}}_K \})$}
        \label{fig:centroids_skmeans}
    \end{subfigure}
    
    \caption{Results of numerical comparisons between \ref{alg:clust} Algorithm (in blue) and DAMEX Algorithm (in pink) in Panels \subref{fig:err_damex}-\subref{fig:tfpp_damex}. Comparison of $D(\{\textbf{a}_1,\dots,\textbf{a}_K \}, \{ \hat{\textbf{a}}_1,\dots, \hat{\textbf{a}}_K \})$ between \ref{alg:clust} Algorithm (in blue) and sKmeans (in pink) are given in Panel \subref{fig:centroids_skmeans}.}
    \label{fig:num_comp}
\end{figure}

\paragraph{Calibrating parameters.} In \ref{alg:clust}, the parameter $\delta$ is selected using the method proposed in Section \ref{subsec:data_driven}. To fasten computations, we utilise the simulation outcomes detailed in Section \ref{subsec:num_res}, selecting $\delta^\star$ as the threshold that minimises the average value of the criteria across $50$ iterations given in Table \ref{tab:values_k_d}. In the DAMEX Algorithm, we adopt the approach recommended by the authors, selecting the $\lfloor \sqrt{n} \rfloor$ largest values. Due to the propensity of the DAMEX Algorithm to return numerous extreme directions for many $\epsilon$ values (tuning parameter of the DAMEX Algorithm), we opt to merge overlapping directions. This adjustement is crucial for enabling the DAMEX Algorithm to accurately recover the true number of latent factors. Furthermore, we determine $\epsilon$ through trial and error using the exact recovery rate of $K$ (post-merging step) as a benchmark, and settle on $\epsilon =0.3$. Calibrating parameters of the sKmeans algorithm is relatively straightforward. We simply select the $\lfloor \sqrt{n} \rfloor$ largest values and we designate the true number of latent factors $K=6$ (which is typically unknown in practice) as the number of clusters.

\section{Applications}
\label{sec:applications}

\subsection{Extreme precipitations in France}
\label{sec:extreme_precip}
\begin{figure}[ht]
    \centering
    \begin{subfigure}[t]{0.48\linewidth}
        \includegraphics[width=\textwidth]{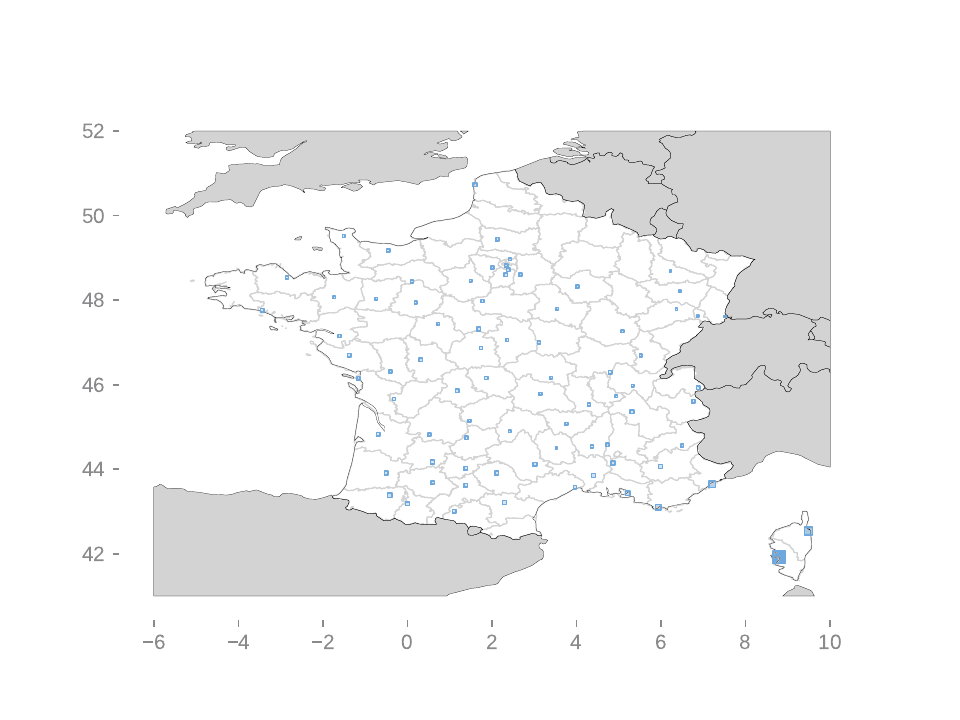}
        \caption{}
        \label{subfig:cluster_corsica}
    \end{subfigure}
    \hspace{0.02\linewidth}
    \begin{subfigure}[t]{0.48\linewidth}
         \includegraphics[width=\textwidth]{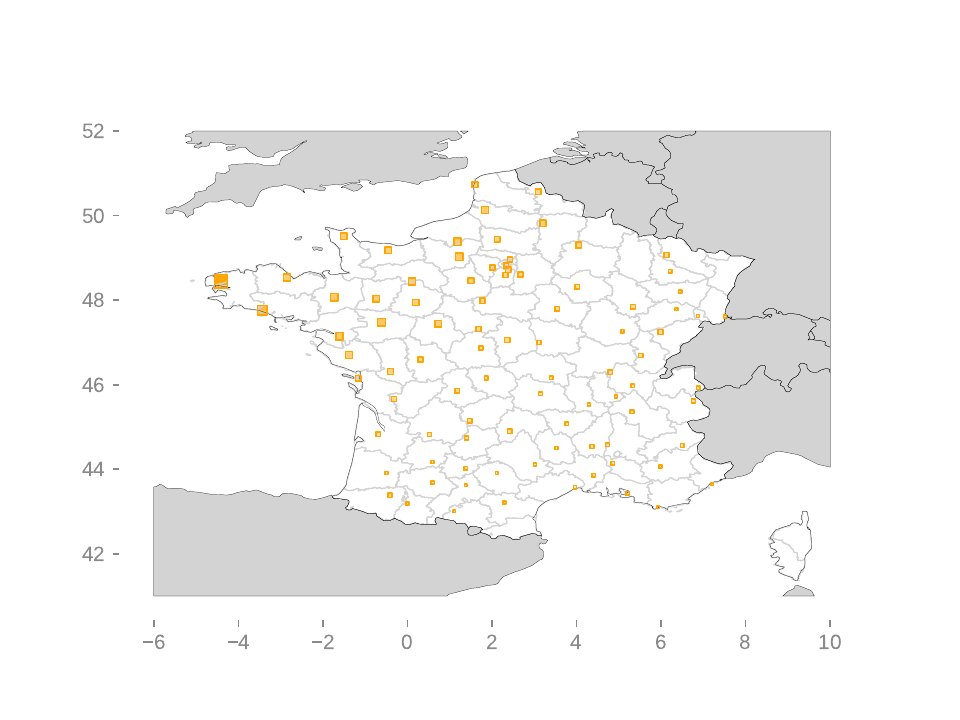}
        \caption{}
        \label{subfig:cluster_west}
    \end{subfigure}
    
    \vspace{0.5em} 

    \begin{subfigure}[t]{0.48\linewidth}
       \includegraphics[width=\textwidth]{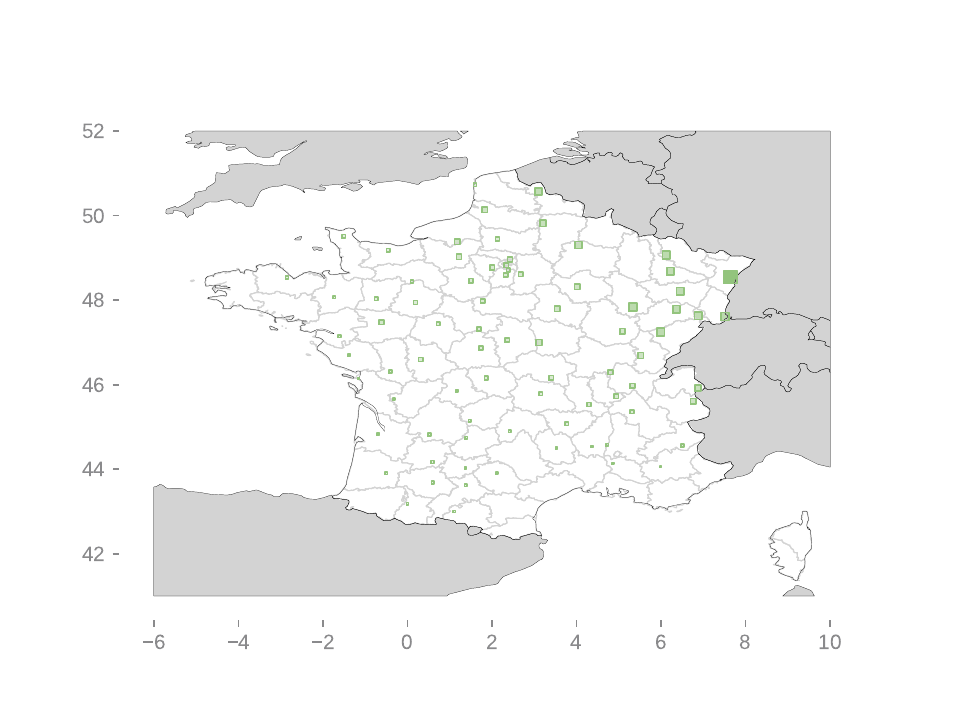}
        \caption{}
        \label{subfig:cluster_east}
    \end{subfigure}
    \hspace{0.02\linewidth}
    \begin{subfigure}[t]{0.48\linewidth}
        \includegraphics[width=\textwidth]{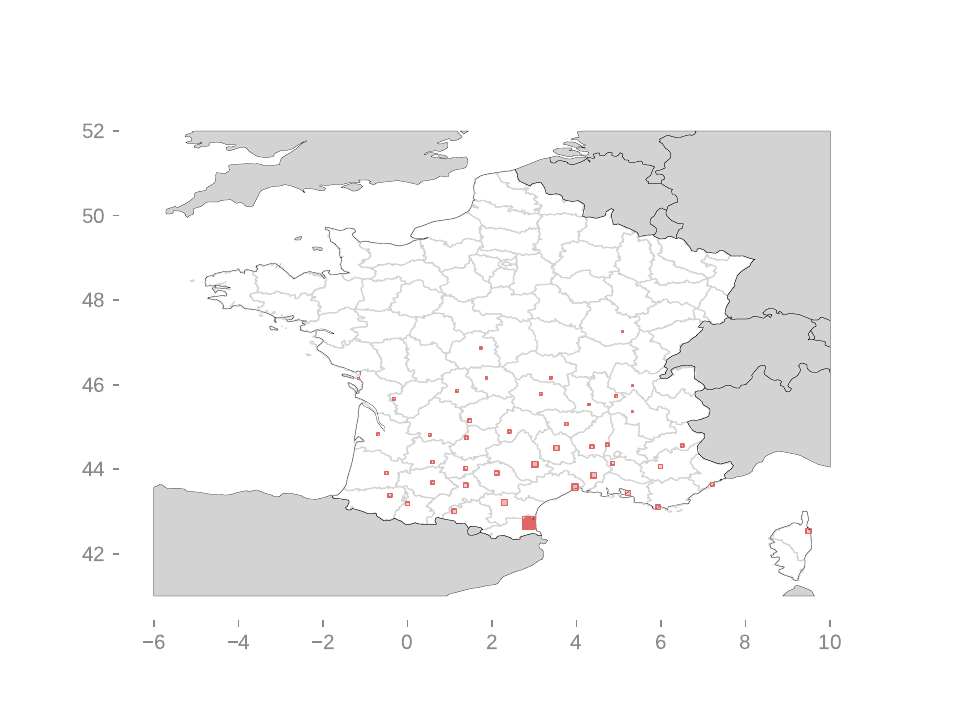}
        \caption{}
        \label{subfig:cluster_south}
    \end{subfigure}
    
    \caption{In Panel \subref{subfig:cluster_corsica}, we show the clusters for the Corsica latent variable. Panel \subref{subfig:cluster_west} displays clusters for the western latent variable. Panel \subref{subfig:cluster_east} shows clusters for the eastern latent variable. Finally, Panel \subref{subfig:cluster_south} depicts clusters for the southern latent variable. The strength of association is indicated by the size and color intensity of each square.}
    \label{fig:spatial_clusters}
\end{figure}

In our analysis, we focus on weekly maxima of hourly precipitation recorded at $92$ weather stations in France during the fall season, spanning from September to November, for the years 1993 to 2011, resulting in $228$ block maxima. This dataset was provided by Météo-France and has been previously used in \cite{bernard2013clustering}. The selection of stations was based on their data quality and ensuring a relatively uniform coverage of France.

We use the process described in Section \ref{subsec:data_driven} to choose the tuning parameter $\delta$. The entire process recommends using $c_\ell \approx 0.82$ as the most suitable threshold value for our analysis. Employing the designated threshold, we unveil four latent variables situated in the western, eastern, southern regions of metropolitan France and Corsica. It is crucial to highlight that our process operates solely based on rainfall records, devoid of any geographical information. Consequently, discerning consistent spatial structures from just rainfall measurements is not a straightforward outcome. Spatial representation of clusters are depicted in Figure \ref{fig:spatial_clusters}. The Corsican cluster highlighted in Figure \ref{fig:spatial_clusters}, Panel \subref{subfig:cluster_corsica}, where the pure variable is located at Ajaccio, exhibits a strong association within the island, while other associations rapidly decline on the mainland of France. The western area above Bordeaux, indicated in Figure \ref{fig:spatial_clusters} Panel \subref{subfig:cluster_west}, exhibits robust dependencies with the central region around Paris. However, beyond these regions, the associations with the latent variable rapidly decrease. Symmetrically, the eastern region, spanning from Lyon and covering the Vosges mountains, Alsace, the Franche-Comté and regions in northeastern France, depicted in Figure \ref{fig:spatial_clusters} Panel \subref{subfig:cluster_east}, displays dependencies with the central regions while diminishing rapidly outside this area. In contrast, the western cluster shows a broader distribution spanning accross the entire country. The southern cluster, in Figure \ref{fig:spatial_clusters} Panel \subref{subfig:cluster_south}, showcases spatial dependencies over Corsica and Mediterranean cities. These associations rapidly fell-off, resulting in the formation of a less spread-out cluster. The clustering results for locations align quite close with \cite{bernard2013clustering, maume2023detecting} dividing France into north and south regions. The key distinction lies in our clusters being overlapping, providing a more nuanced understanding of the variability of each location's affiliation to a cluster. It is noteworthy that the farther a location is from the pure variable, the lesser the corresponding affiliation.

Except for the corsican cluster, the interpretation of the resulting clusters seems straightforward. The extreme rainfall in northern France can be attributed to disturbances originating from the Atlantic, impacting regions like Brittany, Paris, and other northern areas. In the southern regions of France, particularly during the fall, intense rainfall events typically arise from southern winds compelling warm and moist air to interact with the mountainous terrain of the Pyrénées, Cévennes, and Alps. This interaction often leads to the development of severe thunderstorms. While these events can be quite localised, they frequently impact a substantial portion of the Mediterranean coastal area. In the eastern regions, despite the presence of various microclimates, the Vosges moutains serve a delineation between the temperate oceanic climate in the western part and the continental climate in the eastern part, particularly in the Upper Rhine Bassin.

\begin{figure}
    \centering
    \includegraphics[scale=0.5]{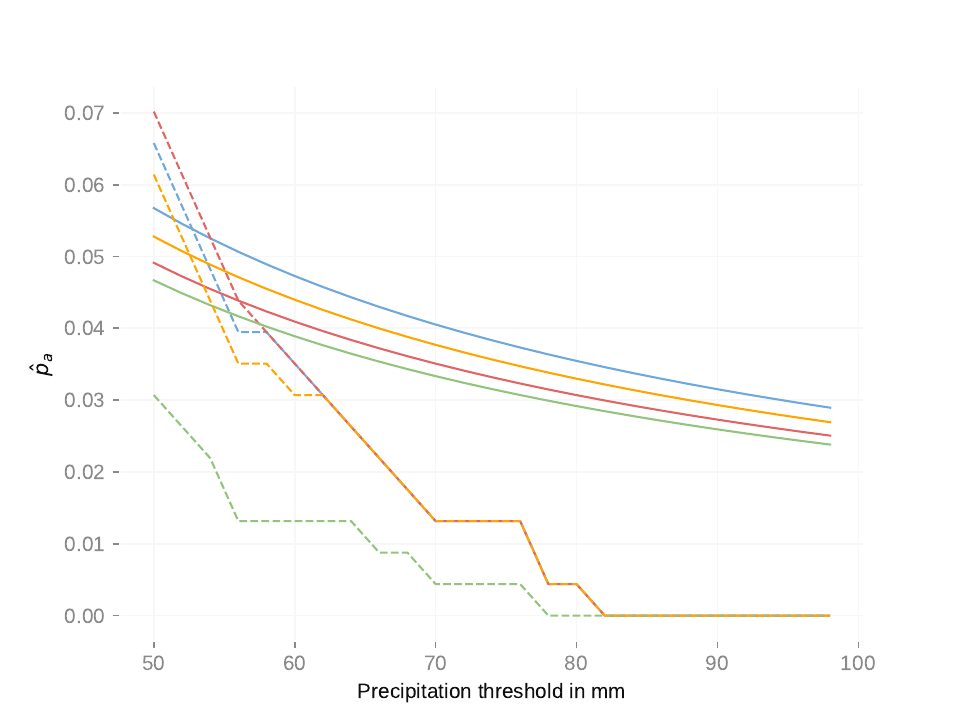}
    \caption{Approximations of $\hat{p}_a$ concerning precipitation in millimeters are illustrated through straight lines in blue, green, red and yellow representing the corsican, western, eastern and southern clusters, respectively. Empirical estimations are portrayed with dashed lines, mirroring the color code for their respective clusters.}
    \label{fig:pmax}
\end{figure}

Since the exponent measure of the linear factor model $\mathbf{X}$ in equation \eqref{eq:lfm} is discrete, calculating probabilities of extreme events, denoted as $\mathbb{P}\{\mathbf{X} \in C \}$ for a set of interest $C$, becomes a straightforward task. In our environmental dataset, determining regions and probabilities as
\begin{equation*}
    C_a(\mathbf{X}) = \cup_{j \in G_a} \{y^{(j)} > x^{(j)} \}, \, p_a(\mathbf{X}) = \mathbb{P}\{ \mathbf{X} \in C_a(\mathbf{X})\}, \, a \in \{1,2,3\}
\end{equation*}
is a common approach, especially when an extreme event at any location could potentially result in a climatological catastrophe. Letting $\hat{A}_{j\ell}$ be the element of the estimated $\hat{A}$, one can show that
\begin{equation*}
    \hat{p}_a(\mathbf{X}) = \sum_{\ell = 1}^3 \underset{j \in G_a}{\max} \, \frac{\hat{A}_{j\ell}}{x^{(j)}}, \, a \in \{1,2,3\}.
\end{equation*}

In Figure \ref{fig:pmax}, we illustrate $\hat{p}_a(\mathbf{X})$ where $\mathbf{X}$ is selected within the range of $50$ to $100$ mm for precipitation. The obtained estimation are of the same magnitude of those of \cite[Appendix B]{kiriliouk2022estimating} for Switzerland. A comparison with empirical estimates reveals that the latter tends to underestimate the probability of heavy rainfall events. Moreover, the linear factor model exhibits the capability to extrapolate, maintaining informative values even as the empirical estimates plummet to zero, losing their informativeness.

\subsection{Wildfires in French Mediterranean}
\label{subsec:wildfires}

\begin{figure}\centering
\subfloat[]{\label{subfig:cluster_0_wildfires}\includegraphics[width=.45\linewidth, height=0.3\textheight]{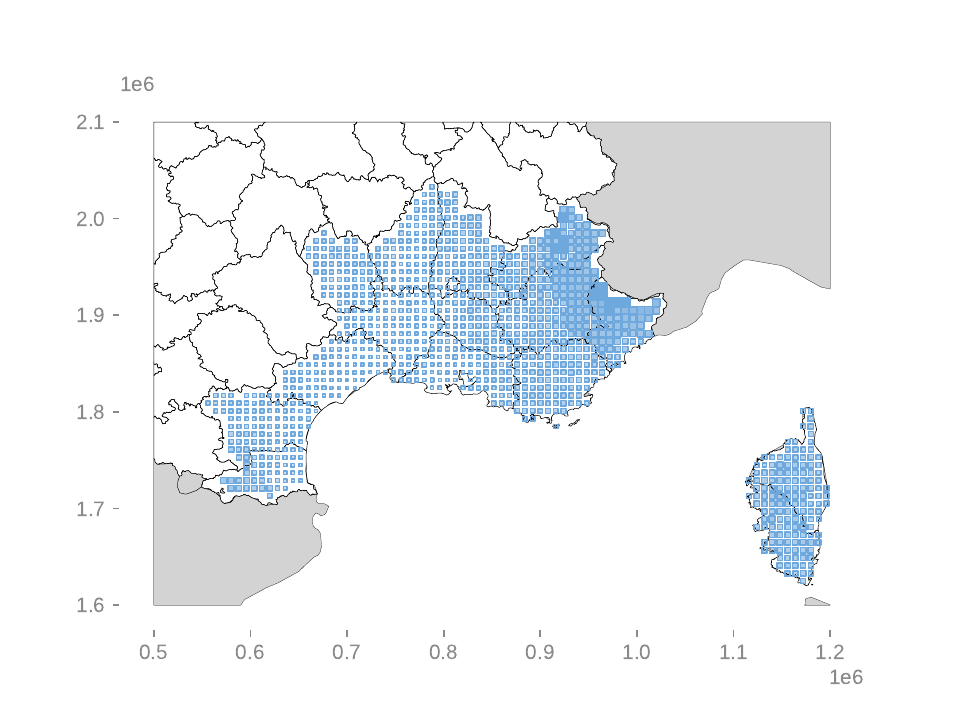}}\hfill 
\subfloat[]{\label{subfig:cluster_1_wildfires}\includegraphics[width=.45\linewidth, height=0.3\textheight]{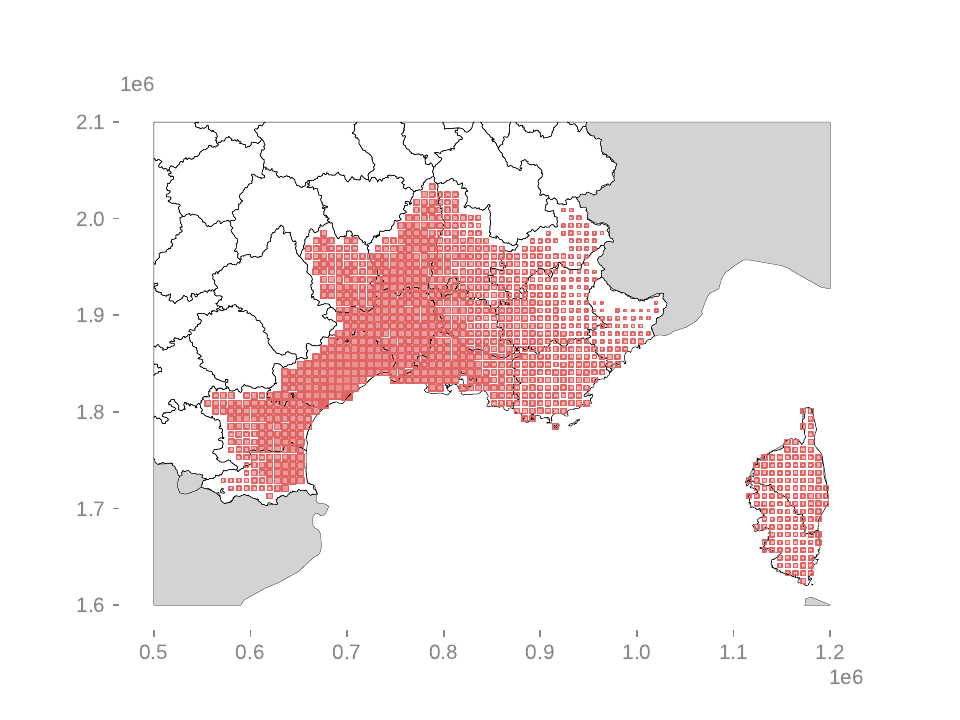}}
\caption{In Panel \subref{subfig:cluster_0_wildfires}, we show the spatial cluster for the first latent variable. Panel \subref{subfig:cluster_1_wildfires} displays the spatial cluster for the second latent variable. The strength of association is indicated by the size and color intensity of each square.}
\label{fig:wildfires}
\end{figure}

Our case study focuses on the southeastern part of France, covering an area of 80500 km². This region, prone to wildfires, encompasses a broad range of bioclimatic, environmental, and anthropogenic gradients. Approximately $60\%$ of study area consists of easily ignitable forested areas or vegetation types, such as shrubland and other natural herbaceous vegetation. Wildfires face challenges in spreading through the various available cover types. The observation period for this study is 1995-2018, specifically during the extended summer months (June-October). Gridded weather reanalysis data from the SAFRAN model of Météo-France, with an $8$km resolution, is utilised for analysis. This dataset has also undergone extensive examination in \cite{koh2023spatio}, from which we obtain the data.

Understanding the joint impact of variables like temperature, precipitation, and wind speed on fire activity patterns is highly intricate. Various meteorological indices on fire activity patterns have been developed, including the widely used unitless Fire Weather Index (FWI), originally designed for Canadian forests. Typically, FWI values are directly interpreted and used for fire danger mapping. However, our approach involves studying its spatial variability through a linear factor model. In our methodology, we extract monthly maxima of FWI during the extended summer months over the $1143$ pixels, resulting in $100$ observations. Through a data-driven approach to select the threshold, as explained in Section \ref{subsec:data_driven}, we choose $\delta^\star \approx 0.1765$ to obtain two latent factors (see Figure \ref{fig:wildfires}. These factors are directly interpretable in terms of elevation (refer to Figure \ref{fig:wildfires} Panel \subref{subfig:cluster_0_wildfires} and Figure \ref{fig:wildfires} Panel \subref{subfig:cluster_1_wildfires}). Indeed, the first cluster demonstrates a strong association within mountainous areas (Western Alps, Corsica, and Pyrénées), while the second exhibits associations within lowlands prone to fire activity (\cite{frejaville2015spatiotemporal}) and heatwaves (\cite{ruffault2016objective}), mid-elevation hinterlands, and foothills.

\section{Discussion}

We have introduced a comprehensive methodology for estimating the parameters of a discrete spectral measure of a max-stable distribution. Our approach lies into model-based clustering and proves to be both rapid and convenient. Additionally, we have provided statistical assurances for our method, ensuring favorable outcomes even in high dimensions where the relationship between the dimensionality, represented by $d$, and the sample size, denoted as $n$, may vary and potentially result in larger $d$ values. These results are robust, grounded in general conditions that span a diverse array of applications. Our methodology however does not accommodate multivariate regularly varying distributions with an unknown tail index $\alpha$. Instead of the considered framework, let $\textbf{Z}$ consists of independent regularly varying random variables with a known tail index $\alpha$, then Theorem \ref{thm:matrix_product} can be reformulated as
\begin{equation*}
    \mathcal{X} = \bar{A} \odot \bar{A}^\top,
\end{equation*}
where $\bar{A}$ represents the standardised loading matrix of $\mathbf{X}$, that is
\begin{equation*}
    \bar{A} = (\bar{A}_{ja})_{d \times K} = \left( \frac{A_{ja}^\alpha}{\sum_{a=1}^K A_{ja}^\alpha} \right)_{j =1\dots,d, a = 1,\dots,K}.
\end{equation*}
Then, the entire proposed procure applies to $\bar{A}$ as long as the pure variable condition is satisfied, namely Condition \ref{cond:(ii)}. Setting the scaling condition $\sum_{a=1}^K A_{ja}^\alpha = 1$ for any $j = 1,\dots,d$, we can, however estimate the matrix $A$ with the further addition of the estimation of the tail index $\alpha$, demanding more intricate technical details within our non-asymptotic framework with weakly dependent observations. This aspect remain of significant interest for applications. Indeed, since marginal standardisation does not impact dependence modeling, it is a common practice to separate marginal and dependence modeling. This involves standardising each marginal to a common distribution and then focusing solely on modeling the dependence structure. However, in certain applications where the goal is to estimate failure regions of the form $\{ \sum_{j=1}^d v^{(j)} X^{(j)} > x \}$ with $x$ being large and $\sum_{j=1}^d v^{(j)} = 1$, $v^{(j)} >0$, such an approach may be suboptimal. In these cases, it is necessary to directly model the original vector $\mathbf{X}$. These failure regions are particularly relevant in climate applications, as demonstrated by \cite{kiriliouk2020climate} and \cite{kiriliouk2022estimating}.

One can also consider the contaminated linear factor model
\begin{equation*}
    \mathbf{X} = A \textbf{Z} + \sigma \eta,
\end{equation*}
where $A \in \mathbb{R}^{d \times K}$ satisfies Condition \ref{cond:(i)}, $\textbf{Z} = (Z^{(1)},\dots, Z^{(K)})$ is a $K$-dimensional vector with i.i.d. standard Fréchet distributed components, $\sigma > 0$ regulates the signal-to-noise ratio and $\eta$ is a common factor noise distributed as a standard Fréchet. The standardised loading matrix $\bar{A}$ of $\mathbf{X}$ is expressed as:
\begin{equation*}
    \bar{A} = \left( \frac{A_{ja} + \sigma}{1+\sigma} \right)_{j =1\dots,d, a = 1,\dots,K}.
\end{equation*}
The current challenge is that Theorem \ref{thm:identifiability_I} \ref{item_i:identifiability_I} no longer applies since
\begin{equation*}
    \chi(i,j) = \frac{\sigma}{1+\sigma} \neq 0, \quad i \in I_a, j \in I_b, b \neq a.
\end{equation*}
So, we can no longer recover latent factors using pairwise asymptotic independence obtained from the proposed model in \eqref{eq:lfm}. However, taking $a \in [K]$ with $|I_a| \geq 2$, it is readily verified that
\begin{equation*}
    \chi(i,k) = \frac{A_{ka} + \sigma}{1+\sigma} < 1 = \chi(i,j)
\end{equation*}
for any $k \notin I_a$ and $i,j \in I_a$. Thus, a procedure to identify $[K]$ is possible using the more stringent condition
\begin{Assumption}{(ii'')}
	\label{cond:(ii'')}
    For any $a \in \{1,\dots,K\}$, there exist at least two indices $j \in \{1,\dots,d\}$ such that $A_{ja} =1$ and $A_{jb} = 0$, $\forall b \neq a$.
\end{Assumption}
Without knowledge of $\sigma$, the matrix $A$ can be recovered up to a multiplication constant (and by multiplication by a permutation matrix). It is also crucial to emphasise that such Condition \ref{cond:(ii'')} paves the way to reduce the computational complexity of our procedure method.

A possible extension of the methodology is matrix-valued data, which has become increasingly prevalent in many applications. Most existing clustering methods for this type of data are tailored to the mean model and do not account for the (extremal) dependence structure of the variables. To extract information from the extremal dependence structure for clustering, we can propose a new latent variable model for the variables arranged in matrix form, with some unknown loading matrices representing the clusters for rows and columns.

Drawing an analogy with the linear factor model studied in this paper, assume the variables are stacked as a random matrix $X \in \mathbb{R}^{p \times q}$ which follows the decomposition
\begin{equation*}
	X = A Z B^\top + E,
\end{equation*}
where $Z \in \mathbb{R}^{K_1 \times K_2}$ is a latent variable matrix which is regularly varying, meaning that there exist a scaling sequence $\{c_n\}$ and a measure $\Lambda_Z$ on $\mathcal{M}_{K_1, K_2}(\mathbb{R}_+)$ such that the following vague convergence holds:
\begin{equation*}
	n \mathbb{P}\left\{ c_n^{-1} Z \in \cdot \right\} \overunderset{v}{n \rightarrow \infty}{\longrightarrow} \Lambda_Z(\cdot).
\end{equation*}
$A \in \mathbb{R}^{p \times K_1}$ and $B \in \mathbb{R}^{q \times K_2}$ are the unknown loading matrices for the rows and the columns, respectively. $E \in \mathbb{R}^{p \times q}$ represents the random noise matrix with entries having lighter tails.

\section*{Acknowledgements}
This work has been supported by the project ANR McLaren (ANR-20-CE23-0011). Financial support by the Deutsche Forschungsgemeinschaft (DFG, German Research Foundation; Project-ID 520388526; TRR 391: Spatio-temporal Statistics for the Transition of Energy and Transport) is gratefully acknowledged.

\section*{Data Availability Statement}

Our code for generating the numerical results and applications can be accessed via the following link \href{https://github.com/Aleboul/Linear_factor_models}{https://github.com/Aleboul/Linear\_factor\_models}.

\printbibliography

\appendix

\section{Investigation into the computation time of clique algorithm}
\label{sec:computation_time}

\begin{figure}
    \centering
    \includegraphics[width=0.5\textwidth, height=0.3\textheight]{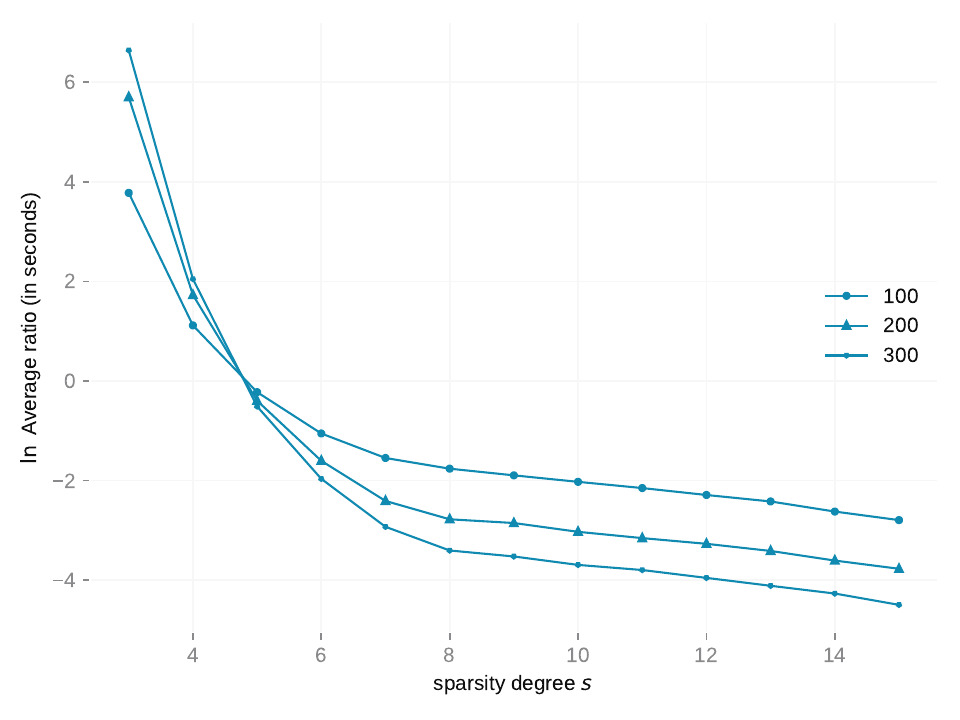}
    \caption{Average ratio $T_{BK} / T_{MILP}$ in seconds set to the log scale with respect to the sparsity degree $s \in \{2,3,\dots,15\}$ and $d \in \{100,200,300\}$.}
    \label{fig:computation time}
\end{figure}

In this section, we explore the computation time required to identify a clique using the extremal correlation matrix, irrespective of whether the matrix is sparse or not. To achieve this, we examine the approach outlined in the main paper, which involves the following binary problem,
\begin{equation*}
\begin{aligned}
\max_{x^{(i)}} \quad & \sum_{i=1}^d x^{(i)}\\
\textrm{s.t.} \quad & x^{(i)} + x^{(j)} \leq 1, \quad \forall (i,j) \in \bar{E}\\
  &x^{(i)} \in \{0,1\}, \quad i = 1,\dots,d,    \\
\end{aligned}
\end{equation*}
and the Bron-Kerbosh algorithm (\cite{bron1973algorithm}). The matrix $A$ is constructed as follows: we designate the initial rows to comprise the first $20$ pure variables. For generating $A_J$, where $j \in J$, we randomly choose the support from the set ${1,2,\dots,20}$ with a sparsity $s \in \{2,3,\dots,15\}$. Subsequently, we form the extremal correlation matrix $\mathcal{X} = A \odot A^{\top}$ and investigate a clique using the two aforementioned methods in $20$ replications. We examine three scenarios with varying dimensions, denoted as $d \in \{100,200,300\}$. We denote the time spent recovering the clique through the adjacency matrix $E$ computed with the extremal correlation matrix $\mathcal{X}$ as $T_{BK}$ for the Bron-Kerbosch algorithm and $T_{MILP}$ for the binary problem. The results are illustrated in Figure \ref{fig:computation time}. For a concise interpretation of the numerical results, when $d = 300$ and $s = 3$, the binary problem is 768 times faster than the Bron-Kerbosch algorithm, and conversely, the Bron-Kerbosch algorithm is 100 times faster when $s = 15$.

As anticipated, when the sparsity $s$ is low ($s < 4$), the binary problem proves to be the most effective in recovering the maximum clique, whereas the Bron-Kerbosch algorithm exhibits superior performance for a sparsity index $s \geq 4$. Regardless of the dimension, the (log) average ratio is decreasing and shows a rapid deceleration when $s \geq 4$. The contrast between the two methods becomes more pronounced with increasing considered dimensions.

\section{Algorithm}
\label{subsec:algorithms}

We give below the specifics of Algorithm \ref{alg:pure_variable} motivated in Section \eqref{sec:estimation_chap_4}, the Algorithm \ref{alg:htsp} and summarize our final algorithm in Algorithm \ref{alg:clust} (Soft Clustering lineaR fActor Model).

\begin{algorithm}
    \renewcommand{\thealgorithm}{(PureVar)}
	
	\caption{}

\begin{algorithmic}[1]
\Procedure{PureVar}{$\widehat{\mathcal{X}}$,$\delta$}
    \State Initialize: $\mathcal{I} = \emptyset$
    \State Construct the graph $G = (V,E)$ where $V = [d]$ and $(i,j) \in E$ if $\hat{\chi}_{n,m}(i,j) \leq \delta$
    \State Find a maximum clique, $\bar{\mathcal{G}}$, of $G$
    \For{$i \in \bar{\mathcal{G}}$}
        \State $\hat{I}^{(i)} = \{ j \in [d] \setminus \{i\} \, : 1-\hat{\chi}_{n,m}(i,j) \leq \delta \}$
        \State $\hat{I}^{(i)} = \hat{I}^{(i)}\cup\{i\}$
        \State $\hat{\mathcal{I}} = \textbf{MERGE}(\hat{I}^{(i)}, \hat{\mathcal{I}})$
    \EndFor
    \State Return $\hat{\mathcal{I}}$ and $\hat{K}$ as the number of sets in $\hat{\mathcal{I}}$
\EndProcedure
\end{algorithmic}
\label{alg:pure_variable}
\end{algorithm}

\begin{algorithm}
    \renewcommand{\thealgorithm}{(MERGE)}
	
	\caption{}

\begin{algorithmic}[1]
\Procedure{MERGE}{$\hat{I}^{(i)}$,$\hat{\mathcal{I}}$}
    \For{$G \in \hat{\mathcal{I}}$}
        \If{$G \cap \hat{I}^{(i)} \neq \emptyset$}
			\State $G = G \cap \hat{I}^{(i)}$
            \State Return $\hat{\mathcal{I}}$
    	\EndIf
    \EndFor
    \State $\hat{I}^{(i)} \in \hat{\mathcal{I}}$
    \State Return $\hat{\mathcal{I}}$
\EndProcedure
\end{algorithmic}
\label{alg:merge}
\end{algorithm}

\begin{algorithm}
    \renewcommand{\thealgorithm}{(HTSP)}
	
	\caption{}

\begin{algorithmic}[1]
\Procedure{HTSP}{$\hat{\mathcal{X}}$,$\delta$, $\hat{I}$}
	\For {$ j \in [d] \setminus \hat{I}$}
		\State $\bar{\chi}^{(j)} = \left( \frac{1}{|\hat{I}_a|} \sum_{i \in \hat{I}_a} \hat{\chi}_{n,m}(i,j) \right)_{a = 1,\dots,\hat{K}}$
		\State $\bar{\beta}^{(j)} = \left(\bar{\chi}^{(j)}_a \mathds{1}_{ \{ \bar{\chi}^{(j)}_a > \delta \} } \right)_{a = 1,\dots,\hat{K}}$
    	\State $\widehat{\mathcal{S}} = \textrm{supp}(\bar{\beta}^{(j)})$
    	\State $\left.\hat{\beta}^{(j)}\right|_{\widehat{\mathcal{S}}} = \mathcal{P}_{\Delta_{\hat{K}-1}}(\left.\bar{\beta}^{(j)}\right|_{\widehat{\mathcal{S}}})$, $\left.\hat{\beta}^{(j)}\right|_{\widehat{\mathcal{S}}^c} = 0$
    \EndFor
\EndProcedure
\end{algorithmic}
\label{alg:htsp}
\end{algorithm}

\begin{algorithm}
    \renewcommand{\thealgorithm}{(SCRAM)}
	
	\caption{}

\begin{algorithmic}[1]
\Procedure{SCRAM}{$\hat{\mathcal{X}}$, the tuning parameter $\delta$}
    \State Apply Algorithm \ref{alg:pure_variable} to obtain the number of clusters $\hat{K}$, the estimated set of pure variables $\hat{I}$ and its partition of $\hat{\mathcal{I}}$.
    \State Estimate $A_I$ by $\hat{A}_{\hat{I}}$ from \eqref{eq:estimator_A_I}.
    \State Estimate $A_J$ by $\hat{A}_{\hat{J}}$ applying Algorithm \ref{alg:htsp}. Combine $\hat{A}_{\hat{I}}$ with $\hat{A}_{\hat{J}}$ to obtain $\hat{A}$.
    \State Estimate fuzzy clusters $\hat{\mathcal{G}} = \{ \hat{G}_1, \dots,\hat{G}_K\}$ from \eqref{eq:estimator_over_clust} by using $\hat{A}$.
    \State Output $\hat{A}$ and $\hat{\mathcal{G}}$.
\EndProcedure
\end{algorithmic}
\label{alg:clust}
\end{algorithm}
\section{Proofs of Section \ref{sec:identifiability}}
\begin{proof}[Proof of Theorem \ref{thm:matrix_product}]
    Let $i,j \in \{1,\dots,d\}$ be arbitrary with $i \neq j$. Define $Y^{(i)} = \sum_{a = 1}^K A_{ia} Z^{(a)}$ and $Y^{(j)} = \sum_{a=1}^K A_{ja}Z^{(a)}$. Note that $\textbf{Y}$ and $\textbf{X}$ have the same exponent measure since they differ only by a sum of a random variable with a lighter tail (see Lemma \ref{lem:tail_balance_condition}). So we only have to compute bivariate extremal correlations for $\textbf{Y}$ to obtain those of $\textbf{X}$. In order to obtain bivariate regular variation of $\textbf{Y}^{(i,j)} = (Y^{(i)}, Y^{(j)})$, consider the map $\psi$ from $\mathbb{R}_+^{K} \rightarrow [0,\infty)^2$ defined by
    \begin{equation*}
        \psi(z^{(1)}, \dots, z^{(K)}) = \left(\sum_{a=1}^K A_{ia} z^{(a)}, \sum_{a=1}^K A_{ja} z^{(a)}\right).
    \end{equation*}
    For a measurable subset $A$ of $\mathbb{R}^2$, separated from $0$, we obtain by corollary 2.1.14 of \cite{kulik2020heavy}:
    \begin{align*}
        \Lambda_{\textbf{Y}^{(i,j)}}(A) = \Lambda_\textbf{Z} \circ \psi^{-1}(A) &= \sum_{a = 1}^K \delta_0 \otimes \dots \otimes \Lambda_{Z^{(a)}} \otimes \dots \otimes \delta_0 \circ \psi^{-1}(A) \\
        &= \sum_{a=1}^K \int_{0}^{\infty} \mathds{1}_A(A_{ia}s, A_{ja}s) s^{-2} ds.
    \end{align*}
    Applying to $A = (1,\infty) \times (1,\infty)$ and $A = (1,\infty) \times \mathbb{R}_+$, we get respectively
    \begin{align*}
        \Lambda_{\textbf{Y}^{(i,j)}}((1,\infty) \times (1,\infty)) = \Lambda_{\textbf{Z}} \circ \psi^{-1}(A) &= \sum_{a = 1}^K \int_{0}^{\infty} \mathds{1}_{(1,\infty)\times (1,\infty)}(A_{ia}s, A_{ja}s) s^{-2} ds \\
        &= \sum_{a=1}^K \int_{0}^{\infty} \mathds{1}_{ \{ s > 1/A_{ia}, s > 1 / A_{ja} \} } s^{-2} ds\\
        &= \sum_{a=1}^K \left( \frac{1}{A_{ia}} \vee \frac{1}{A_{ja}} \right)^{-1} = \sum_{a=1}^K (A_{ia} \wedge A_{ja}),
    \end{align*}
    and
    \begin{equation*}
        \Lambda_{\textbf{Y}^{(i,j)}}((1,\infty) \times \mathbb{R}_+) = \sum_{a=1}^K A_{ia} = 1.
    \end{equation*}
    Thus
    \begin{equation*}
        \chi(i,j) = \underset{x \rightarrow \infty}{\lim} \frac{\mathbb{P}\{ Y^{(i)} > x, Y^{(j)} > x\}}{\mathbb{P}\{Y^{(i)} > x\}} = \frac{\Lambda_{\textbf{Y}^{(i,j)}}((1,\infty) \times (1,\infty))}{\Lambda_{\textbf{Y}^{(i,j)}}((1,\infty) \times \mathbb{R}_+)} = \sum_{a=1}^K A_{ia} \wedge A_{ja}.
    \end{equation*}
    \end{proof}
    We know state and prove two lemmata that are crucial for the main results of this section. All results are proved under the condition that model \eqref{eq:lfm} and Conditions \ref{cond:(i)}-\ref{cond:(ii)} hold.
    \begin{lemma}
        \label{lem:identifiability_I_1}
        For any $a \in [K]$, $i \in I_a$ and $|I_a| \geq 2$ we have
        \begin{enumerate}
            \item $\chi(i,j) = 1$ for all $j \in I_a$, \label{item_i_lem:identifiability_I_1}
            \item $\chi(i,j) < 1$ for all $j \notin I_a$. \label{item_ii_lem:identifiability_I_1}
        \end{enumerate}
    \end{lemma}
    \begin{proof}[Proof of Lemma \ref{lem:identifiability_I_1}]
        For any given $i \in \{1,\dots,d\}$, we define the set $s(i) := \{1 \leq a \leq K : A_{ia} \neq 0\}$. For any $i \in I_a$ and $j \neq i$, we have
        \begin{align*}
            \chi(i,j) = \sum_{a \in s(i)} A_{ia} \wedge A_{ja} = A_{ja} \leq 1,
        \end{align*}
        we observe that we have equality in the above display for $j \in I_a$ and strict inequality for $j \notin I_a$ which proves the lemma.
    \end{proof}
    \begin{lemma}
        \label{lem:identifiability_I_2}
        We have $S_i \cup \{i\} = I_a$ and $M_i = 1$ for any $i \in I_a$, with $|I_a| \geq 2$ and $a \in [K]$. 
    \end{lemma}
    \begin{proof}[Proof of Lemma \ref{lem:identifiability_I_2}]
        Lemma \ref{lem:identifiability_I_1} implies that, for any $i \in I_a$, $M_i = 1$ and $S_i = I_a \setminus \{i\}$ which proves the lemma.
    \end{proof}
    \begin{proof}[Proof of Theorem \ref{thm:identifiability_I}]
        \paragraph{Proof of \ref{item_i:identifiability_I}} By condition \ref{cond:(ii)}, for any $a \in [K]$, there exists $i_a \in [d]$ such that $X^{(i_a)} = Z^{(a)} + E^{(i_a)}$. By its very nature under the model \eqref{eq:lfm}, the vector $(X^{(i_1)},\dots,X^{(i_K)})$ is the largest vector being asymptotically independent, i.e.,
        \begin{equation}
            \label{eq:pure_clique}
            \chi(i,j) = 0, \quad \forall i,j \in \{i_1,\dots,i_K\},
        \end{equation}
        see \cite[Proposition 5.24]{resnick2008extreme}. Let us construct the simple undirected graph $G = (V,E)$ with a finite set of vertices $V = [d]$ and a finite set of ordered pairs $(i,j)$ of edges such that $(i,j) \in E$ if $\chi(i,j) = 0$. Through this construction of G, the search for a maximum clique in G is equivalent to searching for a set of indices, denoted as $\{i_1, \dots, i_K\}$, satisfying equation \eqref{eq:pure_clique}. Consequently, we established \ref{item_i:identifiability_I}.
        \paragraph{Proof of \ref{item_ii:identifiability_I}} Consider any $j \in [d]$ with $M_i = 1$ for $i \in I_a$. Since $|I_a| \geq 2$, by Lemma \ref{lem:identifiability_I_1}, the maximum is achieved for any pairs $j,k \in I_a$. However, if $j \notin I_a$, we have $\chi(j,k) < 1$ for all $k \neq j$. Hence $j \in I_a$ and this conclude the proof of the sufficiency part. It remains to prove the necessity part.
        
        Let $i \in I_a$ for some $a \in [K]$ and $j \in I_a \cap S_i$. Since $j \in S_i$ and $|I_a| \geq 2$, we have $\chi(i,j) = M_i = 1$, as a result of Lemma \ref{lem:identifiability_I_2}, which proves \ref{item_ii:identifiability_I}.

        \paragraph{Proof of \ref{item_iii:identifiability_I}} We start with the following construction approach. Let $N = [d]$ be the set of all variables indices and $O = \emptyset$. Let $M_i$ and $S_i$ be defined in \eqref{eq:M_i} and \eqref{eq:S_i}, respectively.
        \begin{enumerate}[label=\textcolor{frenchblue}{\bf(\theenumi)}]
        \item Construct the undirected graph $G = (V,E)$ where $(i,j) \in E$ if $\chi(i,j) = 0$.
        \item Find a maximum clique of $G$ denoted as $\bar{\mathcal{G}}$.
        \item Choose $i \in Q$ and calculate $M_i$ and $S_i$. \label{item:step_3}
        \begin{enumerate}[label=\textcolor{frenchblue}{(\alph*)}]
            \item If $M_i = 1$, set $I^{(i)} = S_i \cup \{i\}$, $O = O \cup \{i\}$ and $Q \setminus \{i\}$.
            \item Otherwise, replace $Q$ by $Q \setminus \{i\}$.
        \end{enumerate}
        \item Repeat Step \ref{item:step_3} until $Q = \emptyset$.
        \end{enumerate}
        We show that $\{ I^{(i)} : i \in O\} = \mathcal{I}$. Let $i \in O$ be arbitrary fixed. By \ref{item_i:identifiability_I} and \ref{item_ii:identifiability_I} of Theorem \ref{thm:identifiability_I}, we have $i \in I$. Thus, there exists $a \in [K]$ such that $i \in I_a$. If $|I_a| \geq 2$, by Lemma \ref{lem:identifiability_I_2}, $i \in I_a$ implies $I_a = S_i \cup \{i\} = I^{(i)}$. On the other hand, let $a \in [K]$ be arbitrary fixed. By condition \ref{cond:(ii)}, there exists, at least one $i \in I_a$. If $|I_a| = 1$, then by \ref{item_i:identifiability_I}, we have $I^{(i)} = I_a$. If $|I_a| \geq 2$ and $j \in I_a$, then $\chi(i,j) = 1$ and $j \in S_i$, once again, by Lemma \ref{lem:identifiability_I_2}, $S_i \cup \{i\} = I_a$, that is $I^{(i)} = I_a$.
    \end{proof}

    \begin{proof}[Proof of Theorem \ref{thm:identifiability_J}]

    Theorem \ref{thm:identifiability_I} establishes that the set $\mathcal{X}$ uniquely determines both $I$ and its partition $\mathcal{I}$, with the exception of potential permutations of labels. When we have $I$ and its partition $mathcal{I}$ available, represented as $\{I_1, \dots, I_K\}$, for any index $i$ belonging to $I$, there exists a single integer $1 \leq a \leq K$ such that $i \in I_a$. We then construct a row vector $A_{i\cdot}$ of dimension $K$, akin to the canonical basis $\textbf{e}_a$ in $\mathbb{R}^K$, where the element at position $a$ equals 1, and all other elements are 0. Consequently, the matrix $A_I$, which has dimensions $|I| \times K$ and is composed of rows $A_{i\cdot}$, is uniquely determined, except for possible multiplications by permutation matrices.

    We show below that $A_J$ is also identifiable up to a signed permutation matrix. We begin by observing that, for each $i \in I_k$ for some $k \in [K]$ and any $j \in J$, Model \eqref{eq:lfm} implies
    \begin{equation*}
        \chi(i,j) = \sum_{a \in s(i)} A_{ia} \wedge A_{ja} = A_{jk}
    \end{equation*}
    and, after averaging over all $i \in I_k$,
    \begin{equation*}
        A_{jk} = \frac{1}{|I_k|} \sum_{i \in I_k} \chi(i,j).
    \end{equation*}
    Repeating this for every $k \in [K]$, we obtain the formula
    \begin{equation*}
        A_{j \cdot} = \left( \frac{1}{|I_1|} \sum_{i \in I_1} \chi(i,j), \dots, \frac{1}{|I_K|} \sum_{i \in I_K} \chi(i,j) \right),
    \end{equation*}
    for each $j \in J$, which shows that $A_J$ can be determined uniquely from $\mathcal{X}$ up to a permutation. Therefore, $A_J$ is identifiable which concludes the proof.
    \end{proof}
    \section{Proof of Section \ref{sec:estimation_chap_4}}
    For the sake of notations, we set $\hat{\nu}_{n,m}(1,\dots,d) := \hat{\nu}_{n,m}$.
    \begin{lemma}
        \label{lem:exp_ineq_mado}
        Let $(\mathbf{X}_t, t \in \mathbb{Z})$ satisfies the conditions in Theorem \ref{thm:exp_ineq}. Choose a $c_2 \in (0,\infty)$, and let $z = y - k^{-c_2}$ for any $y \geq k^{-c_2}$. Then for $k \geq 4$, there is a constant $c_1 > 0$ such that
        \begin{equation*}
            \mathbb{P} \left\{ |\hat{\nu}_{n,m} - \nu_m | \geq y + \frac{1}{k+1} \right\} \leq (2d+1) \exp \left\{ - \frac{c_1 k z^2}{1 + z \ln k (\ln \ln k)} +c_2 \ln k \right\}.
        \end{equation*}     
    \end{lemma}
        \begin{proof}
            Since for every $j \in \{1,\dots,d\}$, $U_{m,1}^{(j)}$ is uniformly distributed under the unit segment, we directly obtain
            \begin{equation*}
                \mathbb{E}\left[ \frac{1}{d} \sum_{j=1}^d U_{m,1}^{(j)} \right] = \frac{1}{2}.
            \end{equation*}
            Furthermore, by simple computations, and using the very nature of scaled ranks, we have
            \begin{equation*}
                \frac{1}{k} \sum_{i=1}^k \frac{1}{d} \sum_{j=1}^d \hat{U}_{n,m,i}^{(j)} =  \frac{1}{k} \sum_{i=1}^k \frac{1}{d} \sum_{j=1}^d \frac{i}{k+1} = \frac{1}{2}.
            \end{equation*}
            Hence, the desired quantity that we want to control can be simply rewritten as:
            \begin{equation*}
                |\hat{\nu}_{n,m} - \nu_m| = \left| \frac{1}{k} \sum_{i=1}^k \left[ \bigvee_{j=1}^d \hat{U}_{n,m,i}^{(j)} - \mathbb{E}\bigvee_{j=1}^d U_{m,i}^{(j)} \right] \right|.
            \end{equation*}
            Introduce by $\tilde{U}_{n,m,i}^{(j)}$ the scaled ranks using the factor $k$ (instead of $k+1$ for $\hat{U}_{n,m,i}^{(j)}$), $j=1,\dots,d$, $i=1,\dots,k$. Direct computation gives
            \begin{equation*}
                |\hat{\nu}_{n,m} - \nu_m| \leq \frac{1}{k+1} + \left| \frac{1}{k} \sum_{i=1}^k \left[ \bigvee_{j=1}^d \tilde{U}_{n,m,i}^{(j)} - \mathbb{E}\bigvee_{j=1}^d U_{m,i}^{(j)} \right] \right|.
            \end{equation*}
            The remaining term can be upper bounded by
            \begin{equation*}
                E_1 + E_2 := \left| \frac{1}{k} \sum_{i=1}^k \left[ \bigvee_{j=1}^d \tilde{U}_{n,m,i}^{(j)} - \bigvee_{j=1}^d U_{m,i}^{(j)} \right] \right| + \left| \frac{1}{k} \sum_{i=1}^k \left[ \bigvee_{j=1}^d U_{m,i}^{(j)} - \mathbb{E}\bigvee_{j=1}^d U_{m,i}^{(j)} \right] \right|.
            \end{equation*}
            Furthermore, $E_1$ is upper bounded by the following well-known quantity
            \begin{equation*}
                E_1 \leq \underset{j \in \{1,\dots,d\}}{\sup} \, \underset{i \in \{1,\dots,k\}}{\sup} \, \left| \tilde{U}_{n,m,i}^{(j)} - U_{m,i}^{(j)} \right| \leq \underset{j \in \{1,\dots,d\}}{\sup} \, \underset{x \in \mathbb{R}}{\sup} \, \left|\hat{F}_{n,m}^{(j)}(x) - F_m^{(j)}(x) \right|.
            \end{equation*}
            Thus, applying union bound and Lemma \ref{lem:ineq_gc} to $E_1$  and Lemma \ref{lem:merlevede} to $E_2$ (taking $n=k$ in the statement of both lemmas), we deduce the result.
        \end{proof}        
        \begin{proof}[Proof of Theorem \ref{thm:exp_ineq}]
            Taking Lemma \ref{lem:exp_ineq_mado} with $d = 2$ and $c_2 = 1$, we obtain
            \begin{equation*}
            \mathbb{P} \left\{ |\hat{\nu}_{n,m}(i,j) - \nu_m(i,j) | \geq y + \frac{1}{k+1} \right\} \leq 5 \exp \left\{ - \frac{c_1 k z^2}{1 + z \ln k (\ln \ln k)} + \ln k \right\},
        \end{equation*}
        where $z = y - k^{-1}$. Now taking
        \begin{equation}
            z = c_1 \left( \sqrt{\frac{\ln\left(\frac{k}{\delta}\right)}{k}} + \frac{\ln k \ln\ln(k)\ln \left( \frac{ k}{\delta} \right)}{k}\right),
        \end{equation}
        it implies that with probability at least $1-\delta$
        \begin{equation*}
            |\hat{\nu}_{n,m}(i,j) - \nu_m(i,j)| \leq c_1 \left( \sqrt{\frac{\ln\left(\frac{k}{\delta}\right)}{k}} + \frac{\ln k \ln \ln(k) \ln \left( \frac{k}{\delta} \right)}{k}\right) + \frac{1}{k+1} + \frac{1}{k}.
        \end{equation*}
        Now, using that
        \begin{equation*}
        	\hat{\chi}_{n,m}(i,j) = f(\tilde{\nu}_{n,m}(i,j)),
        \end{equation*}
        where
        \begin{align*}
            f\colon [0,1/6] & \rightarrow[0,1]\\
            x&\mapsto 2-\frac{0.5+x}{0.5-x},
        \end{align*}
        which is nonincreasing and $9$-lipschitz, hence one has
        \begin{align*}
        	|\hat{\chi}_{n,m}(i,j) - \chi_m(i,j)| = |f(\tilde{\nu}_{n,m}(i,j)) - f(\nu_m(i,j))| \leq 9 |\tilde{\nu}_{n,m}(i,j) - \nu_m(i,j)| \leq |\hat{\nu}_{n,m}(i,j) - \nu_m(i,j)|.
        \end{align*}
        We thus obtain that with probability at least $1-d^{-c_0}$ through a union bound
        \begin{equation*}
            \underset{1 \leq i < j \leq d}{\sup} \, |\hat{\chi}_{n,m}(i,j) - \chi_m(i,j) | \leq c_1 \left( \sqrt{\frac{\ln\left(k d \right)}{k}} + \frac{\ln k \ln \ln(k) \ln \left( k d \right)}{k}  \right),
        \end{equation*}
        for a sufficiently large constant $c_1$, thus the desired result.
        \end{proof}
        
    \section{Proof of Section \ref{sec:stat_guarantees}}

    \subsection{Proof of Section \ref{subsec:stat_guarantees_k_i}}

    \begin{lemma}
        \label{lem:stat_gua_K_I}
        Under Condition \ref{cond:SSC}, for any $i \in I_a$ with some $a \in [K]$, the following inequalities hold on the event
        \begin{enumerate}[label=\textcolor{frenchblue}{($A\arabic*$)}]
            \item $\hat{\chi}_{n,m}(i,j) \leq \delta$ for all $j \in I_b$ for some $b \in [K]$ with $b \neq a$; \label{lem:stat_gua_K_I_1}
            \item $1-\hat{\chi}_{n,m}(i,j) \leq \delta$ for all $j \in I_b$; \label{lem:stat_gua_K_I_2}
            \item $1-\hat{\chi}_{n,m}(i,k) > \delta$ for all $k \notin I_a$. \label{lem:stat_gua_K_I_3}
        \end{enumerate}
    \end{lemma}
    \begin{proof}
        For the entire proof, we work under the event $\mathcal{E}$ defined in \eqref{eq:mathcal_E}. To prove \ref{lem:stat_gua_K_I_1}, we observe that for any $j \in I_b$, with $b \in [K]$ and $b \neq a$, $\chi(i,j) = 0$, whence
        \begin{equation*}
            \hat{\chi}_{n,m}(i,j) \leq \chi(i,j) + \delta  = \delta.
        \end{equation*}
        So \ref{lem:stat_gua_K_I_1} holds.

        To prove \ref{lem:stat_gua_K_I_2}, taking $j \in I_a$ gives $\chi(i,j) = 1$, then
        \begin{equation*}
            1-\hat{\chi}_{n,m}(i,j) \leq 1 - \chi(i,j) + \delta = \delta,
        \end{equation*}
        and hence \ref{lem:stat_gua_K_I_2} holds.

        To obtain \ref{lem:stat_gua_K_I_3}, for $k \notin I_a$, Condition \ref{cond:SSC1} implies
        \begin{equation*}
            \chi(i,j) = A_{ka} < 1-2\delta.
        \end{equation*}
        Next, we obtain
        \begin{equation*}
            1-\hat{\chi}_{n,m}(i,j) \geq 1-\delta - \chi(i,j) = 1-\delta - A_{ka} > \delta.
        \end{equation*}
    \end{proof}

    \begin{proof}[Proof of Theorem \ref{thm:stat_K_I}] We work on the event $\mathcal{E}$ throughout the proof. Without loss of generality, we assume that the label permutation $\pi$ is the identity. Following \cite{bing2020adaptative}, we start by point out that the three following claims are sufficient to prove \ref{thm:stat_K_I_1}-\ref{thm:stat_K_I_3}. Let $[\hat{K}]$ and $\hat{I}^{(i)}$ be respectively the set of integers in the maximum clique $\bar{\mathcal{G}}$ of $G$ given in Step 4 of Algorithm \ref{alg:pure_variable} and the set defined in Step 6 of Algorithm \ref{alg:pure_variable}.
        \begin{enumerate}[label=\textcolor{frenchblue}{\bf(\arabic*)}]
            \item For any $i \in J$, we have \textcolor{frenchblue}{Pure(i)} = \textcolor{frenchblue}{False}; \label{proof_thm:stat_K_I_1}
            \item For any $i \in I_a$ and $a \in [K]$, we have \textcolor{frenchblue}{Pure(i)} = \textcolor{frenchblue}{True}, $I_a = \hat{I}^{(i)}$. \label{proof_thm:stat_K_I_2}
        \end{enumerate}
        To prove \ref{proof_thm:stat_K_I_1}, let $i \in J$ be fixed. We first prove that \textcolor{frenchblue}{Pure(i)} = \textcolor{frenchblue}{False} under $\hat{I}^{(i)} \cap I \neq \emptyset$. Under this hypothesis, we have $[\hat{K}] \subseteq [K]$ and no variables $i \in J$ belongs to $[\hat{K}]$ by Step 4 of Algorithm \ref{alg:pure_variable}. Now, if $i$ was taken at Step 6 of Algorithm \ref{alg:pure_variable}, then by $\hat{I}^{(i)} \cap I \neq \emptyset$, there exists $b \in [K]$ and $j \in I_b$ such that
        \begin{equation*}
            1 - \hat{\chi}_{n,m}(i,j) \leq  \delta,
        \end{equation*}
        which is prevented from \ref{lem:stat_gua_K_I_3} of Lemma \ref{lem:stat_gua_K_I}. This shows that for any $i \in J$, if $\hat{I}^{(i)} \cap I \neq \emptyset$, then \textcolor{frenchblue}{Pure(i)} = \textcolor{frenchblue}{False}. 
        
        Therefore to complete the proof of \ref{proof_thm:stat_K_I_1}, we show $\hat{I}^{(i)} \cap I \neq \emptyset$ is impossible when $i \in J$, under our assumptions. By construction of the algorithm, we have to verify that no $i \in J$ belongs to $i \in [\hat{K}]$ in Step 4 of Algorithm \ref{alg:pure_variable}. 
        We have, using \ref{lem:stat_gua_K_I_1} of Lemma \ref{lem:stat_gua_K_I}, for every $k \in I_a$ and $j \in I_b$ with $a,b \in [K]$
        \begin{equation*}
            \hat{\chi}_{n,m}(k,j) \leq  \delta.
        \end{equation*}
        Hence $[K]$ is a clique and $ [K] \subseteq [\hat{K}]$. 
        Now suppose $i \in [\hat{K}]$ while $i \in J$, then we have
        \begin{equation}
            \label{eq:proof_thm_stat_K_I}
            \hat{\chi}_{n,m}(i,j) \leq \delta \textrm{ for any } j \in [\hat{K}], \, j \neq i.
        \end{equation}
        Take $k \in I_{a^*}$ and $j \in I_{b^*}$ for $a^*, b^* \in [K]$ such that $A_{ia^*} > 2\delta$ and $A_{ib^*} > 2 \delta$ where the existence of such indices in $[K]$ is guaranteed by Condition \ref{cond:SSC2}. We hence obtain
        \begin{equation*}
            \chi(i,k) = \sum_{a \in s(i)} A_{ia} \wedge A_{ka} \geq A_{ia^*} \wedge A_{ka^*} = A_{ia^*} 
        \end{equation*}
        where the last inequality follows from $k \in I_{a^*}$. Hence,
        \begin{equation*}
            \chi(i,k) > A_{ia^*} > 2\delta
        \end{equation*}
        where the last inequality stems down from Condition \ref{cond:SSC2}. Then, under $\mathcal{E}$,
        \begin{equation*}
            \hat{\chi}_{n,m}(i,k) \geq \chi(i,k) - \delta > \delta.
        \end{equation*}
        The same arguments hold for $j \in I_{b^*}$ and hence
        \begin{equation*}
            \hat{\chi}_{n,m}(i,k) > \delta \textrm{ and } \hat{\chi}_{n,m}(i,j) > \delta,
        \end{equation*}
        which contradicts \eqref{eq:proof_thm_stat_K_I} and guarantees that $\hat{I}^{(i)} \cap I = \emptyset$ is impossible when $i \in J$. Indeed, the maximum clique that we can obtain from Step 4 of Algorithm \ref{alg:pure_variable} is $[\hat{K}] \setminus \{i\}$ by the above inequality.
        
        To prove \ref{proof_thm:stat_K_I_2}, since $\hat{I}^{(i)} \cap I \neq \emptyset$ with $i \in I_a$ under $\mathcal{E}$ from the discussion of \ref{proof_thm:stat_K_I_1}, then the statement of \ref{proof_thm:stat_K_I_2} should only be verified at step 6 of the algorithm since only pure variables are gathered at step 4 of the algorithm. Now, from step 6 of Algorithm \ref{alg:pure_variable}, we have to show that
        \begin{equation*}
            1-\hat{\chi}_{n,m}(i,j) \leq \delta,
        \end{equation*}
        for any $j \in \hat{I}^{(i)}$ and $j \in I_a$. Since $i \in I_a$, \ref{lem:stat_gua_K_I_2} in Lemma \ref{lem:stat_gua_K_I} states that the above inequality stands. Thus we have shown that for any $i \in I_a$, \textcolor{frenchblue}{Pure(i)} = \textcolor{frenchblue}{True}. We conclude the proof.
    \end{proof}

    \subsection{Proof of Section \ref{subsec:stat_guarantees_A}}

    \begin{proof}[Proof of Theorem \ref{thm:stat_A}]
    \paragraph{Proof of \ref{item_i:thm_stat_A_upper_bound}}
        The proof of Theorem \ref{thm:stat_A}, item \ref{item_i:thm_stat_A_upper_bound} implies two steps:
        \begin{enumerate}[label=\textcolor{frenchblue}{($S\arabic*$)}]
            \item We write $\bar{A} = AP$, and prove the first error bound for $\hat{A}_{\hat{I}}- \bar{A}_{\hat{I}}$.; \label{thm:stat_A_1}
            \item We prove the error bounds $\hat{A}_{\hat{J}} - \bar{A}_{\hat{J}}$. \label{thm:stat_A_2}
        \end{enumerate}
        For ease of the notation and without loss of generality, we make the blanket assumption that the permutation matrix $P$ is the identity so that $\bar{A} = A$ for the remainder of the proof. Let $s(j) = ||A_{j\cdot}||_0$ for $j =1,\dots,d$. For the first step \ref{thm:stat_A_1}, from the construction of $\hat{A}_{\hat{I}}$ and parts \ref{thm:stat_K_I_1}-\ref{thm:stat_K_I_3} in Theorem \ref{thm:stat_K_I}, we have for any $i \in \hat{I}_a$ and the definition of $I$ implies $A_{ia} = 1$. Then
        \begin{equation*}
            ||\hat{A}_{\hat{I}} - A_{\hat{I}}||_{\infty} = \underset{j \in \hat{I}}{\max} ||\hat{A}_{j\cdot}-A_{j\cdot} ||_{\infty} = 0
        \end{equation*}
        Then for any $j \in \hat{I}$, we have
        \begin{equation*}
            ||\hat{A}_{j\cdot} - A_{j\cdot} ||_2 = 0.
        \end{equation*}
        For the second step of the proof \ref{thm:stat_A_2}, we will make use of the results of the Lemma stated here first and proved at the end of this section.
        \begin{lemma}
            \label{lem:stat_A}
            Under the conditions of Theorem \ref{thm:stat_A}, on the event $\mathcal{E}$, we have $\beta^{(j)}_a = 0$ implies $\bar{\beta}^{(j)}_a = 0$, for any $j \in \hat{J}$ and $a \in [\hat{K}]$.
        \end{lemma}
        Let us head into the proof of \ref{thm:stat_A_2}. For each $j \in \hat{J}$, we have by the very nature of our estimator:
        \begin{equation*}
            ||\hat{A}_{j\cdot} - A_{j\cdot}||_2 = ||\hat{\beta}^{(j)} - \beta^{(j)}||_2 \leq ||\bar{\beta}^{(j)} - \beta^{(j)}||_2 + ||\hat{\beta}^{(j)} - \bar{\beta}^{(j)}||_2.
        \end{equation*}
        Because this is a projection
		\begin{equation*}
			||\left.\hat{\beta}^{(j)}\right|_{\widehat{\mathcal{S}}} - \left.\bar{\beta}^{(j)}\right|_{\widehat{\mathcal{S}}}||_2 \leq ||\left.\beta^{(j)}\right|_{\widehat{\mathcal{S}}} - \left.\bar{\beta}^{(j)}\right|_{\widehat{\mathcal{S}}}||_2,
		\end{equation*}   
        hence
        \begin{equation*}
            ||\hat{\beta}^{(j)} - \bar{\beta}^{(j)}||_2 \leq ||\beta^{(j)} - \bar{\beta}^{(j)}||_2.
        \end{equation*}
       Then
        \begin{equation*}
            ||\hat{A}_{j\cdot} - A_{j\cdot}||_2 = ||\hat{\beta}^{(j)} - \beta^{(j)}||_2 \leq 2||\bar{\beta}^{(j)} - \beta^{(j)}||_2.
        \end{equation*}
        Furthermore, we can show that
        \begin{equation*}
            ||\bar{\beta}^{(j)} - \beta^{(j)}||_\infty \leq 2\delta,
        \end{equation*}
        indeed, for any $a \in \{1,\dots,K\}$ with $\bar{\chi}^{(j)}_a > \delta$ we have
        \begin{equation*}
            |\bar{\beta}^{(j)}_a - \beta^{(j)}_a| = \left| \frac{1}{|\hat{I}_a|} \sum_{i \in \hat{I}_a} \hat{\chi}_{n,m}(i,j) - A_{ja} \right| \leq \frac{1}{|\hat{I}_a|}\sum_{i \in \hat{I}_a} \left| \hat{\chi}_{n,m}(i,j) - A_{ja} \right|.
        \end{equation*}
        By Theorem \ref{thm:stat_K_I}, $\hat{I} = I$, then if $i \in I_a$, then $A_{ja} = \chi(i,j)$ and as we are on the event $\mathcal{E}$, we obtain that
        \begin{equation*}
            |\hat{\chi}_{n,m}(i,j) - \chi(i,j)| \leq \delta.
        \end{equation*}
        Thus,
        \begin{equation*}
            |\bar{\beta}^{(j)}_a - \beta^{(j)}_a| \leq 2 \delta,
        \end{equation*}
        whenever $a \in \{1,\dots,K\}$ with $\bar{\chi}^{(j)}_a > \delta$. Now take $a \in \{1,\dots,K\}$ such that $\bar{\chi}_a^{(j)}\leq \delta$, we obtain
        \begin{equation*}
            |\bar{\beta}^{(j)}_a - \beta^{(j)}_a| = A_{ja}.
        \end{equation*}
        If $i \in I_a$, then $\chi(i,j) = A_{ja}$ and under the event $\mathcal{E}$, we obtain
        \begin{equation*}
            A_{ja} \leq \hat{\chi}_{n,m}(i,j) + \delta.
        \end{equation*}
        Then for any $i \in I_a$
        \begin{equation*}
            A_{ja} \leq \frac{1}{|\hat{I}_a|} \sum_{i \in \hat{I}_a} \hat{\chi}_{n,m}(i,j) + \delta = \bar{\chi}^{(j)}_a + \delta \leq 2 \delta.
        \end{equation*}
        Hence, we have, as stated
        \begin{equation*}
            ||\bar{\beta}^{(j)} - \beta^{(j)}||_\infty \leq 2\delta.
        \end{equation*}
        Then following Lemma \ref{lem:stat_A} and using $\hat{K} = K$ on the event $\mathcal{E}$ gives
        \begin{equation*}
            ||\bar{\beta}^{(j)} - \beta^{(j)}||_2 = \left( \sum_{a=1}^K |\bar{\beta}^{(j)}_a - \beta^{(j)}_a|^2 \right)^{1/2} = \left( \sum_{a \in s(j)} |\bar{\beta}^{(j)}_a - \beta^{(j)}_a|^2 \right)^{1/2} \leq 2 \sqrt{s(j)} \delta.
        \end{equation*}
        This completes the proof of the last step and of Theorem \ref{thm:stat_A} \ref{item_i:thm_stat_A_upper_bound}.
    \paragraph{Proof of \ref{item_ii:thm_stat_A_supp_recov}.} In the initial stage of the proof, let us demonstrate that, under the event $\mathcal{E}$
    \begin{equation}
        \label{eq:supp_recov_proj_hard}
        \textrm{supp}(\hat{\beta}^{(j)}) = \textrm{supp}(\bar{\beta}^{(j)}), \, \forall j \in \{1,\dots,d\}.
    \end{equation}
    Through our initial construction, we immediately obtain that $\hat{\beta}^{(j)}_a = 0$ whenever $\bar{\beta}^{(j)}_a = 0$ for any $a \in [\hat{K}]$. Now, let us consider any $a \in [\hat{K}]$ with $\bar{\beta}^{(j)}_a > 0$, a condition equivalent to $\bar{\beta}^{(j)}_a > \delta$. Our task is to establish that $\bar{\beta}^{(j)}_a > \tau$ where
    \begin{equation*}
        \tau := \frac{1}{\rho} \left(\sum_{a=1}^\rho \bar{\beta}^{(j)}_a -1 \right) \textrm{ and } \rho =\max \left\{ p \in \textrm{supp}(\bar{\beta}^{(j)}) : \bar{\beta}_p^{(j)} > \frac{1}{p} \left(\sum_{a=1}^p \bar{\beta}^{(j)}_a -1 \right) \right\}.
    \end{equation*}
    Let us show that $p = \textrm{supp}(\bar{\beta}^{(j)})$. Indeed for any $a \in \textrm{supp}(\bar{\beta}^{(j)})$
    \begin{equation*}
        \bar{\beta}^{(j)}_a = \frac{1}{|\hat{I}_a|} \sum_{i \in \hat{I}_a} \hat{\chi}_{n,m}(i,j).
    \end{equation*}
    If $i \in I_a$, then $\chi(i,j) = A_{ja}$ and we obtain the following inequality under the event $\mathcal{E}$
    \begin{equation*}
        \hat{\chi}_{n,m}(i,j) \leq A_{ja} + \delta.
    \end{equation*}
    We obtain simultaneously
    \begin{equation*}
    	\hat{\chi}_{n,m}(i,j) \leq A_{ja} + \delta, \; \forall i \in \hat{I}_a, \quad \bar{\beta}^{(j)}_a \leq A_{ja} + \delta, \; a \in \textrm{supp}(\bar{\beta}^{(j)}).
    \end{equation*}
    Summing across all instances of $a \in \textrm{supp}(\bar{\beta}^{(j)})$ and employing Condition \ref{cond:(i)},
    \begin{equation*}
        \sum_{a=1}^p \bar{\beta}^{(j)}_a \leq 1 + p \delta,
    \end{equation*}
    this leads us to achieve
    \begin{equation*}
        \frac{1}{p} \left( \sum_{a=1}^p \bar{\beta}^{(j)}_a -1 \right) \leq \delta.
    \end{equation*}
    Building upon our initial assumption, we can express
    \begin{equation*}
        \bar{\beta}^{(j)}_a > \delta \geq \frac{1}{p} \left( \sum_{a=1}^p \bar{\beta}^{(j)}_a -1 \right).
    \end{equation*}
    From this, we can infer $\rho = \textrm{supp}(\bar{\beta}^{(j)})$ but also $\bar{\beta}^{(j)}_a > \tau$, hence $\hat{\beta}^{(j)}_a > 0$. We obtain the result of the initial stage stated in \eqref{eq:supp_recov_proj_hard}.

    Let us remember that Lemma \ref{lem:stat_A} suggests $\textrm{supp}(\bar{\beta}^{(j)}) \subseteq \textrm{supp}(\beta^{(j)})$, and by the initial stage of the proof (see Equation \eqref{eq:supp_recov_proj_hard}), we infer $\textrm{supp}(\hat{\beta}^{(j)}) \subseteq \textrm{supp}(\beta^{(j)})$ for any $j \in \hat{J}$. Hence $\textrm{supp}(\hat{A}_{\hat{J}}) \subseteq \textrm{supp}(A_{\hat{J}})$. Furthermore Theorem \ref{thm:stat_K_I} provides the result that $\hat{I}_a = I_a$ for all $a\in [\hat{K}]$. From the way we construct $\hat{A}_{\hat{I}}$, we have $\textrm{supp}(\hat{A}_{\hat{I}}) \subseteq \textrm{supp}(A_{\hat{I}})$. Therefore, we have proved $\textrm{supp}(\hat{A}) \subseteq \textrm{supp}(A)$.

    On the contrary, considering any $j \in J_1$, we have the knowledge that $\beta^{(j)}_a > 2 \delta$ for every $a \in \textrm{supp}(\beta^{(j)})$. Exploiting this insight and the additional observation that $|\bar{\chi}^{(j)}_a - \beta^{(j)}_a| \leq \delta$, we deduce
    \begin{equation*}
        |\bar{\chi}^{(j)}_a| \geq |\beta^{(j)}_a| - |\bar{\chi}^{(j)}_a - \beta^{(j)}_a| > \delta,
    \end{equation*}
    which implies $\bar{\beta}^{(j)}_a > 0$, hence $\textrm{supp}(\beta^{(j)}) \subseteq \textrm{supp}(\bar{\beta}^{(j)})$ with $j \in J_1$. Using the initial stage of the proof, we have $\textrm{supp}(A_{J_1}) \subseteq \textrm{supp}(\hat{A}_{J_1})$.
    
    \paragraph{Proof of \ref{item_iii:thm_stat_A_clust_recov}.}
    	The proof is in the same line of \cite[Theorem 7]{bing2020adaptative}, we recall it for consistency. We first show that $TFPP(\widehat{\mathcal{G}}) = 0$. From the result of part \ref{item_ii:thm_stat_A_supp_recov}, we know that $\textrm{supp}(\hat{A}) \subseteq \textrm{supp}(A)$. Thus,
    	\begin{equation*}
    		\sum_{j \in [d], a \in [K]} \mathds{1}_{ \{ A_{ja} = 0, \hat{A}_{ja} > 0 \} },
    	\end{equation*}
    	which implies $TFPP(\widehat{\mathcal{G}}) = 0$.
    	In order to prove the result of $TFNP(\widehat{\mathcal{G}})$, observe
    	\begin{equation*}
    		\sum_{j \in [d], a \in [K]} \mathds{1}_{ \{ A_{ja} > 0 \} } = |I| + \sum_{j \in J} s(j),
    	\end{equation*}
    	with $s(j) = ||A_{j \cdot}||_0$ for each $j \in J$. For a given $I$, we partition $[d] = I \cup J_1 \cup (J \setminus J_1)$. Theorem \ref{thm:stat_K_I} implies that $\hat{I} = I$ and from the way we construct $\hat{A}_{\hat{I}}$, we have
    	\begin{equation*}
    		\sum_{j \in I} \mathds{1}_{ \{ \hat{A}_{ja} > 0, A_{ja} = 0  \} } =0.
    	\end{equation*} 
    	Next, we consider the set $J_1$. Part \ref{item_ii:thm_stat_A_supp_recov} gives $\normalfont{supp}(A_{J_1}) = supp(\hat{A}_{J_1})$ which yields
    	\begin{equation*}
    		\sum_{j \in J_1, a \in [K]}\mathds{1}_{ \{ \hat{A}_{ja} > 0, A_{ja} = 0 \} }=0.
    	\end{equation*}
    	Finally, we consider the set $J \setminus J_1$. By examining the proof of part \ref{item_ii:thm_stat_A_supp_recov}, we have necessarily $\hat{A}_{ja} > 0$ if $A_{ja} > 2\delta$ for any $j\in J_1$, and $a \in [K]$. Thus, 
    	\begin{equation*}
    		\sum_{j \in J \setminus J_1, a \in [K]} \mathds{1}_{ \{ A_{ja} > 0, \hat{A}_{ja} = 0 \} } \leq \sum_{j \in J \setminus J_1} t(j).
    	\end{equation*}
    	We hence obtain by combining the above inequalities
    	\begin{equation*}
    		TFNP(\widehat{\mathcal{G}}) = (\widehat{\mathcal{G}}) = \frac{\sum_{j\in [d], a \in [K]} \mathds{1}_{\{ A_{ja} > 0, \hat{A}_{ja} = 0 \}}}{\sum_{j\in [d], a \in [K]} \mathds{1}_{\{ A_{ja} > 0\}}} \leq \frac{\sum_{ j \in J \setminus J_1} t(j)}{|I| + \sum_{j\in J} s(j)} \leq \frac{\sum_{j \in J_1} t(j)}{|I| + \sum_{j \in J} s(j)}.
    	\end{equation*}
    \end{proof} 
    To conclude this section, we give below the proof of the intermediary result used in the proof.
    \begin{proof}[Proof of Lemma \ref{lem:stat_A}]
        Suppose that $\beta^{(j)}_a =0$, which implies, by definition, $A_{ja} = 0$. Now take $i \in  I_a$, we thus have under $\mathcal{E}$
        \begin{equation*}
            \hat{\chi}_{n,m}(i,j) \leq \chi(i,j) + \delta = A_{ja} + \delta \leq 2\delta.
        \end{equation*}
        We thus obtain, under the event $\mathcal{E}$ and $\beta^{(j)}_a = 0$, that
        \begin{equation*}
            \bar{\chi}^{(j)}_a = \frac{1}{|\hat{I}_a|} \sum_{i \in \hat{I}_a} \hat{\chi}_{n,m}(i,j) \leq 2 \delta
        \end{equation*}
        and we obtain
        \begin{equation*}
            \bar{\beta}_a^{(j)} = 0.
        \end{equation*}
    \end{proof}

    \section{Proofs of Section \ref{sec:num_res}}
    \label{sec:proof_num_res}

    Denote in this section by $M_n^{(a)}$ the maximum of $Z_i^{(a)}$, $1 \leq i \leq n$ and $1 \leq a \leq K$. Also, for convenience, we call a function $u$ on $\mathbb{R}$ a normalizing function if $u$ is non-decreasing, right continuous, and $u(x) \rightarrow \pm \infty$, as $x \rightarrow \pm \infty$.

    \begin{proposition}
    \label{prop:dma}
    Suppose $(\textbf{Z}_t, t \in \mathbb{Z})$ is a moving maxima process of order $p$ as described in \eqref{eq:moving_maxima} where margins of $\boldsymbol{\epsilon}_1$ are standard Pareto with an Archimedean copula function with upper tail equal to $1$, see \eqref{eq:rv}, then there exist sequences $u_n^{(a)}$, $1 \leq a \leq K$, such that
        \begin{equation*}
            \mathbb{P}\left\{ M_n^{(a)} \leq u_n^{(a)}(x^{(a)}), 1 \leq a \leq K \right\} \rightarrow \Pi_{a=1}^K e^{-\frac{1}{x^{(a)}}}, \quad \normalfont{\textbf{x}} \in (0,\infty)^K.
        \end{equation*}
    \end{proposition}
    \begin{proof}
        This result is a direct application of \cite[Theorem 5.2]{hsing1989extreme} where most prominent arguments are taking from Examples in \cite[Section 6]{hsing1989extreme}. For definitions of conditions $D(u_n(x^{(1)}),\dots,u_n(x^{(K)}))$ and $D^{''}(u_n(x^{(1)}),\dots,u_n(x^{(K)}))$, we also refer to \cite{hsing1989extreme}. For $u_n^{(a)}(x) = \frac{nx(1-\rho^{p+1})}{1-\rho}$, $n \geq 1$, $x \in \mathbb{R}$, we obtain
        \begin{equation*}
            \mathbb{P}\left\{ M_n^{(a)} \leq u_n^{(a)}(x) \right\} = \left( \Pi_{\ell = 0}^p \mathbb{P}\left\{ \epsilon_1^{(a)} \leq \rho^{-\ell} u_n(x) \right\} \right)^n = \left( \Pi_{\ell=0}^p \left(1-\frac{\rho^\ell (1-\rho)}{nx(1-\rho^{p+1})} \right) \right)^n.
        \end{equation*}
        Noticing that
        \begin{equation*}
            \Pi_{\ell=0}^p \left(1-\frac{\rho^\ell (1-\rho)}{nx(1-\rho^{p+1})} \right) = \exp\left\{ \sum_{\ell=0}^p \ln \left(1-\frac{\rho^\ell (1-\rho)}{nx(1-\rho^{p+1})} \right) \right\},
        \end{equation*}
        and using $\exp\{x\} = 1+x+O(x^2)$ and $\ln(1-x) = -x+O(x^2)$ as $x \rightarrow 0$, we obtain that
        \begin{equation*}
            \Pi_{\ell=0}^p \left(1-\frac{\rho^\ell (1-\rho)}{nx(1-\rho^{p+1})} \right) = 1-\frac{1}{nx} + O\left(\frac{1}{n^2}\right),
        \end{equation*}
        hence
        \begin{equation*}
            \mathbb{P}\left\{ M_n^{(a)} \leq u_n^{(a)}(x) \right\} \underset{n \rightarrow \infty}{\longrightarrow} e^{-\frac{1}{x}} \mathds{1}_{\{x \geq 0 \}}.
        \end{equation*}
        Furthermore, since $\sigma(\textbf{Z}_t, t\leq 0)$ and $\sigma(\textbf{Z}_t, t\geq p+1)$ are two independent $\sigma$-fields, the condition $D(u_n(x^{(1)}),\dots,u_n(x^{(K)}))$ holds immediately for $(\textbf{Z}_t, t\in \mathbb{Z})$ for each $\textbf{x} \in \mathbb{R}^d$. Thus, it suffices to show that the condition $D^{''}(u_n(x^{(1)}),\dots,u_n(x^{(K)}))$ holds for each $\textbf{x} \in (0,\infty)^K$. For any fixed $\textbf{x} \in (0,\infty)^K$, one obtains simply the estimate
        \begin{align*}
            &1-\Pi_{\ell = 0}^{i-1} \mathbb{P}\left\{ \epsilon_{1,1} \leq \rho^{-\ell}u_n(x^{(1)}) \right\} = \frac{1-\rho^{i}}{1-\rho^{p+1}nx_1} + O\left(\frac{1}{n^2}\right), \\
            &1-\Pi_{\ell = p-i+1}^{p} \mathbb{P}\left\{ \epsilon_{1,2} \leq \rho^{-\ell}u_n(x^{(2)}) \right\} = \frac{\rho^{p-i+1}(1-\rho^{i})}{1-\rho^{p+1}nx_2} + O\left(\frac{1}{n^2}\right).
        \end{align*}
        Now, by \eqref{eq:rv}, one can obtain the following estimate:
        \begin{equation}
            \label{eq:prop_dma_i}
            1-\Pi_{\ell = p-i+1}^p \mathbb{P}\left\{ \epsilon_{1,1} \leq \rho^{-\ell-i}u_n(x^{(1)}), \epsilon_{1,2} \leq \rho^{-\ell}u_n(x^{(2)}) \right\} = \frac{\rho^i(1-\rho^{p-i+1})}{(1-\rho^{p+1})nx_1} + \frac{(1-\rho^{p-i+1})}{(1-\rho^{p+1})nx_2} + O\left( \frac{1}{n^2} \right),
        \end{equation}
        for all $n$. Indeed
        \begin{align*}
            &1-\mathbb{P}\left\{ \epsilon_{1,1} \leq \rho^{-\ell-i}u_n(x^{(1)}), \epsilon_{1,2} \leq \rho^{-\ell}u_n(x^{(2)}) \right\} = \\ &1-\varphi^\leftarrow\left( \varphi\left(1-\frac{\rho^{\ell + i}(1-\rho)}{(1-\rho^{p+1})nx_1} \right) +  \varphi\left(1-\frac{\rho^{\ell}(1-\rho)}{(1-\rho^{p+1})nx_2} \right)\right) = \\ & \frac{n}{n} \times \left(1-\varphi^\leftarrow\left( \varphi(1-\frac{1}{n}) \left\{\frac{\varphi\left(1-\frac{\rho^{\ell + i}(1-\rho)}{(1-\rho^{p+1})nx_1} \right)}{\varphi(1-\frac{1}{n})} +  \frac{\varphi\left(1-\frac{\rho^{\ell}(1-\rho)}{(1-\rho^{p+1})nx_2} \right)}{\varphi(1-\frac{1}{n})} \right\}\right) \right).
        \end{align*}
        The function $x \mapsto 1/\varphi(1-1/x)$ is regularly varying at infinity with index $1$. Therefore, its inverse function, the function $t \mapsto 1/(1-\varphi^\leftarrow(1/t))$ is regularly varying at infinity with index 1 (\cite[Theorem 1.5.12]{bingham1989regular}), and thus the function $1-\varphi^\leftarrow$ is regularly varying at zero with index $1$. We also have
        \begin{align*}
            &\frac{\varphi\left(1-\frac{\rho^{\ell + i}(1-\rho)}{(1-\rho^{p+1})nx_1} \right)}{\varphi(1-\frac{1}{n})} \underset{n \rightarrow \infty}{\longrightarrow} \frac{\rho^{\ell+i}(1-\rho)}{x_1}, \quad 
            \frac{\varphi\left(1-\frac{\rho^{\ell}(1-\rho)}{(1-\rho^{p+1})nx_2} \right)}{\varphi(1-\frac{1}{n})} \underset{n \rightarrow \infty}{\longrightarrow} \frac{\rho^\ell(1-\rho)}{x_2}.
        \end{align*}
        By the uniform convergence theorem (\cite[Theorem 1.5.2]{bingham1989regular}), the below term
        \begin{equation*}
            n \times \left(1-\varphi^\leftarrow\left( \varphi(1-\frac{1}{n}) \left\{\frac{\varphi\left(1-\frac{\rho^{\ell + i}(1-\rho)}{(1-\rho^{p+1})nx_1} \right)}{\varphi(1-\frac{1}{n})} + \frac{\varphi\left(1-\frac{\rho^{\ell}(1-\rho)}{(1-\rho^{p+1})nx_2} \right)}{\varphi(1-\frac{1}{n})} \right\}\right) \right)
        \end{equation*}
        converges to
        \begin{equation*}
            \frac{\rho^{\ell+i}(1-\rho)}{x_1}+\frac{\rho^\ell(1-\rho)}{x_2}.
        \end{equation*}
        And then, after elementary estimation, we obtain \eqref{eq:prop_dma_i}. Hence for $1 \leq i \leq p$, we have
        \begin{align*}
            &\mathbb{P}\left\{ Z_{1}^{(1)} > u_n(x^{(1)}), Z_{i}^{(2)} > u_n(x^{2)}) \right\} = \\ &1-\mathbb{P}\left\{ Z_{1}^{(1)} \leq u_n(x^{(1)}) \right\} - \mathbb{P}\left\{ Z_{1}^{(2)} \leq u_n(x^{(2)}) \right\} + \mathbb{P}\left\{ Z_{1}^{(1)} \leq u_n(x^{(1)}), Z_{i}^{(2)} \leq u_n(x^{(2)}) \right\} = \\ &1-\Pi_{\ell=0}^{p} \mathbb{P}\left\{ \epsilon_{1}^{(1)} \leq \rho^{-\ell} u_n(x^{(1)}) \right\} - \Pi_{\ell=0}^p \mathbb{P}\left\{ \epsilon_{1}^{(2)} \leq \rho^{-\ell} u_n(x^{(2)}) \right\} + \Pi_{\ell=0}^{i-1} \mathbb{P}\left\{ \epsilon_{1}^{(1)} \leq \rho^{-\ell} u_n(x^{(1)}) \right\} \times \\ & \Pi_{\ell = 0}^{p-i}\mathbb{P}\left\{ \epsilon_{1}^{(1)} \leq \rho^{-\ell - i}u_n(x^{(1)}), \epsilon_1^{(2)} \leq \rho^{-\ell} u_n(x^{(2)}) \right\} \Pi_{\ell = p-i+1}^{p} \mathbb{P}\left\{ \epsilon_{1}^{(2)} \leq \rho^{-\ell}u_n(x^{(2)}) \right\} = O \left( \frac{1}{n^2} \right).
        \end{align*}
        Since for $i > p$, we trivially obtain that
        \begin{equation*}
            \mathbb{P}\left\{ Z_{1}^{(1)} > u_n(x^{(1)}), Z_{i}^{(2)} > u_n(x^{(2)}) \right\} = O \left( \frac{1}{n^2} \right),
        \end{equation*}
        and noticing that $\epsilon_i^{(1)}$ and $\epsilon_i^{(2)}$ are playing symmetric role to $\epsilon_i^{(j)}$ and $\epsilon_i^{(k)}$ for $1 \leq j < k \leq K$, one concludes from this that the condition $D^{''}(u_n(x^{(1)}),\dots,u_n(x^{(K)}))$ holds for each $\textbf{x} \in (0,\infty)^K$. Hence, applying \cite[Theorem 5.2]{hsing1989extreme}, we obtain the result.
    \end{proof}

    We give below technical details on the heuristics $d_m = O(1/m)$ made in Section \ref{sec:num_res}. Let us consider the model in \eqref{eq:lfm} without noise, i.e.,  $\textbf{X} = A \textbf{Z}$. For $\theta > 0$, the Clayton copula is defined as
    \begin{equation*}
        C_{\theta}(u,v) = \left[ 1 + \{ (u^{-\theta} - 1) + (v^{-\theta} - 1) \} \right]^{-1/\theta}, \, (u,v) \in [0,1]^2.
    \end{equation*}
    The copula of the pair of componentwise maxima of an i.i.d. sample of size $m$ from continuous distribution with copula $C_\theta$ is equal to
    \begin{equation*}
        \{ C_\theta(u^{1/m}, v^{1/m}) \}^m = C_{\theta / m}(u,v).
    \end{equation*}
    For establishing the heuristic, consider $a, b \in [K]$ with $a \neq b$ and $i \in I_a, j \in I_b$. Let $C_m^{(i,j)}$ denote the copula between the pair $M_m^{(i)}$ and $M_m^{(j)}$. Remember that these maxima are drawn from Pareto distributions, appropriately scaled in the independent setting by $m$. Denote $C_{\infty}^{(i,j)}$ the extreme value copula between two independent standard Fréchet. The pre-asymptotic madogram is hence defined by
    \begin{align*}
    	\nu_m(i,j) &= \frac{1}{2} - \int_0^1 C_m^{(i,j)}(u,u)du = \frac{1}{2} - \int_0^1 C_\theta^{(i,j)} \left(u^{1/m},u^{1/m} \right)^m du \\
    			   &= \frac{1}{2} - \int_0^1 C_{\theta/m}^{(i,j)}(u,u)du.
    \end{align*}
    Using the same computations, we obtain for the madogram
    \begin{align*}
    	\nu(i,j) = \frac{1}{2} - \int_0^1 C_{\infty}^{(i,j)}(v,v)dv.
    \end{align*}
    Since $C_{\theta}$ is positive lower orthant dependent for any $\theta > 0$, it follows that $C_{\theta/m}$ is positive lower orthant dependent for any $m\geq 1$ and $\theta > 0$. We hence obtain that $\nu_m(i,j) \in [0,1/6]$ and using that the function $f(x) = (0.5+x)/(0.5-x)$ is Lipschitz for $x \in [0,1/6]$, we obtain
    \begin{align*}
    	\left| \chi_m(i,j) - \chi(i,j) \right| = \left| \frac{0.5+\nu_m(i,j)}{0.5-\nu_m(i,j)} - \frac{0.5+\nu(i,j)}{0.5-\nu(i,j)} \right| \leq 9 |\nu_m(i,j) - \nu(i,j)|.	
	\end{align*}
	Now
	\begin{align*}
		|\nu_m(i,j) - \nu(i,j)| = &\bigg| \int_{0}^1 C_{\theta/m}^{(i,j)}\left(u,u\right)du - \int_0^1 C_{\infty}^{(i,j)}(u,u)du \bigg|.
	\end{align*}
	Now, applying \cite[Proposition 4.3]{bucher2014extreme}, we obtain
	\begin{align*}
		\int_0^1 \left| C_{\theta/m}^{(i,j)}(u,u) - C_\infty^{(i,j)}(u,u) \right| dv \leq \underset{u,v \in [0,1]^2}{\sup} \, \left| C_{\theta/m}^{(i,j)}(u,v) - C_\infty^{(i,j)}(u,v) \right| = O\left(\frac{1}{m}\right).
	\end{align*}
    \section{Supplementary Lemmata}
    \begin{lemma}
		\label{lem:tail_balance_condition}
		Let $\normalfont{\textbf{X}} = \normalfont{\textbf{X}}_1 + \normalfont{\textbf{X}}_2$ where both $\normalfont{\textbf{X}}_1$ and $\normalfont{\textbf{X}}_2$ are regularly varying with respective exponent measures $\Lambda_{\normalfont{\textbf{X}}_1}$, $\Lambda_{\normalfont{\textbf{X}}_2}$ and the following tail balance condition holds:
		\begin{equation}
			\label{eq:tail_balance_condition}
				\underset{x \rightarrow \infty}{\lim} \, \frac{\mathbb{P}\left\{ ||\normalfont{\textbf{X}}_2|| > x \right\}}{ \mathbb{P}\left\{ ||\normalfont{\textbf{X}}_1|| > x \right\} }=0.
		\end{equation}
		Then $\normalfont{\textbf{X}}$ is regularly varying with exponent measure $\Lambda_{\normalfont{\textbf{X}}}$ given by
		\begin{equation*}
			\Lambda_{\normalfont{\textbf{X}}} = \Lambda_{\normalfont{\textbf{X}}_1}.
		\end{equation*}
	\end{lemma}
	Under Condition \eqref{eq:tail_balance_condition} one may expect that the tail behavior of $\textbf{X}$ is mainly influenced by that of $\textbf{X}_1$.
	\begin{proof}
		Without loss of generality, we suppose that $\textbf{X}_1$ is regularly varying with tail index $\alpha$ equal to unity. We must prove that the sequence of measure $\{ \Lambda_x \}$ defined by
		\begin{equation*}
			\Lambda_x(\cdot) = \mathbb{P}\left\{ x^{-1}(\textbf{X}_1, \textbf{X}_2) \in \cdot \right\} / \mathbb{P}\left\{ ||\textbf{X}_1|| > x \right\},
		\end{equation*}
		is the only possible limit along a subsequence by applying \cite[Lemma B.1.29]{kulik2020heavy} and that the measure $\Lambda_{(\textbf{X}_1,\textbf{X}_2)}$ defined by $\Lambda_{(\textbf{X}_1,\textbf{X}_2)} = \Lambda_{\textbf{X}_1} \otimes \delta_0$ is the only possible limit along a subsequence by applying \cite[Lemma B.1.31]{kulik2020heavy}. Let $f$ be a bounded uniformly continuous function with support in $A_1 \times \mathbb{R}^d$ with $A_1$ separated from zero. Fix $\epsilon > 0$. Then there exists $||\textbf{x}_2|| \leq \eta$ which implies $||f(\textbf{x}_1, \textbf{x}_2) - f(\textbf{x}_1,0)|| \leq \epsilon$. By Condition \eqref{eq:tail_balance_condition} and since $A_1$ is separated from zero, we have
\begin{align*}
	\underset{x \rightarrow \infty}{\lim} \, \Lambda_x(f) &= \underset{x \rightarrow \infty}{\lim}\, \mathbb{E}\left[ f(x^{-1}(\textbf{X}_1, \textbf{X}_2)) \mathds{1}_{ \{ ||\textbf{X}_2|| > \eta x \} } \right] / \mathbb{P}\left\{ ||\textbf{X}_1 || > x \right\} \\
	&= \underset{x \rightarrow \infty}{\lim} \, \frac{\mathbb{E}\left[ f(x^{-1}(\textbf{X}_1, \textbf{X}_2)) \mathds{1}_{ \{ ||\textbf{X}_2|| > \eta x \} } \right]}{\mathbb{P}\left\{ ||\textbf{X}_2 || > x \right\}} \underset{x \rightarrow \infty}{\lim}\, \frac{\mathbb{P}\left\{ ||\textbf{X}_2 || > x \right\}}{\mathbb{P}\left\{ ||\textbf{X}_1 || > x \right\}} = 0.
\end{align*}
For every $\eta > 0$, since $||\textbf{X}_1||$ and $||\textbf{X}_2||$ are both regularly varying conserving the same tail index, the assumption implies
\begin{align*}
	\underset{x \rightarrow \infty}{\lim} \, \frac{\mathbb{P}\left\{ ||\textbf{X}_2 || > \eta x \right\}}{\mathbb{P}\left\{ ||\textbf{X}_1 || > x \right\}} = \underset{x \rightarrow \infty}{\lim} \, \frac{\mathbb{P}\left\{ ||\textbf{X}_2 || > \eta x \right\}}{\mathbb{P}\left\{ ||\textbf{X}_1 || > \eta x \right\}} \underset{x \rightarrow \infty}{\lim} \frac{\mathbb{P}\left\{ ||\textbf{X}_1 || > \eta x \right\}}{\mathbb{P}\left\{ ||\textbf{X}_1 || > x \right\}} = 0 \times \eta^{-1} = 0.
\end{align*}
Then,
\begin{align*}
	&\underset{x \rightarrow \infty}{\lim \sup} \, \mathbb{E}\left[ ||f(x^{-1}(\textbf{X}_1, \textbf{X}_2) - f(x^{-1}(\textbf{X}_1,0)|| \right] / \mathbb{P} \left\{ ||\textbf{X}_1|| > x \right\} \\ 
	&\leq \underset{x \rightarrow \infty}{\lim \sup} \, \mathbb{E}\left[ ||f(x^{-1}(\textbf{X}_1, \textbf{X}_2)) - f(x^{-1}(\textbf{X}_1,0)) || \mathds{1}_{ \{ ||\textbf{X}_2|| \leq \eta x \} } \right] / \mathbb{P}\left\{ ||\textbf{X}_1|| > x \right\} \\ &+ \underset{x \rightarrow \infty}{\lim \sup} \, \mathbb{E}\left[ ||f(x^{-1}(\textbf{X}_1, \textbf{X}_2)) - f(x^{-1}(\textbf{X}_1,0)) || \mathds{1}_{ \{ ||\textbf{X}_2|| > \eta x \} } \right] / \mathbb{P}\left\{ ||\textbf{X}_1|| > x \right\} \leq \epsilon.
\end{align*}
Since $\epsilon$ is arbitrary and $\delta_0 \otimes \Lambda_{\textbf{X}_2}(f) = 0$, this proves that
\begin{align*}
	\underset{x \rightarrow \infty}{\lim} \, \Lambda_{(\textbf{X}_1, \textbf{X}_2)}(f) &= \underset{x \rightarrow \infty}{\lim} \, \mathbb{E}\left[ f(x^{-1}(\textbf{X}_1,0)) \right] / \mathbb{P}\left\{ ||\textbf{X}_1|| > x \right\} \\
	&= \Lambda_{\textbf{X}_1} \otimes \delta_0 (f) = \Lambda_{(\textbf{X}_1, \textbf{X}_2)}(f).
\end{align*}
If now $f(\textbf{x}_2, \textbf{x}_2) = g(\textbf{x}_2)$ with $g$ a continuous function with support separated from zero, then $\Lambda_{\textbf{X}_1} \otimes \delta_0 = 0$ and 
\begin{align*}
	\underset{x \rightarrow \infty}{\lim} \, \Lambda_x(f) &= \underset{x \rightarrow \infty}{\lim} \mathbb{E}\left[ f(x^{-1}(\textbf{X}_2,\textbf{X}_2)) \right] / \mathbb{P}\left\{ ||\textbf{X}_1|| > x \right\} = \underset{x \rightarrow \infty}{\lim}\, \mathbb{E}\left[ g(x^{-1}\textbf{X}_2) \right] / \mathbb{P}\left\{ ||\textbf{X}_1|| > x \right\} \\
	&= \underset{x \rightarrow \infty}{\lim} \, \frac{\mathbb{P}\left\{ ||\textbf{X}_2|| > x \right\}}{\mathbb{P}\left\{ ||\textbf{X}_1|| > x \right\}} \underset{x \rightarrow \infty}{\lim} \, \frac{\mathbb{E}\left[ g(x^{-1} \textbf{X}_2) \right]}{\mathbb{P}\left\{ ||\textbf{X}_2|| > x \right\}} \\
	&= 0 \times \Lambda_{\textbf{X}_2}(f) = \Lambda_{(\textbf{X}_1, \textbf{X}_2)}(f).
\end{align*}
This proves that $\Lambda_{\textbf{X}} = \Lambda_{\textbf{X}_1} \otimes \delta_0$ is the only possible limit for the sequence $\{ \Lambda_x \}$ along any subsequence. We must now prove that $\{ \Lambda\}_x$ is relatively compact. Define $U_n = \left\{ (\textbf{x}_1, \textbf{x}_2)  \, : \, ||\textbf{x}_1|| + ||\textbf{x}_2|| > e^n \right\}$. The sets $U_n$, $n \in \mathbb{Z}$ satisfy the assumptions of \cite[Lemma B.1.29]{kulik2020heavy}. This proves that the sequence $\{ \Lambda_x\}$ is relatively compact and we conclude that $\Lambda_x \overset{v^{\#}}{\rightarrow} \Lambda_{(\textbf{X}_1, \textbf{X}_2)}$. We directly obtain that for any $A \in \mathcal{B}(\mathbb{R}^d)$ separated from $0$
\begin{equation*}
	\Lambda_{\textbf{X}_1+\textbf{X}_2} (A) = \Lambda_{\textbf{X}_1}(A),
\end{equation*}
hence the result.
	\end{proof}
    We recall the following Bernstein inequality from \cite{merlevede2009bernstein} and Lemma S.2 in \cite{cordoni2023consistent} where we recall the proof of the second for consistency purposes.
    \begin{lemma}
        \label{lem:merlevede}
        Let $(Y_t)_{t \geq 1}$, be a sequence of mean zero, stationary random variables whose absolute values is uniformly bounded by $\bar{y} < \infty$, and with exponential decaying strong mixing coefficients. Then for $n \geq 4$ and $z \geq 0$, there is a constant $c_1 > 0$, depending only on the mixing coefficient and such that
        \begin{equation*}
            \mathbb{P}\left\{ \left| \frac{1}{n} \sum_{i=1}^n Y_i \right| \geq z \right\} \leq \exp\left\{ - \frac{c_1 n z^2}{\bar{y}^2+z\bar{y} \ln n (\ln \ln n)} \right\}
        \end{equation*}
    \end{lemma}
        \begin{lemma}
            \label{lem:ineq_gc}
            Under the assumptions of Lemma \ref{lem:merlevede}, choose a $c_2 \in (0,\infty)$, and let $z:= y-n^{-c_2}$ for any $y \geq n^{-c_2}$. Then there is a constant $c_1 \geq 0$ such that
            \begin{equation*}
                \mathbb{P}\left\{ \underset{x \in \mathbb{R}}{\sup} \,\left| \hat{F}_n(x) - F(x)\right| \geq y \right\} \leq 2\exp\left\{ - \frac{c_1 n z^2}{1+z \ln n (\ln \ln n)} + c_2 \ln(n) \right\}
            \end{equation*}
    \end{lemma}
    \begin{proof}
        Using standard techniques, we replace the supremum by the maximum over a finite number of elements. We then apply Lemma \ref{lem:merlevede}.

        To do so, for fixed but arbitrary $\epsilon > 0$, we construct intervals $[x_\ell^L, x_\ell^U]$ with $\ell = 1,2,\dots,N(\epsilon)$, such that $|F(x) - F(z)| \leq \epsilon$ for $x,z \in [x_\ell^L, x_\ell^U]$. Fix an arbitrary $\epsilon > 0$ and divide interval $[0,1]$ into $N(\epsilon) \geq \epsilon^{-1}$, some positive integer to be chosen later, intervals $[t_{\ell -1}, t_\ell[$ where $0 = t_0 < t_1 < \dots < t_{N(\epsilon)} = 1$ such that $t_{\ell} - t_{\ell -1} \leq \epsilon$. Now, take $N(\epsilon)$ the smallest integer greater than or equal to $\epsilon^{-1}$. Define variables $-\infty \leq x_1^L \leq x_2^L\leq \dots \leq x_{N(\epsilon)}^L = \infty$ as $x_\ell^L := \inf \, \{ x \in \mathbb{R}, F(x) \geq t_{\ell -1} \}$. Similarly define variables $-\infty \leq x_1^U \leq x_2^U \leq \dots \leq x_{N_{(\epsilon)}}^U = \infty$ as $x_\ell^U := \sup \, \{ x \in \mathbb{R} : F(x) \leq t_\ell \}$. This construction has the aforementioned properties. Note that we can have $[x_\ell^L, x_\ell^U]$ equal to a singleton, i.e., $x_\ell^L = x_\ell^U$ if there are discontinuities in $F$ and such that discontinuities are larger than $\epsilon$.

        From the fact that $F(x)$ and $\hat{F}_n(x)$ are monotonically increasing, we have that $F(x_\ell^L) \leq F(x) \leq F(x_\ell^U)$ and $\hat{F}_n(x_\ell^L) \leq \hat{F}_n(x) \leq \hat{F}_n(x_\ell^U)$ for $x \in [x_\ell^L, x_\ell^U]$. Also recall that $\mathbb{E} \hat{F}_n(x) = F(x)$. In consequence
        \begin{align*}
            &\hat{F}_n(x) - F(x) \leq \hat{F}_n(x^{U}_\ell) - F(x^{L}_\ell) \leq \hat{F}_n(x^{U}_\ell) - F(x^{U}_\ell) + \epsilon, \\
            &\hat{F}_n(x) - F(x) \geq \hat{F}_n(x^{L}_\ell) - F(x^{U}_\ell) \geq \hat{F}_n(x^{L}_\ell) - F(x^{L}_\ell) - \epsilon.
        \end{align*}
        using monotonicity and the fact $|F(x_\ell^U) - F(x_\ell^L)| \leq \epsilon$ by construction. And thus for any $x$
        \begin{equation*}
            \hat{F}_n(x^{L}_\ell) - F(x^{L}_\ell) - \epsilon \leq \hat{F}_n(x) - F(x) \leq F_n(x^{U}_\ell) - F(x^{U}_\ell) + \epsilon.
        \end{equation*}
        Hence, for any $x$
        \begin{equation*}
            |\hat{F}_n(x) - F(x)| \leq \underset{\ell \in \{1,\dots,N(\epsilon)-1\} }{\max} \, \max \left\{ |\hat{F}_n(x^{U}_\ell) - F(x^{U}_\ell)|, |\hat{F}_n(x^{L}_\ell) - F(x^{L}_\ell)| \right\} + \epsilon
        \end{equation*}
        Since it holds for any $x$, we obtain
        \begin{equation*}
            \underset{x \in [0,1]}{\sup} \,|\hat{F}_n(x) - F(x)| \leq \underset{\ell \in \{1,\dots,N(\epsilon)-1\} }{\max} \, \max \left\{ |\hat{F}_n(x^{U}_\ell) - F(x^{U}_\ell)|, |\hat{F}_n(x^{L}_\ell) - F(x^{L}_\ell)| \right\} + \epsilon.
        \end{equation*}
        Set $\epsilon = n^{-c_2}$. Using union bound and apply Lemma \ref{lem:merlevede} twice with $Y_i = (1-\mathbb{E}) \mathds{1}_{\{ X_i \leq x\}}$ for arbitrary, but fixed $x$, and $z = y - \epsilon = y - n^{-c_2}$.
    \end{proof}

\end{document}